\documentclass[final,authoryear,3p,times]{elsarticle}
\usepackage{amssymb}
\usepackage{amsthm}
\usepackage{amsmath}

\usepackage{graphicx}
\usepackage{subfigure}
\usepackage{epstopdf}

\usepackage{stfloats}
\usepackage{bbm}

 \usepackage{lineno}

\newtheorem{theorem}{Theorem}[section]
\newtheorem{lemma}[theorem]{Lemma}

\newtheorem{proposition}[theorem]{Proposition}

\newtheorem{property}[theorem]{Property}

\newtheorem{corollary}[theorem]{Corollary}

\newtheorem{conjecture}[theorem]{Conjecture}

\journal{}

\begin{document}

\begin{frontmatter}



\title{Point-Normal Subdivision Curves and Surfaces}


\author{Xunnian Yang\corref{cor1}}
\date{}
%\ead[cor1]{Tel.: +86 571 87951609; fax: +86 571 87951428.}
\ead{yxn@zju.edu.cn}

\address{School of mathematical sciences, Zhejiang University, Hangzhou 310058, China}

\begin{abstract}
 This paper proposes to generalize linear subdivision schemes to nonlinear subdivision schemes for curve and surface modeling by refining vertex positions together with refinement of unit control normals at the vertices. For each round of subdivision, new control normals are obtained by projections of linearly subdivided normals onto unit circle or sphere while new vertex positions are obtained by updating linearly subdivided vertices along the directions of the newly subdivided normals. Particularly, the new position of each linearly subdivided vertex is computed by weighted averages of end points of circular or helical arcs that interpolate the positions and normals at the old vertices at one ends and the newly subdivided normal at the other ends.
 The main features of the proposed subdivision schemes are three folds:
  (1) The point-normal (PN) subdivision schemes can reproduce circles, circular cylinders and spheres using control points and control normals;
  (2) PN subdivision schemes generalized from convergent linear subdivision schemes converge and can have the same smoothness orders as the linear schemes;
  (3) PN $C^2$ subdivision schemes generalizing linear subdivision schemes that generate $C^2$ subdivision surfaces with flat extraordinary points can generate visually $C^2$ subdivision surfaces with non-flat extraordinary points.
  Experimental examples have been given to show the effectiveness of the proposed techniques for curve and surface modeling.
\end{abstract}


\end{frontmatter}

%\linenumbers


%%%%%%%%%%%%%%%%%%%%%%%%%%%%%%%%%%%%%%%%%%%%%%%%%%%%%%%%%%%%%%%%%%%%%%%%%%%%%%%%%%%%%%%%%%%
%%% Sectioin 1                                                                          %%%
%%%%%%%%%%%%%%%%%%%%%%%%%%%%%%%%%%%%%%%%%%%%%%%%%%%%%%%%%%%%%%%%%%%%%%%%%%%%%%%%%%%%%%%%%%%

\section{Introduction}
\label{Sec:intro}

Subdivision curves and surfaces are recursively generated free-form curves and surfaces from coarse polygons or rough initial meshes with arbitrary topology. Due to their flexility for shape representation and easiness to implement, subdivision curves and surfaces have become powerful tools for geometric modeling and computer graphics \citep{DeRoseKT98:subdivisionCharacter,DynLevin2002ActaNumerica}. This paper proposes a class of nonlinear subdivision schemes that generalize linear subdivision schemes for curve and surface modeling.

\subsection{Related work}

A large number of subdivision schemes used for geometric modeling are linear schemes. The subdivision algorithms presented by \citet{Chaikin74:Subdivision}, \citet{CatmullClarkSubdivision}, \citet{DooSabinSubdivision}, \citet{LoopSubdivision}, or \citet{LaneRiesenfeld80:DisplayPiecewiseSurface}, etc. are subdivision schemes generalizing uniform B-spline curves or surfaces. The schemes presented in \citep{Sederberg1998:NonuniformRecursiveSubdivision,Cashman2009:NURBSHighDegree} are the generalizations of non-uniform B-spline curves and surfaces.
The interpolatory subdivision schemes such as the 4-point scheme \citep{DynLG87FourPointSubdivision}, the butterfly scheme \citep{DynLG90butterflyScheme,ZorinSS96:InterpolationSubdivision} and the Kobbelt scheme \citep{Kobbelt96:InterpolatorySubdivision}, etc. can generate smooth curves and surfaces no longer consisting of piecewise polynomials. The linear non-stationary subdivision schemes have level dependent masks and they can be used to generate curves and surfaces defined in mixed spaces composed of polynomials and transcendental functions \citep{Fang2014GeneralizedSurfSubd,ContiDyn2021NonstationarySubdivisionPerspective}. Particularly, conics and rotational surfaces defined by trigonometric functions can be modelled by non-stationary subdivision schemes from regular control polygons or control meshes \citep{MorinWareenW01:subdivisionforsurfaceofrevolution}.

Nonlinear subdivision schemes include manifold valued subdivision and geometric subdivision. Linear subdivision schemes can be adapted to manifold valued subdivision by using geodesic averaging rules on manifolds, exponential map or by projection of linearly subdivided points onto surfaces \citep{WallnerDyn05ConvergenceByProximity}.
If the input data are scalars, the original data can be subdivided by nonlinear averaging \citep{SchaeferViugaGoldmn08:NonlinearAveraging}.
Newly subdivided vertices by geometric schemes for curve or surface modeling are computed by estimation of local geometric quantities like turning angles \citep{DynHormann12:GeometricCondTangentContinuity}, tangent lines or tangent planes \citep{Yang2005GeomSubd, Yang06:NormalBasedSubd}, osculating circles \citep{SabinDodgson2005CirclePreservingSubd,ChalmovianskyJuettler07:nonlinearCirclePreservingSubd}, or fitting Clothoids \citep{ReifWeinmann2020:ClothoidFitting}, etc. The geometric schemes for curve modeling can preserve circles or Clothoids and can generate tangent continuous curves as well but the geometric subdivision schemes for surface modeling have not been able to consistently outperform linear schemes \citep{Cashman12BeyondCatmullClark}. By replacing the linear averaging steps of recursive subdivision schemes with circle averaging, visually smooth subdivision curves and surfaces can be generated \citep{CashmanHormannReif13GenralizedLRalgorithm, LipovetskyD16weightedPointNormalPairs, LipovetskyDyn20PointNormalPairs}. Though the recursive circle averaging schemes are promising for fair curve and surface modeling, the convergence and smoothness analysis of the schemes are not available.

Popular subdivision schemes such as Catmull-Clark subdivision and Loop subdivision have only $C^1$ continuity at the extraordinary points. This is not enough for fair shape design.
\citet{Prautzsch1998G2subdivisionsurface} first proposed to improve the smoothness orders of Catmull-Clark subdivision at extraordinary points by tuning the eigenvalues of subdivision matrices. The modified Catmull-Clark subdivision scheme generates $C^2$ subdivision surfaces with forced zero curvature at the extraordinary points. \citet{Levin2006C2subdivision} proposed to update Catmull-Clark subdivision surfaces by blending with lower order polynomial patches near the extraordinary points. Similarly, \citet{Zorin2006C2ByBlending} proposed to blend Loop subdivision surfaces with parametric patches to achieve $C^2$ continuity at the extraordinary points. Using order 1 jet data, jet subdivision with the same subdivision stencils as the Loop scheme can achieve flexible $C^2$ continuity at extraordinary vertices of valence 3 \citep{XueYuDuchamp2006Jetsubdivision}. When a control mesh owns polar configuration, polar subdivision can be employed to generate $C^2$ subdivision surfaces \citep{MylesPeters2009Bi3C2PolarSubdivision}. Even though these pioneering algorithms work well under some situations, searching for a $C^2$ surface subdivision algorithm that is easy to implement and capable of generating perfect shape is still the ``holy grail" for geometric modeling \citep{ReifSabin19OldProblem}.

\subsection{Our approach}

We propose point-normal (PN) subdivision schemes for curve and surface modeling by generalizing traditional linear subdivision schemes.
In addition to control points within initial control polygons or control meshes, we assume unit control normals are also given at all or partial control points. Unlike previous approaches that use control normals to compute initial matrix weights for matrix weighted rational subdivision \citep{Yang16MatrixRational} or compute refined points and normals from circles each fits two old point-normal pairs \citep{LipovetskyDyn20PointNormalPairs}, we compute refined control normals by projecting the linearly subdivided normals onto unit circle or sphere and update the linearly subdivided vertices along the newly subdivided normals by weighted averages of end points of circular or helical arcs that interpolate the subdivided normals at one ends as well as the old points and normals at the other ends.

The PN subdivision schemes can reproduce circles, circular cylinders or spheres when the initial control points and control normals are sampled from those geometric primitives, even with uneven sampling. This type of nonlinear subdivision can reduce to traditional linear subdivision when the control normals vanish or equal the same vector. We prove that the convergence and smoothness orders of univariate PN subdivision schemes as well as the convergence and $C^1$ smoothness of bivariate PN subdivision schemes at the extraordinary points are the same as the corresponding linear subdivision schemes. Therefore, the proposed nonlinear subdivision can guarantee high orders of smoothness when the linear subdivision scheme does. We have also generalized the modified Catmull-Clark subdivision scheme that generates $C^2$ subdivision surfaces with flat extraordinary points to PN modified Catmull-Clark subdivision scheme. It is observed that the PN $C^2$ subdivision surfaces are curvature continuous too, but the curvatures at the extraordinary points can be no longer vanishing.

Briefly, the main contributions of the paper are as follows:
\begin{itemize}
\item We propose a class of nonlinear subdivision schemes by generalizing linear subdivision schemes. The new subdivision schemes permit shape control using control points and control normals and they can reproduce classical geometric primitives like circles, circular cylinders and spheres.
\item The proposed nonlinear subdivision schemes have solid theoretical foundations. It is proved that the convergence and high orders of smoothness of univariate PN subdivision schemes as well as the $C^1$ smoothness of bivariate PN subdivision at the extraordinary points are the same as the linear schemes.
\item PN subdivision schemes can be simple solutions to modeling fair $C^2$ subdivision surfaces. Particularly, the PN subdivision schemes generalizing linear schemes that generate $C^2$ subdivision surfaces with flat extraordinary points can generate visually $C^2$ subdivision surfaces with non-flat extraordinary points.
\end{itemize}

\subsection{Outline}
The paper is organized as follows. In Section~\ref{Sec:Basics} we review some basic results of binary subdivision and we present new subdivision schemes in Section~\ref{Sec:PN-subdscheme}. Section~\ref{Sec:Converngence} is devoted to the theoretical analysis of convergence and smoothness of the proposed subdivision schemes. We further construct PN $C^2$ subdivision surfaces in Section~\ref{Sec:PN-C2-subdivision surface}. We present several modeling examples by the proposed schemes in Section~\ref{Sec:Examples} as well as some discussions in Section~\ref{Sec:Discussions}. Section~\ref{Sec:conclusions} concludes the paper with a brief summary.

%%%%%%%%%%%%%%%%%%%%%%%%%%%%%%%%%%%%%%%%%%%%%%
%%%% Section 2
%%%%%%%%%%%%%%%%%%%%%%%%%%%%%%%%%%%%%%%%%%%%%%

%\section{Basics of binary subdivision}
\section{Preliminaries and notations}
\label{Sec:Basics}
This section review some basic results about binary linear or nonlinear subdivision which serve as preliminaries of our proposed new subdivision schemes. Notations are introduced simultaneously.

\subsection{Univariate binary subdivision}
Assume $\{\mathbf{p}_i^0: i\in \mathbb{Z}\}$ are a sequence of points in 2D or 3D space, the binary subdivision of the polygon defined by the given points with mask $\mathbf{a}=\{a_i:i\in \mathbb{Z}\}$ is as follows
\begin{equation}
\label{Eqn:stationary binary subd}
  \mathbf{p}_i^{k+1}=\sum_{j\in \mathbb{Z}}a_{i-2j}\mathbf{p}_j^k, \ \ \ \  i\in \mathbb{Z}.
\end{equation}
Assume
\[
P^k=(\cdots;\mathbf{p}_{j-1}^k; \mathbf{p}_j^k;\mathbf{p}_{j+1}^k;\cdots)
\]
be a column of points $\mathbf{p}_j^k$, $j\in\mathbb{Z}$.
Note that the symbol $(\mathbf{p}_a,\mathbf{p}_b,\mathbf{p}_c)$ represents a row of elements $\mathbf{p}_a$, $\mathbf{p}_b$ and $\mathbf{p}_c$. We use $(\mathbf{p}_a,\mathbf{p}_b,\mathbf{p}_c)^{\top_{blk}}:=(\mathbf{p}_a;\mathbf{p}_b;\mathbf{p}_c)$ to denote the transpose of a matrix in terms of block elements in the following text. The conventional transpose of a vector or matrix $\mathbf{V}$ is represented by $\mathbf{V}^\top$.

Let
\[
S_a=\left(
      \begin{array}{ccccc}
        \cdots & \cdots & \cdots & \cdots & \cdots \\
        \cdots & a_{i-1-2(j-1)} & a_{i-1-2j} & a_{i-1-2(j+1)} & \cdots \\
        \cdots & a_{i-2(j-1)} & a_{i-2j} & a_{i-2(j+1)} & \cdots \\
        \cdots & a_{i+1-2(j-1)} & a_{i+1-2j} & a_{i+1-2(j+1)} & \cdots \\
        \cdots & \cdots & \cdots & \cdots & \cdots \\
      \end{array}
    \right)
\]
be a bi-infinite matrix.
Then Equation (\ref{Eqn:stationary binary subd}) can be reformulated in matrix form as
\begin{equation}
\label{Eqn:Pk+1=SaPk}
P^{k+1}=S_a P^k.
\end{equation}
Particularly, $\mathbf{p}_i^{k+1}=(S_a)_iP^k$, where $(S_a)_i$ represents the $i$th row of the matrix $S_a$.
It is always assumed that the mask $\mathbf{a}$ has a limited support for subdivision curve or surface modeling. This just implies that each row of matrix $S_a$ has a limited number of non-zero elements.

The symbol of subdivision scheme $S_a$ with mask $\mathbf{a}$ is given by
$a(z)=\sum_{i\in \mathbb{Z}}a_iz^i$.
A necessary condition for the convergence of the subdivision scheme $S_a$ is that the mask should satisfy
$\sum_j a_{2j}=\sum_j a_{2j+1}=1$.
See Theorem 1 in \citep{Dyn02analysisbyLaurentPolynomials}. Since $a(1)=2$ and $a(-1)=0$, the symbol can be factorized into
\[
a(z)=(1+z)q(z).
\]
Let $\Delta P^k=\{\Delta\mathbf{p}_i^k=\mathbf{p}_i^k-\mathbf{p}_{i-1}^k: i\in \mathbb{Z}\}$. From Theorem 2 in \citep{Dyn02analysisbyLaurentPolynomials} we know that
\[
\Delta P^{k+1} = \Delta (S_a P^k) = S_q \Delta P^k.
\]
Let $\Delta (S_a)_i = (S_a)_i - (S_a)_{i-1}$.
The elements within $\Delta P^{k+1}$ are computed by
\[
\Delta \mathbf{p}_i^{k+1} = \Delta (S_a)_i P^k = (S_q)_i \Delta P^k, \ \ \ i\in \mathbb{Z}.
\]

Assume $P^k(t)$ be a piecewise linear curve that interpolates points $\mathbf{p}_i^k$ at knots $2^{-k}i$ for $i\in \mathbb{Z}$. If the sequence of curves $\{P^k(t), k\in \mathbb{Z}_+\}$ converge uniformly to a limit curve $P(t)$ as $k$ approaches infinity, the curve $P(t)$ is continuous. Then, the subdivision scheme $S_a$ is convergent and denoted as $S_a\in C^0$.
On the other hand, if $\Delta P^{k+1}=S_q \Delta P^k$ tends to zero as $k$ approaches infinity, it means that the scheme $S_q$ is contractive.
It is shown (Theorem 3 in \citep{Dyn02analysisbyLaurentPolynomials}) that the subdivision scheme $S_a$ converges if and only if the scheme $S_q$ is contractive. For algorithm details on checking whether or not the scheme $S_q$ is contractive, we refer the readers to \citep{Dyn02analysisbyLaurentPolynomials}.

Besides convergence, higher orders of smoothness of a subdivision curve can also be checked by using the symbol $a(z)$ of the scheme. If $a(z)=\frac{(1+z)^m}{2^m}b(z)$, the $m$th order differences of $P^k$ can be computed by
\begin{equation}
\label{Eqn:Delta m Pk}
\frac{\Delta^m P^k}{(2^{-k})^m} = \frac{\Delta^m S_a P^{k-1}}{(2^{-k})^m} = S_b \frac{\Delta^m P^{k-1}}{(2^{-(k-1)})^m},
\end{equation}
where $\Delta^m = \Delta(\Delta^{m-1})$ is defined recursively.
From the representation $\Delta^m S_a = (\cdots; \Delta^m (S_a)_{i-1}; \Delta^m (S_a)_i; \Delta^m (S_a)_{i+1}; \cdots)$,
we know that the two operators used to compute the differences of subdivided vertices from old ones by Equation (\ref{Eqn:Delta m Pk}) satisfy
\begin{equation}
\label{Eqn:Delta m Sa_i}
\Delta^m (S_a)_i = \frac{(S_b)_i}{2^m}\Delta^m.
\end{equation}
From Theorem 4 in \citep{Dyn02analysisbyLaurentPolynomials} we know that the subdivision scheme $S_a\in C^m$ when $S_b$ is convergent. Particularly, the $m$th order derivative of the limit curve $P(t)$ at each dyadic point is obtained as
\[
\lim_{k\rightarrow \infty \atop k>l} \left(\frac{\Delta^m P^k}{(2^{-k})^m}\right)_{i2^{k-l}} = P^{(m)}(i2^{-l}).
\]
It is also known that a $C^m$ continuous subdivision curve has H\"{o}lder regularity of $C^{m+\alpha}$, where $0<\alpha\leq1$. How to compute the H\"{o}lder regularity has been discussed in \citep{Rioul1992RegularityCriteria,DynLevin2002ActaNumerica,HormannSabin08SubdSchemeswithCubicPrecision}. If the subdivision scheme $S_a\in C^m$, the differences of the subdivided points satisfy
\begin{equation}
\label{Eqn:holder regularity}
\left\|\frac{\Delta^m \mathbf{p}_i^k}{2^{-km}} - \frac{\Delta^m \mathbf{p}_{i-1}^k}{2^{-km}}\right\|<c_0 2^{-k\alpha}
\end{equation}
where $\|\cdot \|$ represents the Euclidean norm of a vector and $c_0$ is a constant.
We denote the norm of a point sequence or a difference sequence within this paper as follows
\[
\left\|\frac{\Delta^m P^k}{(2^{-k})^m}\right\|_\infty=\sup_{i\in\mathbb{Z}}\left\{\left\|\frac{\Delta^m \mathbf{p}_i^k}{2^{-km}}\right\|\right\}.
\]

In contrast to stationary subdivision that has a fixed mask during the whole subdivision process, the mask can also be level dependent or even position dependent when a non-stationary or non-uniform subdivision curve is generated.
Assume that $\mathbf{a}_k=\{a_i^k:i\in \mathbb{Z}\}_{k\in \mathbb{Z}_+}$, the points refined by the non-stationary subdivision scheme is obtained as
\[
P^{k+1}=S_{a_k} P^k.
\]
\citet{DynLevin1995AsymptoticallyEquivalent} first proposed the \emph{asymptotically equivalent} theory for analyzing the convergence and smoothness of non-stationary subdivision schemes by comparing with the stationary ones.
The subdivision scheme $S_{\{a_k\}}$ is asymptotically equivalent with $S_a$, if
\[
\sum_{k\in \mathbb{Z}_+} \|S_{a_k}-S_a\|_\infty < +\infty,
\]
where $\|S_{a_k}-S_a\|_\infty = \max_{i\in\{0,1\}} \sum_{j\in \mathbb{Z}}|a_{i-2j}^k - a_{i-2j}|$.
If $S_{\{a_k\}}$ is asymptotically equivalent with $S_a$, it is denoted as $S_{\{a_k\}}\approx {S_a}$.
Furthermore, the subdivision scheme $S_{\{a_k\}}$ is termed \emph{stable} if there exists a constant $K_a>0$ such that for all $k,n\in\mathbb{Z}_+$,
\[
\|S_{a_{k+n}}\cdots S_{a_{k+1}}S_{a_k}\|_\infty < K_a.
\]

\begin{proposition}
\label{Proposition:asympototically equivalent}
(Theorem 7b in \citep{DynLevin1995AsymptoticallyEquivalent}) If $S_{\{a_k\}}\approx {S_a}$, where $S_a$ is a $C^0$ stationary binary subdivision scheme with a finitely supported mask, then $S_{\{a_k\}}$ is $C^0$ and stable.
\end{proposition}

Though asymptotical equivalence is useful for convergence analysis of non-stationary or even nonlinear subdivision, but it is too restrictive for smoothness analysis of general non-stationary or nonlinear subdivision. Instead, the following proposition serves as a basic tool for such purposes.
\begin{proposition}
\label{Proposition:perturb subdivision}
(Proposition 3.1 in \citep{DynLevinYoon2014NewMethodForAnalysis}) Let $S_{\{a_k\}}$ be a linear and stable ($C^0$) subdivision scheme. Let $\{\varepsilon^k\}_{k\in \mathbb{Z}_+}$ be a sequence of sequences, $\varepsilon^k=\{\varepsilon_j^k\}_{j\in \mathbb{Z}}$, satisfying
\[
\sum_{k=1}^\infty \|\varepsilon^k\|_\infty < +\infty.
\]
Then, the perturbed subdivision scheme
\[
f^k=S_{a_k} f^{k-1} + \varepsilon^k, \ \ \ \ k=1,2,\ldots
\]
converges to a $C^0$ limit for any initial data $f^0\in l^\infty(\mathbb{Z})$.
\end{proposition}

Propositions \ref{Proposition:asympototically equivalent} and \ref{Proposition:perturb subdivision} together with Equation (\ref{Eqn:Delta m Pk}) and Equation (\ref{Eqn:Delta m Sa_i}) will be used as basic tools for the convergence and smoothness analysis of univariate PN subdivision schemes in Section \ref{Subsection:analysis of univariate PN subdivision}.

\subsection{Subdivision surfaces with extraordinary vertices}
\label{Subsection:subdivision of irregular meshes}
Bivariate subdivision schemes defined on regular quad meshes or regular triangulations can have symbol $a(z)=a(z_1,z_2)$. In particular, if $a(z_1,z_2)$ have factors like $(1+z_1)^m$ or $(1+z_2)^m$, etc., the convergence and smoothness of bivariate subdivision on regular meshes can be analyzed using the same technique as that for univariate subdivision. See references \citep{CavarettaDM1991StationarySubdivision,DynLevin2002ActaNumerica} for more details on this topic. On the other hand, convergence and smoothness of subdivision surfaces at extraordinary vertices have to be analyzed in a different way.

Vertices of valence not equal to 4 within a quad mesh and vertices of valence not equal to 6 within a triangular mesh are extraordinary vertices. While stationary subdivision surfaces with regular control meshes are actually parametric surfaces defined by control points and refinable basis functions, the subdivision surface near an extraordinary vertex is just composed of a sequence of surface rings \citep{Reif95:UnifiedApproachSubdivision}. We take similar notations as used in \citep{PetersReif08SubdivisionBook}.
Assume $Q=(\mathbf{q}_0;\ldots;\mathbf{q}_{\bar{l}})$ be a set of control points surrounding an isolated extraordinary vertex of valence $n$.
Let
\[
\mathbf{\Sigma}^0:=[0,1]^2\backslash[0,1/2)^2, \ \ \mathbf{\Sigma}^m:=2^{-m}\mathbf{\Sigma}^0, \ \ \mathbf{S}_n^m:=\mathbf{\Sigma}^m\times \mathbb{Z}_n, \ \ m\in \mathbb{N}_0,
\]
with $\mathbb{Z}_n$ the integers modulo $n$. Then the surface ring $\mathbf{x}_m$ is a parametric surface defined on domain $\mathbf{S}_n^m$ and the whole domain for the subdivision surface near the extraordinary vertex is
\[
\mathbf{S}_n=\bigcup_{m\in\mathbb{N}_0}\mathbf{S}_n^m \cup \{\mathbf{0}\}.
\]

Let $G:=(g_0,\ldots,g_{\bar{l}})$, where $g_l\in C^k(\mathbf{S}_n^0,\mathbb{R}), \ l=0,\ldots,\bar{l}$, are a set of scalar valued generating functions (see Definition 4.9 in \citep{PetersReif08SubdivisionBook}) that form a partition of unity,
$
\sum_{l=0}^{\bar{l}}g_l(\mathbf{s})=1, \ \mathbf{s}\in \mathbf{S}_n^0.
$
The surface ring $\mathbf{x}_0(\mathbf{s})$ is then represented as
\[
\mathbf{x}_0(\mathbf{s})=\sum_{l=0}^{\bar{l}}g_l(\mathbf{s})\mathbf{q}_l=G(\mathbf{s})Q.
\]
Let $S=(s_{ij})_{0\leq i,j\leq \bar{l}}$ be a subdivision matrix with all rows summing up to 1. The control points for the $m$th surface ring are obtained as
$Q_m=SQ_{m-1}=\ldots=S^mQ$ and the surface ring is
\begin{equation}
\label{Eqn:x^m(s)}
\mathbf{x}_m(\mathbf{s})=G(2^m\mathbf{s})Q_m=G(2^m\mathbf{s})S^mQ.
\end{equation}
When the surface rings $\{\mathbf{x}_m\}_{m\in \mathbb{N}_0}$ converge to a limit point, the subdivision surface converges at the extraordinary point. If the normal vectors of the surface rings also converge to a limit vector, the subdivision surface is normal continuous at the limit point \citep{DooSabinSubdivision,Reif95:UnifiedApproachSubdivision}.

Most popular linear subdivision algorithms for surface modeling are \emph{standard algorithms} of each the subdivision matrix $S$ has eigenvalues
\[
\lambda_0=1>\lambda_1=\lambda_2>|\lambda_3|\geq\ldots.
\]
Assume the right eigenvectors of the matrix $S$ are $v_i$, $i=0,1,\ldots,\bar{l}$, and the left ones are $w_i^\top$, $i=0,1,\ldots,\bar{l}$. Then the matrix $S$ can be decomposed as $S=VJV^{-1}$, where $V=(v_0,v_1,\ldots,v_{\bar{l}})$, $V^{-1}=(w_0^\top;w_1^\top;\ldots;w_{\bar{l}}^\top)$ and $J$ is the Jordan matrix in terms of the eigenvalues.
Since each row of the matrix $S$ sums up to one, the eigenvector corresponding to $\lambda_0=1$ is $v_0=\mathbbm{1}:=(1;1;\ldots;1)$.

Let $\lambda=\lambda_1=\lambda_2$, $F=GV=(f_0,f_1,f_2,\ldots)$ and $P=V^{-1}Q=(\mathbf{p}_0;\mathbf{p}_1;\mathbf{p}_2;\ldots)$. In particular, we have $f_0=Gv_0=1$, $f_i=Gv_i$, $i=1,2$, and $\mathbf{p}_i=w_i^\top Q$, $i=0,1,2$.
Since $\lambda<1$, by reformulating $\mathbf{x}_m$ as
\[
\mathbf{x}_m=GS^mQ=GVJ^mV^{-1}Q=FJ^mP,
\]
the surface ring can be asymptotically expanded as
\begin{equation}
\label{Eqn:x^m expanded}
\mathbf{x}_m \cong \mathbf{p}_0 + \lambda^m(f_1\mathbf{p}_1+f_2\mathbf{p}_2)
             = \mathbf{p}_0 + \lambda^m \Psi (\mathbf{p}_1;\mathbf{p}_2),
\end{equation}
where $\Psi=(f_1,f_2)$ is the characteristic ring \citep{Reif95:UnifiedApproachSubdivision}.
From Equation (\ref{Eqn:x^m expanded}), it is known that the surface rings converge to a central point as
\[
\lim_{m\rightarrow +\infty}\mathbf{x}_m=\mathbf{p}_0.
\]
By the eigen-decomposition above, one also has
\begin{equation}
\label{Eqn:S^m Q ->p_01}
\lim_{m\rightarrow +\infty}S^mQ=\lim_{m\rightarrow +\infty}VJ^mP=\mathbf{p}_0\mathbbm{1}.
\end{equation}
Let ${^\times}D\Psi=D_1f_1D_2f_2-D_2f_1D_1f_2$ be the Jacobian determinant of the characteristic ring. The characteristic ring $\Psi$ is regular when the sign of ${^\times}D\Psi$ does not change nor vanishes.

For almost all initial control nets, the control points around an extraordinary vertex may not lie on a line or degenerate to one point, it is then assumed that $\mathbf{p}_1$ and $\mathbf{p}_2$ within Equation (\ref{Eqn:x^m expanded}) are linear independent. Based on this assumption, the normal vector at the central point $\mathbf{p}_0$ will be defined and the subdivision surface can be normal continuous at the central point.
\begin{proposition}
\label{Propostion:normalcontinuityforsteadysubdivision}
(Theorem 5.6 in \citep{PetersReif08SubdivisionBook})
A standard algorithm with characteristic ring $\Psi$ is normal continuous with central normal
\[
\mathbf{n}^c=sign({^\times}D\Psi)\frac{\mathbf{p}_1\times \mathbf{p}_2}{\|\mathbf{p}_1\times \mathbf{p}_2\|},
\]
if $\Psi$ is regular.
\end{proposition}

Besides normal continuity, a subdivision surface can have even higher orders of smoothness at the extraordinary points \citep{Prautzsch1998GkSmoothness}. In particular, the $G^2$ (also $C^2$ by reparameterization) continuity at the extraordinary points can be guaranteed when the subdivision matrix $S$ satisfies the following condition.
\begin{proposition}
\label{Proposition:C2 continuity at extraordinary point}
(Theorem 2.1 in \citep{Prautzsch1998G2subdivisionsurface})\footnote{See also Theorem 1 in \citep{Prautzsch2000G2Loopsurface}}
Let 1, $\lambda$, $\lambda$, $\mu$, \ldots, $\zeta$ be all the (possibly complex) eigenvalues of $S$ where $1>|\lambda|>|\mu|\geq\ldots\geq|\zeta|$ and assume two eigenvectors $\mathbf{c}$ and $\mathbf{d}$ associated with the double real eigenvalue $\lambda$. If the first surface ring of the net given by $[\mathbf{c}_1\ldots \mathbf{c}_m]^\top=[\mathbf{c} \ \mathbf{d}]$ is regular without self-intersections and
\[
|\lambda|^k > |\mu|, \ \ \  k=1,2,
\]
then the limiting surface is a $G^k$-surface for almost all initial nets $\mathcal{M}_0$.
\end{proposition}

As will be given in Section \ref{Subsection:analysis of PN subdivision surface}, we analyze the smoothness of a PN subdivision surface at an extraordinary point by comparing with a sequence of linear subdivision surfaces. By computing the central normal vector for every linear subdivision surface, the normal vector at the central point of the PN subdivision surface will be obtained and the normal continuity of the PN subdivision surface will be proved. Proposition \ref{Proposition:C2 continuity at extraordinary point} plays key roles for constructing $C^2$ subdivision surfaces as well as PN $C^2$ subdivision surfaces with arbitrary topology control meshes in Section \ref{Sec:PN-C2-subdivision surface}.

\subsection{Binary subdivision on sphere}
\label{Subsection:subdivision on sphere}
Subdivision of points on a circle or sphere can be used to construct smooth normal fields and have been applied successfully for rendering or animation purposes \citep{AlexaB08SubdivisionShading,WallnerPottmann06:subdivisionforanimation}. In this paper we study nonlinear subdivision schemes for curve and surface modeling along with construction of smooth normal fields by subdivision.

Though linear subdivision schemes can be adapted to data on sphere in several different ways, the projection method composed of linear subdivision followed by a normalization step is one simple but efficient method.
Assume $\{\mathbf{n}_i^0: i\in \mathbb{Z}\}$ are points lying on a unit circle or sphere, the subdivided points are computed by
\[
  \mathbf{n}_i^{k+1}=\frac{\sum_{j\in \mathbb{Z}}a_{i-2j}\mathbf{n}_j^k}{\|\sum_{j\in \mathbb{Z}}a_{i-2j}\mathbf{n}_j^k\|}, \ \ \ \  i\in \mathbb{Z}.
\]
We assume here that the input points on the circle or sphere is locally dense enough such that the denominator does not vanish.

The convergence and $C^1$ continuity of manifold valued subdivision can be analyzed by proximity \citep{WallnerDyn05ConvergenceByProximity}. \citet{XieYu2007SmoothnessInterpSubdProjection} showed that the projection based univariate interpolatory subdivision on a sphere have the same smoothness orders as well as the same H\"{o}lder regularity as that for linear subdivision while \citet{Grohs2009SmoothnessApproSubdProjection} proved the smoothness equivalence between the projection based univariate approximate subdivision on sphere and the linear subdivision. Assume the linear subdivision scheme $S_a\in C^m$ and has the H\"{o}lder regularity $C^{m+\alpha}$, where $0<\alpha\leq1$. Then the subdivision curve $\mathbf{n}(t)$ on sphere has maximum $m$th order of continuous derivatives. Similar to Equation (\ref{Eqn:holder regularity}), the differences of subdivided points on sphere satisfy
\begin{equation}
\label{Eqn:holder regularity_normal}
\left\|\frac{\Delta^m \mathbf{n}_i^k}{2^{-km}}-\frac{\Delta^m \mathbf{n}_{i-1}^k}{2^{-km}}\right\| < c_1 2^{-k\alpha},
\end{equation}
where $c_1$ is a constant.

Besides the univariate subdivision on sphere, linear subdivision schemes for regular or irregular meshes can also be adapted to meshes on sphere. In particular, we have to pay much attention to convergence and smoothness of subdivision near extraordinary vertices. Assume $N=[\mathbf{n}_0;\ldots;\mathbf{n}_{\bar{l}}]$ be a set of points surrounding an isolated extraordinary vertex of valence $n$ on sphere and $S=(s_{ij})_{0\leq i,j\leq \bar{l}}$ be the subdivision matrix as in Equation (\ref{Eqn:x^m(s)}). Then the points on sphere are refined recursively as follows
\begin{equation}
\label{Eqn:refinement of normals at extra vertex}
\mathbf{n}_i^{k+1}=\frac{\sum_{j=0}^{\bar{l}}s_{ij}\mathbf{n}_j^k}{\|\sum_{j=0}^{\bar{l}}s_{ij}\mathbf{n}_j^k\|}, \ \ \ i=0,1,\ldots,\bar{l}.
\end{equation}
\citet{Weinmann2010NonlinearSubdOnIrregularmesh} has shown that the manifold valued subdivision adapted from a standard scheme on irregular meshes converges and the limit function is $C^1$ continuous in the vicinity of an extraordinary point over Reif's characteristic parametrization.
We modify Proposition 2.7 in \citep{Weinmann2010NonlinearSubdOnIrregularmesh} for distance estimation between subdivided points near an extraordinary vertex on sphere, which will be used for convergence and smoothness analysis for our newly proposed nonlinear subdivision scheme for irregular meshes.
\begin{proposition}
\label{Proposition:distancebetweenrefinednormals}
Let $S=(s_{ij})_{0\leq i,j\leq \bar{l}}$ be a standard subdivision matrix. Assume $\mathbf{n}_j^0$, $j=0,1,\ldots,\bar{l}$, are unit normals corresponding to vertices in the vicinity of an extraordinary vertex on an irregular mesh, and $\mathbf{n}_j^k$, $j=0,1,\ldots,\bar{l}$, $k\in\mathbb{N}$, are given by Equation (\ref{Eqn:refinement of normals at extra vertex}). There exist constants $c_2>0$, $0<\gamma<1$, such that
\[
\|\mathbf{n}_j^k - \mathbf{n}_l^k\|\leq c_2\gamma^k, \ \ \ \ j,l\in\{0,1,\ldots,\bar{l}\}.
\]
\end{proposition}

%%%%%%%%%%%%%%%%%%%%%%%%%%%%%%%%%%%%%%%%%%%%%%
%%%% Section 3
%%%%%%%%%%%%%%%%%%%%%%%%%%%%%%%%%%%%%%%%%%%%%%

\section{Point-normal subdivision schemes}
\label{Sec:PN-subdscheme}

We generalize linear subdivision schemes that only refine polygon or mesh vertices to point-normal subdivision schemes that refine polygon or mesh vertices along with the refinement of unit control normals at the vertices. Some basic geometric properties of PN subdivision schemes will be given.

\subsection{The PN subdivision schemes}
Assume $\mathbf{a}=\{a_i:i\in \mathbb{Z}^s\}$ is the mask for univariate ($s=1$) or bivariate ($s=2$) linear subdivision on regular meshes. Let $\{(\mathbf{p}_i^0,\mathbf{n}_i^0): i\in \mathbb{Z}^s\}$ be the initial control points and unit control normals on a polygon or a regular mesh.
The polygon or the mesh with initial control normals is subdivided as follows
\begin{equation}
\label{Eqn:PN binary subd}
\left\{
\begin{array}{ccl}
  \mathbf{q}_i^{k+1}&=&\sum_{j\in \mathbb{Z}^s}a_{i-2j}\mathbf{p}_j^k, \\
  \mathbf{n}_i^{k+1}&=&\frac{\sum_{j\in \mathbb{Z}^s}a_{i-2j}\mathbf{n}_j^k}{\|\sum_{j\in \mathbb{Z}^s}a_{i-2j}\mathbf{n}_j^k\|},
  \hskip 2.0cm i\in \mathbb{Z}^s,\\
  \mathbf{p}_i^{k+1}&=&\mathbf{q}_i^{k+1}+\sum_{j\in \mathbb{Z}^s} a_{i-2j}h_{ij}^k \mathbf{n}_i^{k+1},
\end{array}
\right.
\end{equation}
where
\[
h_{ij}^k=\frac{(\mathbf{n}_j^k+\mathbf{n}_i^{k+1})^\top(\mathbf{p}_j^k-\mathbf{q}_i^{k+1})}
            {(\mathbf{n}_j^k+\mathbf{n}_i^{k+1})^\top\mathbf{n}_i^{k+1}}.
\]
Besides uniform binary subdivision on regular meshes, any other linear subdivision schemes on regular or irregular control meshes can also be extended to PN subdivision. Replacing $a_{i-2j}$ within Equation (\ref{Eqn:PN binary subd}) with $s_{ij}$, $i,j\in\{0,1,\ldots,\bar{l}\}$, which are originally given in Equation (\ref{Eqn:x^m(s)}), we obtain PN subdivision schemes for irregular meshes surrounding extraordinary vertices or extraordinary faces.
The new subdivision schemes are referred as PN-4-point, PN-Catmull-Clark, PN-Butterfly, etc. when they are generalized from traditional linear subdivision schemes 4-point, Catmull-Clark, Butterfly, etc.

\begin{figure}[htp]
  \centering
  \includegraphics[width=0.5\linewidth]{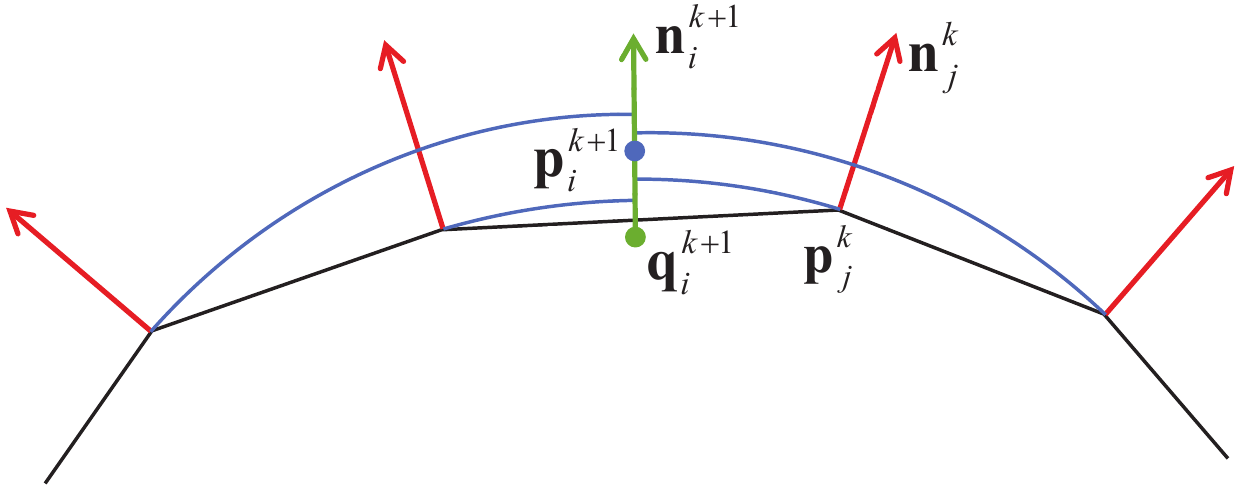}
  \caption{Geometric interpretation of the PN subdivision scheme.}
  \label{Fig:PNsubdivisionscheme}
\end{figure}

Figure~\ref{Fig:PNsubdivisionscheme} illustrates how a newly subdivided point is computed by a PN subdivision scheme. When a linearly subdivided point $\mathbf{q}_i^{k+1}$ and a unit vector $\mathbf{n}_i^{k+1}$ are computed, a line $L$ that passes through point $\mathbf{q}_i^{k+1}$ in the direction $\mathbf{n}_i^{k+1}$ is obtained. Then a height $h_{ij}^k$ from $\mathbf{q}_i^{k+1}$ along the line $L$ is derived based on the assumption that a circular or helical arc interpolates points $\mathbf{p}_j^k$ and $\mathbf{q}_i^{k+1}+h_{ij}^k\mathbf{n}_i^{k+1}$ as well as the normal vectors $\mathbf{n}_j^k$ and $\mathbf{n}_i^{k+1}$ at the two points. The interpolating curve is a circular arc when points $\mathbf{p}_j^k$, $\mathbf{q}_i^{k+1}$ and vectors $\mathbf{n}_j^k$, $\mathbf{n}_i^{k+1}$ lie on the same plane; otherwise, the interpolating curve is a helix segment on a circular cylinder that passes through point $\mathbf{p}_j^k$ and be perpendicular to normals $\mathbf{n}_j^k$ and $\mathbf{n}_i^{k+1}$ at the two ends. The weighted average of the arc end points lying on the line $L$ gives the final new point $\mathbf{p}_i^{k+1}$.
Actually, if the original control points and control normals are sampled from a smooth curve or surface without inflection point or inflection line, the newly subdivided normal may approximate the curve or surface normal very well and the mentioned circular arcs are just the approximate osculating arcs of the curve or surface at the sampled points, which guarantees that the newly subdivided point lies on or close to the original curve or surface. As explained later, this kind of nonlinear subdivision can preserve circles, circular cylinders and spheres, and they even have the same convergence and smoothness orders as the corresponding linear subdivision.

We note that selected initial control normals or a linearly subdivided normal can vanish. If a linearly subdivided normal is a zero vector, it will not be normalized and the new vertex computed by Equation (\ref{Eqn:PN binary subd}) is just the linearly subdivided vertex. Even if a newly subdivided normal $\mathbf{n}_i^{k+1}$ does not vanish, it may have opposite direction with an old control normal $\mathbf{n}_j^k$ and the updating height $h_{ij}^k$ within Equation (\ref{Eqn:PN binary subd}) will not be defined. If this is the case, one can just perturb the normal vector $\mathbf{n}_i^{k+1}$ into e.g. $2\mathbf{n}_i^{k+1}$ within the formula for computing the height $h_{ij}^k$. If the subdivided control normals are computed with no singularities in the first round of subdivision, there will be no singularities in the following subdivision. This is because the subdivided normals will become denser and denser during the subdivision and the newly subdivided normals will be very close to their old neighboring normals.

For convenience of convergence and smoothness analysis to be developed in next section, we reformulate the univariate PN subdivision scheme in matrix form.
Since the normal vectors are subdivided independent of mesh vertices, we rewrite the last expression in Equation (\ref{Eqn:PN binary subd}) as
\begin{equation}
\label{Eqn:PN binary subd A_ij^k}
\mathbf{p}_i^{k+1} = \mathbf{q}_i^{k+1}+\sum_{j\in \mathbb{Z}} a_{i-2j}A_{ij}^k (\mathbf{p}_j^k - \mathbf{q}_i^{k+1}),
\end{equation}
where
\[
A_{ij}^k=\frac{\mathbf{n}_i^{k+1}(\mathbf{n}_j^k+\mathbf{n}_i^{k+1})^\top}{(\mathbf{n}_j^k+\mathbf{n}_i^{k+1})^\top\mathbf{n}_i^{k+1}}.
\]
Note that when the subdivided normals converge, the denominator within $A_{ij}^k$ will converge to 2 as $k$ approaches infinity. Then, the matrices $A_{ij}^k$ are usually well defined for PN subdivision.
Recall that $\sum_{j\in\mathbb{Z}}a_{i-2j}=1$. Substituting the expression of $\mathbf{q}_i^{k+1}$, Equation (\ref{Eqn:PN binary subd A_ij^k}) can be further reformulated as
\begin{equation}
\label{Eqn:PN binary subd reform}
\mathbf{p}_i^{k+1} = \sum_{j\in \mathbb{Z}} a_{i-2j}M_{ij}^k \mathbf{p}_j^k, \ \ \ i\in\mathbb{Z}
\end{equation}
where $M_{ij}^k = I+\sum_{l\in\mathbb{Z}}a_{i-2l}(A_{ij}^k-A_{il}^k)$ and $I$ is the identity matrix. It is easily verified that
$\sum_{j\in\mathbb{Z}}a_{i-2j}M_{ij}^k = I$.

Let $P^k$ be as defined in Equation (\ref{Eqn:Pk+1=SaPk}) and let
\[
M^k=\left(
      \begin{array}{ccccc}
        \cdots & \cdots & \cdots & \cdots & \cdots \\
        \cdots & M_{i-1,j-1}^k & M_{i-1,j}^k & M_{i-1,j+1}^k & \cdots \\
        \cdots & M_{i,j-1}^k   & M_{i,j}^k   & M_{i,j+1}^k   & \cdots \\
        \cdots & M_{i+1,j-1}^k & M_{i+1,j}^k & M_{i+1,j+1}^k & \cdots \\
        \cdots & \cdots & \cdots & \cdots & \cdots \\
      \end{array}
    \right)
\]
be a bi-infinite matrix. Then Equation (\ref{Eqn:PN binary subd reform}) can be reformulated as
\begin{equation}
\label{Eqn:Pk+1=SaMk Pk}
P^{k+1} = (S_a\circ M^k)P^k,
\end{equation}
where $S_a\circ M^k=(a_{i-2j}M_{i,j}^k)_{i,j\in\mathbb{Z}}$ is the Hadamard product of matrices $S_a$ and $M^k$.
We use $\{S_a\circ M^k\}$ to denote the PN subdivision scheme that is generalized from a stationary subdivision scheme $S_a$.
From Equation (\ref{Eqn:Pk+1=SaMk Pk}) we have
\begin{equation}
\label{Eqn:P_i k+1 in matrix form}
    \begin{array}{ccl}
    \mathbf{p}_i^{k+1} &=& (S_a\circ M^k)_i P^k \\
    &=& ((S_a)_i\circ M_i^k) P^k \\
    &=& (S_a)_i ((M_i^k)^{\top_{blk}}\circ P^k), \ \ \ i\in\mathbb{Z}
    \end{array}
\end{equation}
where $(M)_i=M_i$ is the $i$th row of the matrix $M$. In particular,
\[
(S_a)_i=(\ldots, a_{i-2(j-1)}, a_{i-2j}, a_{i-2(j+1)}, \ldots)
\]
and
\[
(M_i^k)^{\top_{blk}}=(\ldots; M_{i,j-1}^k; M_{i,j}^k; M_{i,j+1}^k; \ldots).
\]
We note that $(S_a)_i$ also means $(\ldots, a_{i-2(j-1)}I, a_{i-2j}I, a_{i-2(j+1)}I, \ldots)$ when it is used to compute subdivided vertices.

\subsection{Basic geometric properties}
We present several basic geometric properties of the proposed PN subdivision schemes, which are useful for curve and surface modeling by employing the new subdivision technique.

\begin{property}
(Geometric invariance)
The PN subdivision curves and surfaces are translation/scaling invariant, and the shapes of the subdivision curves and surfaces are also invariant under the rotation of the coordinate system.
\end{property}
\begin{proof}
The translation/scaling invariant property is obvious based on Equation (\ref{Eqn:PN binary subd reform}), we prove that the PN subdivision is invariant under the rotation of the coordinate system. We rewrite Equation (\ref{Eqn:PN binary subd reform}) as follows
\[
\mathbf{p}_i^{k+1} = \sum_{j} a_{i-2j} \sum_{l} a_{i-2l} (I + A_{ij}^k - A_{il}^k) \mathbf{p}_j^k.
\]
We only check that $A_{ij}^k \mathbf{p}_j^k$ is invariant under the rotation of the coordinate system, $A_{il}^k \mathbf{p}_j^k$ can be checked similarly.
Assume $R$ is a rotation matrix that satisfies $R^\top R=I$ and $R^{-1}=R^\top$. We have
\[
\begin{array}{ccl}
R(A_{ij}^k \mathbf{p}_j^k) &=& \frac{R\mathbf{n}_i^{k+1}(R^\top R\mathbf{n}_j^k + R^\top R\mathbf{n}_i^{k+1})^\top\mathbf{p}_j^k}
{(\mathbf{n}_j^k + \mathbf{n}_i^{k+1})^\top R^\top R \mathbf{n}_i^{k+1}} \\
&=& \frac{R\mathbf{n}_i^{k+1}(R\mathbf{n}_j^k + R\mathbf{n}_i^{k+1})^\top R\mathbf{p}_j^k}
{(R\mathbf{n}_j^k + R\mathbf{n}_i^{k+1})^\top R \mathbf{n}_i^{k+1}}.
\end{array}
\]
Since $R\mathbf{p}_j^k$, $R\mathbf{n}_j^k$ and $R\mathbf{n}_i^{k+1}$ are points or vectors in the rotated coordinate system, the proposition is proven.
\end{proof}

\begin{property}\label{Property:normal invariance}
(Invariance to normal direction)
The PN subdivision curves and surfaces are invariant when all control normals have been inversed.
\end{property}
\begin{proof}
Equation (\ref{Eqn:PN binary subd}) holds when all vectors within the equation have been replaced with their opposite vectors. So, the property holds.
\end{proof}

Based on Property \ref{Property:normal invariance}, we will not emphasize the side of a control polygon or a control mesh in which the control normals lie when constructing a PN subdivision curve or surface.

\begin{property}
(Reduce to linear subdivision). If all initial control normals are the same vector, the PN subdivision scheme presented in Equation (\ref{Eqn:PN binary subd}) reduces to a linear subdivision scheme.
\end{property}
\begin{proof}
Assume $\mathbf{n}_{i}^0=\mathbf{n}_0$, $i\in\mathbb{Z}^s$, we have $\mathbf{n}_{i}^k=\mathbf{n}_0$ for all $i\in\mathbb{Z}^s$ and $k\in\mathbb{Z}_+$. Then Equation (\ref{Eqn:PN binary subd}) can be simplified as
\[
\begin{array}{ccl}
\mathbf{p}_i^{k+1}&=&\mathbf{q}_i^{k+1}+\sum_{j\in \mathbb{Z}^s} a_{i-2j} (\mathbf{n}_0)^\top(\mathbf{p}_j^k - \mathbf{q}_i^{k+1}) \mathbf{n}_0 \\
&=& \mathbf{q}_i^{k+1}+ (\mathbf{n}_0)^\top(\sum_{j\in \mathbb{Z}^s} a_{i-2j}\mathbf{p}_j^k - \mathbf{q}_i^{k+1}) \mathbf{n}_0 \\
&=& \sum_{j\in \mathbb{Z}^s} a_{i-2j}\mathbf{p}_j^k.
\end{array}
\]
This proves the property.
\end{proof}

Same as linear subdivision schemes, PN subdivision schemes can reproduce straight lines and planes. Moreover, PN subdivision schemes can also reproduce circles, circular cylinders and spheres.

\begin{property}
(Circle preserving). If the initial control data $(\mathbf{p}_i^0, \mathbf{n}_i^0)$, $i\in\mathbb{Z}$, are sampled from a circle, then all the newly subdivided points and normals by PN subdivision lie on the same circle.
\end{property}
\begin{proof}
Due to the geometric invariance property, we assume the initial control data are sampled from a unit circle centered at the origin. It implies that
$\mathbf{p}_i^0 =\mathbf{n}_i^0$, $i\in\mathbb{Z}$. To prove the property, we should then prove that all newly subdivided vertices lie on the unit circle as the initial data.
Assume $\mathbf{p}_i^k =\mathbf{n}_i^k$, $i\in\mathbb{Z}$, are points and normals lying on the unit circle. Let
\[
l_i^{k+1} = \left\|\sum_{j\in\mathbb{Z}}a_{i-2j}\mathbf{p}_j^k\right\| = \left\|\sum_{j\in\mathbb{Z}}a_{i-2j}\mathbf{n}_j^k\right\|.
\]
We have $\mathbf{n}_i^{k+1} = \frac{1}{l_i^{k+1}} \sum_{j\in\mathbb{Z}}a_{i-2j}\mathbf{n}_j^k$
and $\mathbf{q}_i^{k+1} = \sum_{j\in\mathbb{Z}}a_{i-2j}\mathbf{p}_j^k = l_i^{k+1}\mathbf{n}_i^{k+1}$.
Then the heights $h_{ij}^k$ are computed as
\[
\begin{array}{ccl}
h_{ij}^k &=& \frac{(\mathbf{n}_j^k + \mathbf{n}_i^{k+1})^\top(\mathbf{p}_j^k - \mathbf{q}_i^{k+1})}{(\mathbf{n}_j^k + \mathbf{n}_i^{k+1})^\top\mathbf{n}_i^{k+1}} \\
&=& \frac{(\mathbf{n}_j^k + \mathbf{n}_i^{k+1})^\top(\mathbf{n}_j^k - l_i^{k+1}\mathbf{n}_i^{k+1})}{(\mathbf{n}_j^k + \mathbf{n}_i^{k+1})^\top\mathbf{n}_i^{k+1}} \\
&=& 1-l_i^{k+1}.
\end{array}
\]
Now, the newly subdivided point is obtained as
\[
\begin{array}{ccl}
\mathbf{p}_i^{k+1} &=& \mathbf{q}_i^{k+1} +  \sum_{j\in\mathbb{Z}}a_{i-2j}h_{ij}^k \mathbf{n}_i^{k+1}\\
&=& l_i^{k+1}\mathbf{n}_i^{k+1} +  \sum_{j\in\mathbb{Z}}a_{i-2j}(1-l_i^{k+1}) \mathbf{n}_i^{k+1} \\
&=& \mathbf{n}_i^{k+1}.
\end{array}
\]
Since $\|\mathbf{p}_i^{k+1}\| = \|\mathbf{n}_i^{k+1}\| = 1$, the newly subdivided points and normals lie on the same circle as the initial control data.
\end{proof}

\begin{figure}[h]
  \centering
  \subfigure[]{\includegraphics[width=0.28\linewidth]{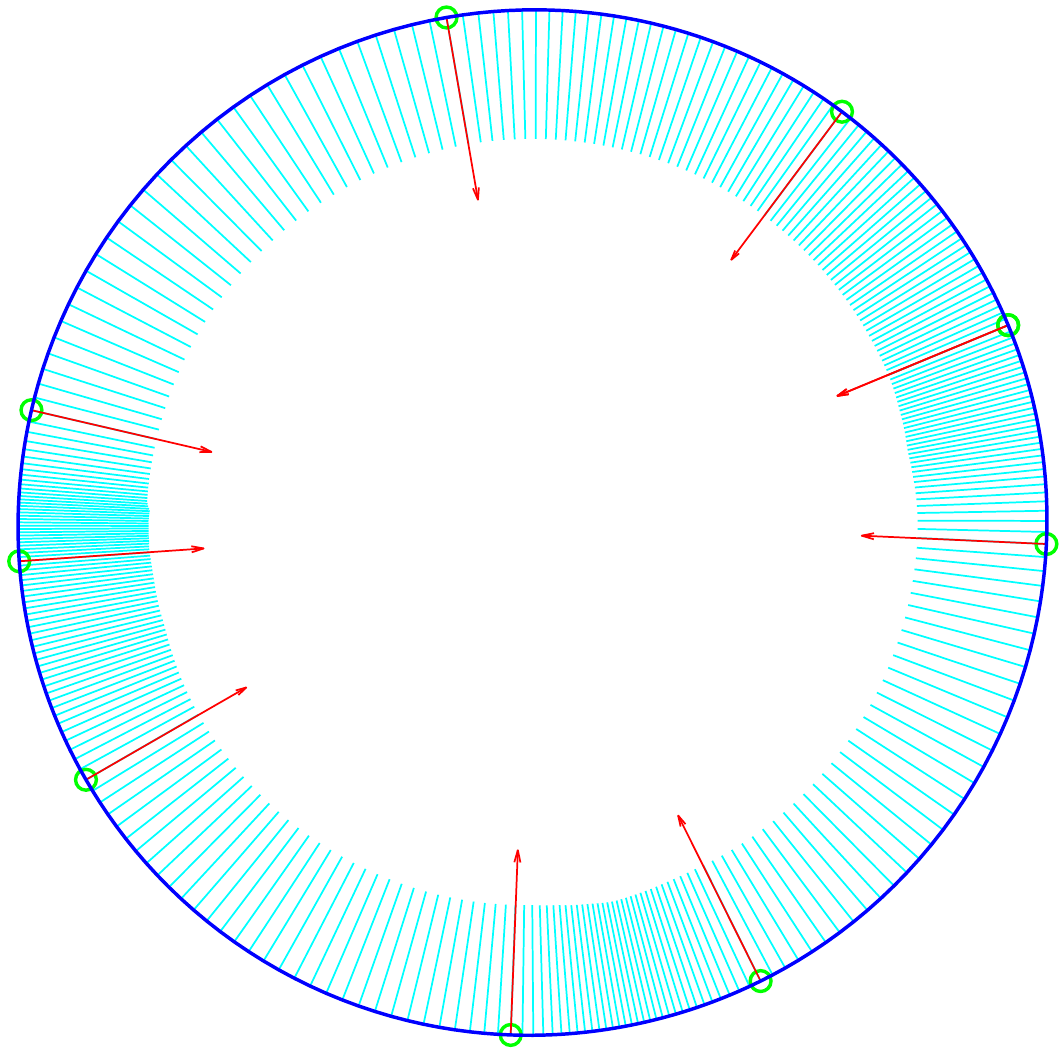}} \ \ \ \ \ \
  \subfigure[]{\includegraphics[width=0.28\linewidth]{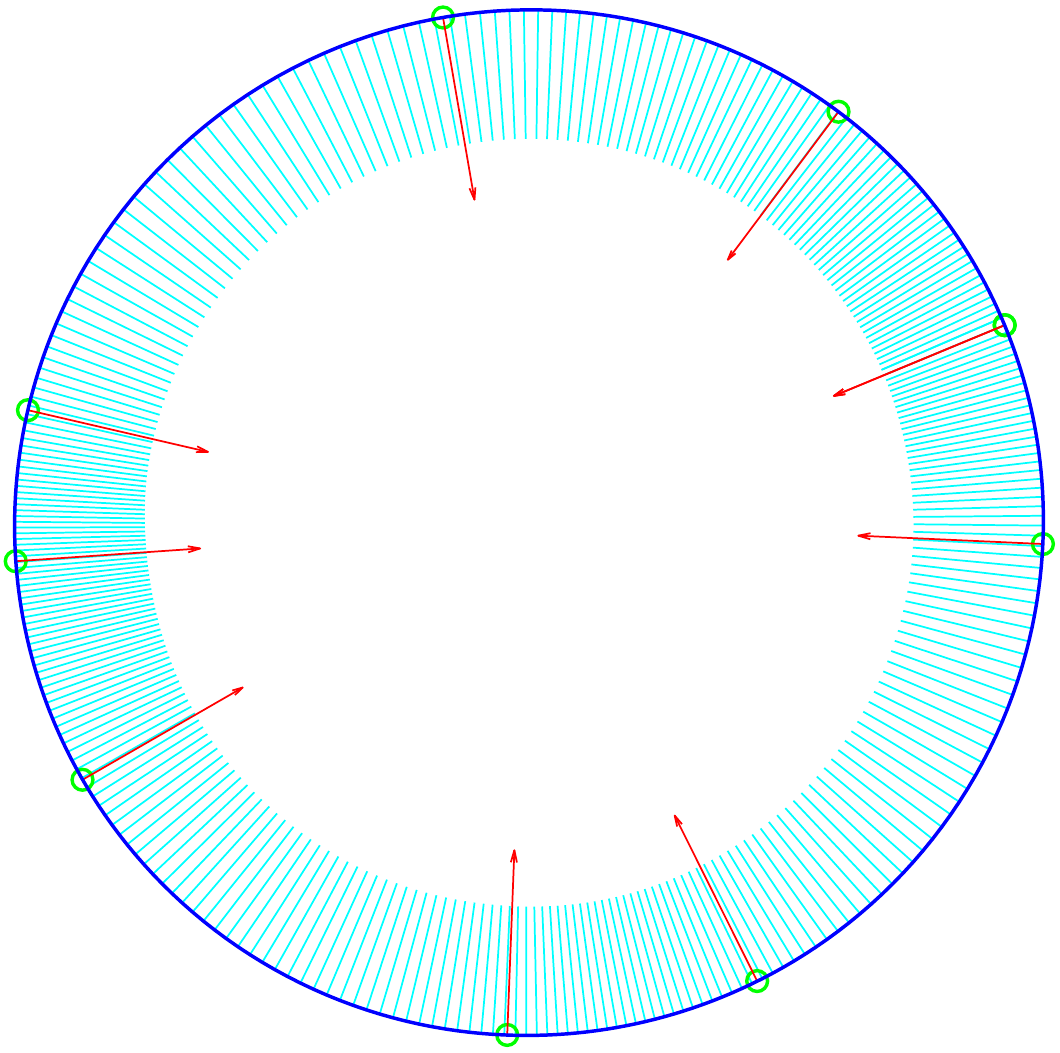}}
  \caption{PN subdivision curves with curvature combs: (a) PN-6-point subdivision; (b) PN cubic B-spline subdivision.}
  \label{Fig:CirclePreservingSubdivision}
\end{figure}

Figure \ref{Fig:CirclePreservingSubdivision} illustrates the circle preserving property of PN subdivision schemes. With unevenly sampled points and normals from a circle, two PN subdivision curves are obtained by PN-6-point subdivision scheme or by PN cubic B-spline subdivision scheme, respectively. The curvature combs show that both of the two PN subdivision schemes reproduce the circle exactly. As old vertices are generally not preserved by PN cubic B-spline subdivision, it generates a subdivision curve with more uniform vertices than the PN-6-point subdivision curve which interpolates all vertices.

\begin{property}
(Cylinder and sphere preserving). If the initial control data $(\mathbf{p}_i^0, \mathbf{n}_i^0)$, are sampled from a circular cylinder or a sphere, then all the newly subdivided points and normals $(\mathbf{p}_i^k, \mathbf{n}_i^k)$, $k\in\mathbb{Z}_+$, by PN subdivision lie on the same cylinder or sphere.
\end{property}
\begin{proof}
The proof of sphere preserving is the same as that for circle preserving, we prove the property of cylinder preserving.

W.l.o.g, we assume the generatrix of a cylinder is parallel to the $z$-axis, and the coordinates of the $k$th subdivided points and normals are given by $\mathbf{p}_i^k=(\mathbf{n}_{ix}^k,\mathbf{n}_{iy}^k,z_i^k)^\top$ and $\mathbf{n}_i^k=(\mathbf{n}_{ix}^k,\mathbf{n}_{iy}^k,0)^\top$. The perpendicular projection of the points and normals onto the $xy$-plane are $\bar{\mathbf{p}}_i^k=\bar{\mathbf{n}}_i^k=\mathbf{n}_i^k$.
Based on Equation (\ref{Eqn:PN binary subd}), the projection of the subdivided point $\mathbf{p}_i^{k+1}$ onto the $xy$-plane is obtained as
\[
\begin{array}{ccl}
\bar{\mathbf{p}}_i^{k+1} &=& \bar{\mathbf{q}}_i^{k+1} + \sum_j a_{i-2j}h_{ij}^k \bar{\mathbf{n}}_i^{k+1} \\
&=& \bar{\mathbf{q}}_i^{k+1} + \sum_j a_{i-2j} \frac{(\mathbf{n}_j^k + \mathbf{n}_i^{k+1})^\top(\mathbf{p}_j^k - \mathbf{q}_i^{k+1})}{(\mathbf{n}_j^k + \mathbf{n}_i^{k+1})^\top\mathbf{n}_i^{k+1}}  \bar{\mathbf{n}}_i^{k+1} \\
&=& \bar{\mathbf{q}}_i^{k+1} + \sum_j a_{i-2j} \frac{(\bar{\mathbf{n}}_j^k + \bar{\mathbf{n}}_i^{k+1})^\top(\bar{\mathbf{p}}_j^k - \bar{\mathbf{q}}_i^{k+1})}{(\bar{\mathbf{n}}_j^k + \bar{\mathbf{n}}_i^{k+1})^\top\bar{\mathbf{n}}_i^{k+1}}\bar{\mathbf{n}}_i^{k+1},
\end{array}
\]
where $\bar{\mathbf{q}}_i^{k+1} = \sum_j a_{i-2j} \bar{\mathbf{p}}_j^k$
and
$\bar{\mathbf{n}}_i^{k+1} = \frac{\sum_j a_{i-2j} \bar{\mathbf{n}}_j^k}{\|\sum_j a_{i-2j} \bar{\mathbf{n}}_j^k\|}$.
From the above expression we know that $\bar{\mathbf{p}}_i^{k+1}$ is also the subdivided point by the projected points $\bar{\mathbf{p}}_j^k$ and projected normals $\bar{\mathbf{n}}_j^k$.
Because the PN subdivision of the projected data $(\bar{\mathbf{p}}_i^0, \bar{\mathbf{n}}_i^0)$ is circle preserving, the PN subdivision of original data $(\mathbf{p}_i^0, \mathbf{n}_i^0)$ is cylinder preserving.
\end{proof}

\begin{figure}[h]
  \centering
  \subfigure[]{\includegraphics[width=0.25\linewidth]{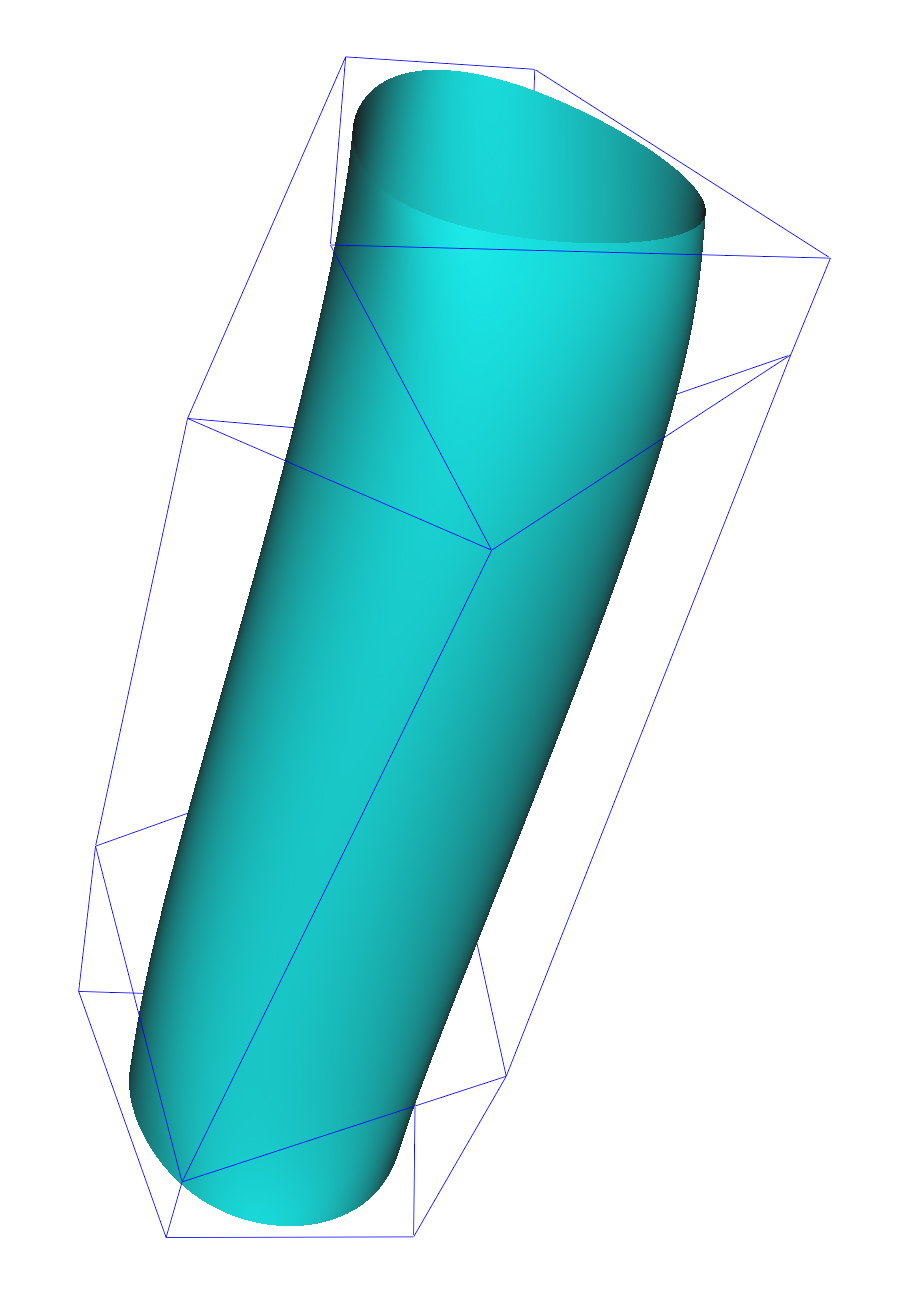}} \ \ \ \ \ \
  \subfigure[]{\includegraphics[width=0.25\linewidth]{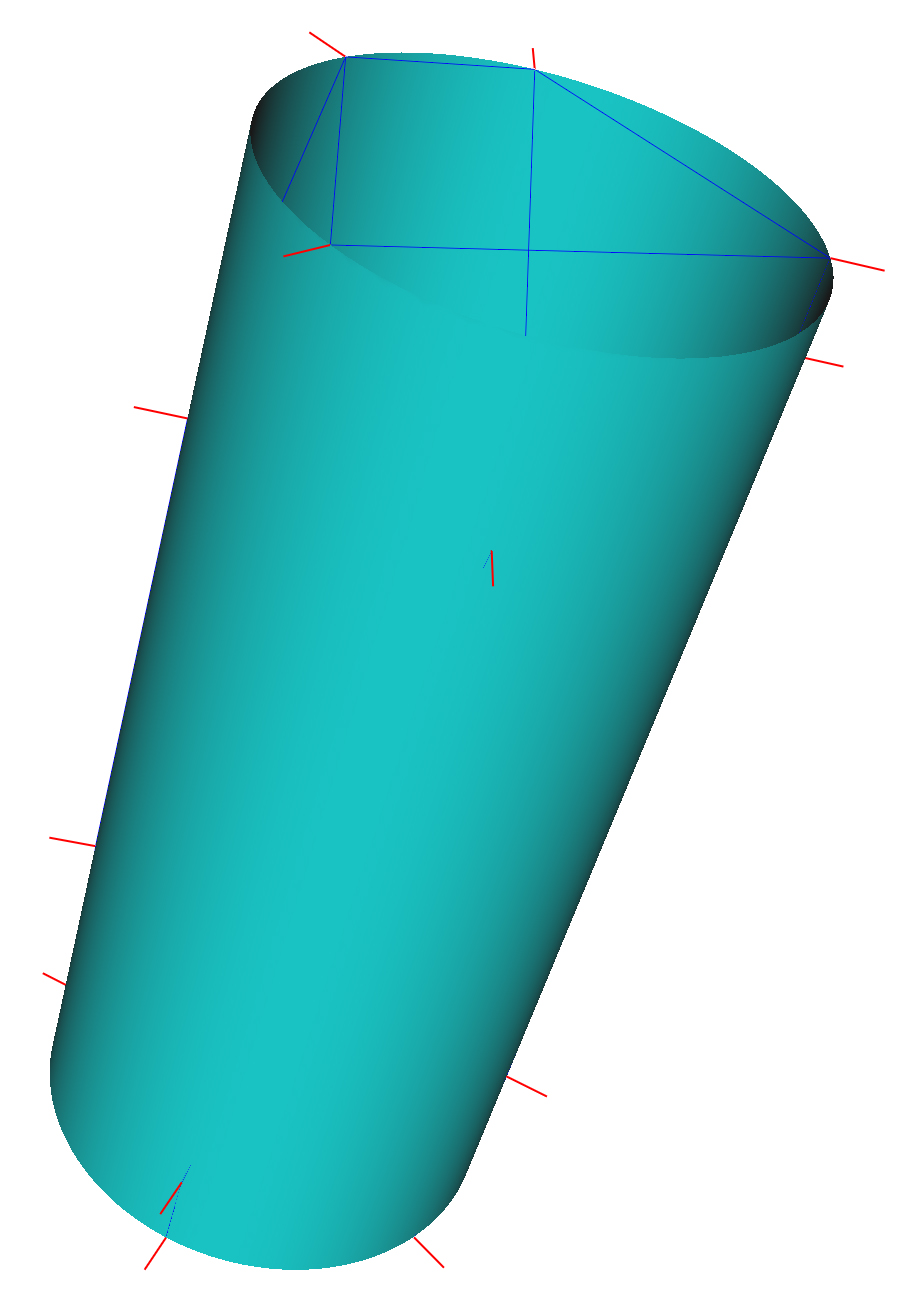}}
  \caption{(a) Catmll-Clark subdivision and (b) PN-Catmull-Clark subdivision of a quad mesh with vertices and normals sampled from a cylinder.}
  \label{Fig:CylinderPreservingSubdivision}
\end{figure}

\begin{figure}[h]
  \centering
  \subfigure[]{\includegraphics[width=0.25\linewidth]{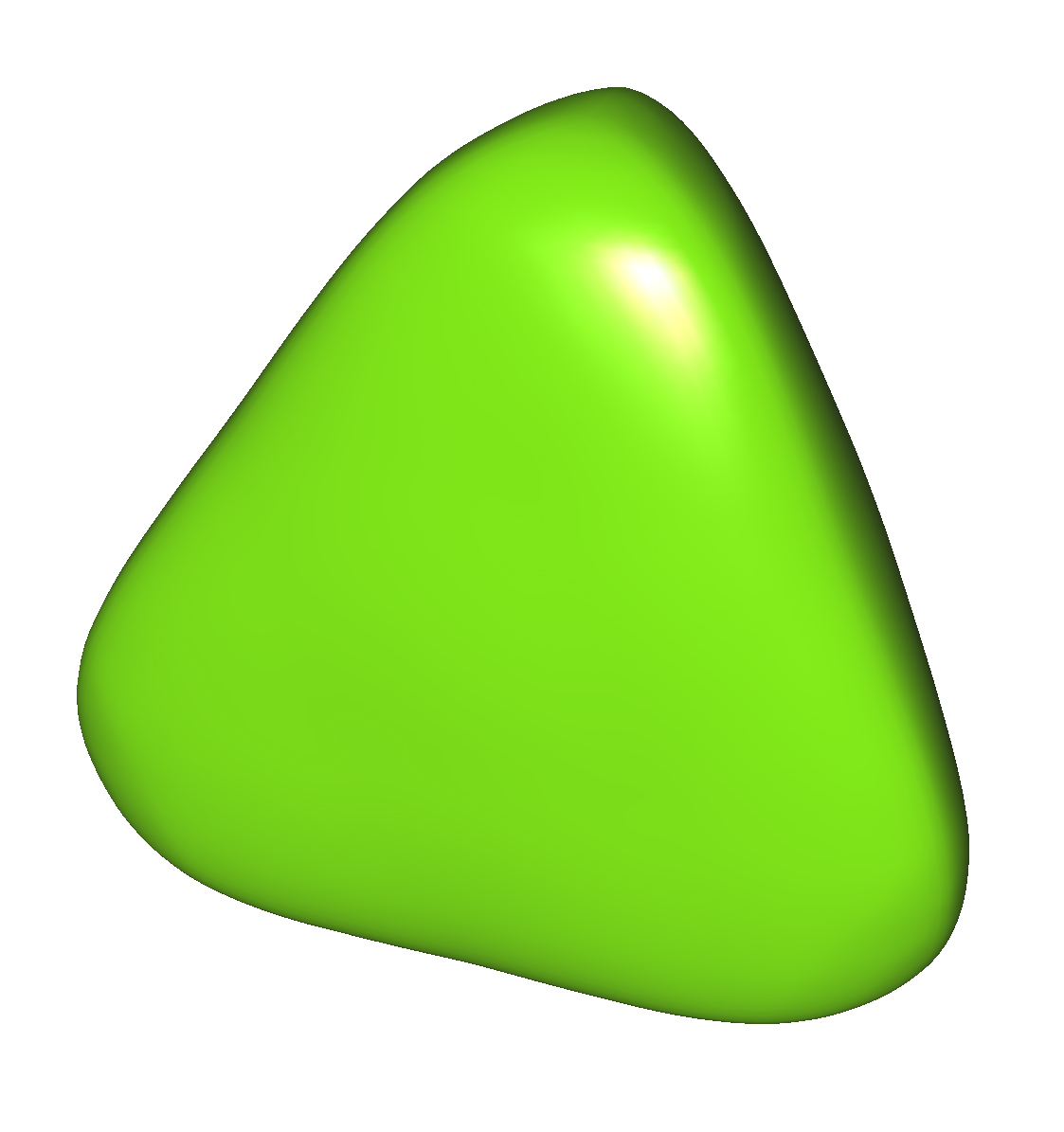}} \ \ \ \ \ \
  \subfigure[]{\includegraphics[width=0.25\linewidth]{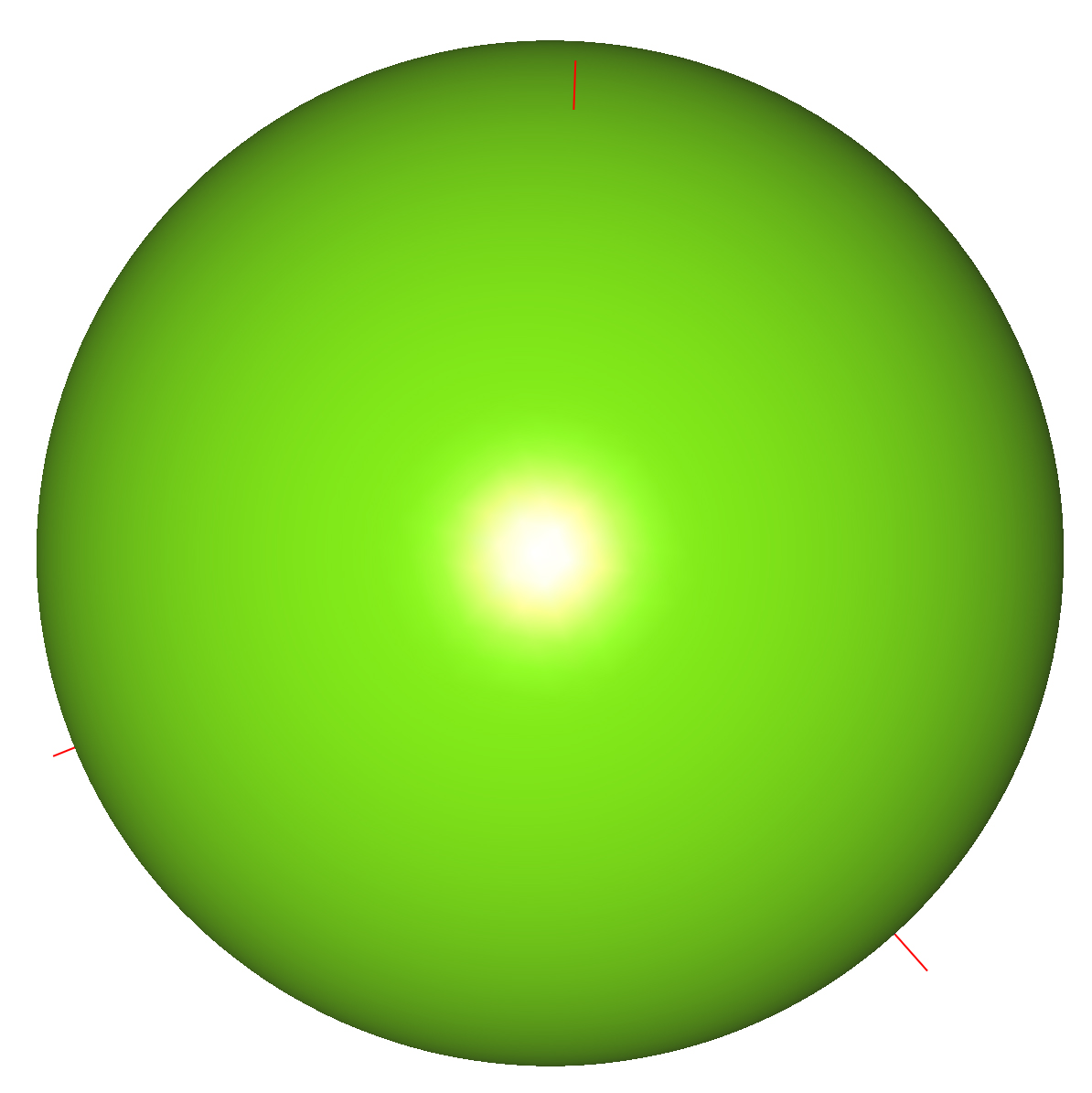}}
  \caption{(a) Butterfly and (b) PN-Butterfly subdivision of a triangular mesh with vertices and normals sampled from a sphere.}
  \label{Fig:SpherePreservingSubdivision}
\end{figure}

Figure \ref{Fig:CylinderPreservingSubdivision}(a) illustrates a quad mesh with vertices sampled from a circular cylinder and a deformed cylinder-like surface generated by traditional Catmull-Clark subdivision. If normal vectors at the vertices are also sampled, a circular cylinder surface that passes through all control vertices is obtained by PN-Catmull-Clark subdivision; see Figure \ref{Fig:CylinderPreservingSubdivision}(b). Figure \ref{Fig:SpherePreservingSubdivision}(a) illustrates a Butterfly subdivision surface constructed from a tetrahedron. By choosing all initial vertex normals as of a sphere, the PN-Butterfly subdivision surface reproduces the sphere exactly; see Figure \ref{Fig:SpherePreservingSubdivision}(b) for the obtained surface.

%%%%%%%%%%%%%%%%%%%%%%%%%%%%%%%%%%%%%%%%%%%%%%
%%%% Section 4
%%%%%%%%%%%%%%%%%%%%%%%%%%%%%%%%%%%%%%%%%%%%%%

\section{Convergence and smoothness analysis}
\label{Sec:Converngence}
This section presents convergence and smoothness analysis of the proposed subdivision schemes. As the convergence and smoothness of subdivided normals by Equation (\ref{Eqn:PN binary subd}) have already been discussed in Section \ref{Subsection:subdivision on sphere}, we pay our attention to convergence and smoothness analysis of PN subdivision curves and surfaces.

\subsection{Analysis of univariate PN subdivision schemes}
\label{Subsection:analysis of univariate PN subdivision}

We first analyze the convergence of univariate PN subdivision scheme defined by Equation (\ref{Eqn:PN binary subd}) or given by Equation (\ref{Eqn:PN binary subd reform}).
\begin{theorem}
\label{Theoem:convergence univar_PN subd}
Assume $S_a$ be the linear subdivision scheme as defined by Equation (\ref{Eqn:stationary binary subd}) and $\{S_a\circ M^k\}$ be the PN subdivision scheme as given by Equation (\ref{Eqn:PN binary subd reform}). If $S_a\in C^0$, then $\{S_a\circ M^k\} \approx S_a$ and $\{S_a\circ M^k\}$ converges.
\end{theorem}
\begin{proof}
We prove the convergence of the PN subdivision scheme $\{S_a\circ M^k\}$ by comparing with the linear subdivision scheme $S_a$.
Let
\begin{equation}
\label{Eqn:A(X1,X2)}
A(X_1,X_2) = \frac{X_1(X_1+X_2)^\top}{(X_1+X_2)^\top X_1}, \ \ \ \ X_1,X_2\in \mathbb{S}^2.
\end{equation}
The matrix $A_{ij}^k$ in Equation (\ref{Eqn:PN binary subd A_ij^k}) is given by $A_{ij}^k =  A(\mathbf{n}_i^{k+1},\mathbf{n}_j^k)$.
Under the assumption that $S_a\in C^0$, we know that the subdivided normal vectors $\mathbf{n}_i^k$ converge and the spherical subdivision curve $\mathbf{n}(t)$ discussed in Section \ref{Subsection:subdivision on sphere} is continuous. Assume $\mathbf{n}(t)$ has the H\"{o}lder regularity of $C^{0+\alpha}$, where $0<\alpha\leq1$. It follows that the function $A(X_1,\mathbf{n}(t))$ has also the H\"{o}lder regularity of $C^{0+\alpha}$ with the variable $X_1$ a fixed vector.

Suppose $|j-l|<B$, where $B$ is the bound of the support of the mask of $S_a$. Applying Equation (\ref{Eqn:holder regularity_normal}), we have  $\|\mathbf{n}_j^k - \mathbf{n}_l^k\| < Bc_12^{-k\alpha}$ and
\[
\begin{array}{ccl}
\|A_{ij}^k - A_{il}^k\|_\infty &=& \|A(\mathbf{n}_i^{k+1},\mathbf{n}_j^k) - A(\mathbf{n}_i^{k+1},\mathbf{n}_l^k)\|_\infty  \\
&\leq& c_A 2^{-k\alpha} \\
&\leq& c_A \gamma^k,
\end{array}
\]
where $c_A$ is a constant and $\gamma=2^{-\alpha}\in(0,1)$.
Based on the expression $M_{ij}^k = I+\sum_{l\in\mathbb{Z}}a_{i-2l}(A_{ij}^k-A_{il}^k)$, we have $\|M_{ij}^k-I\|_\infty \leq \|S_a\|_\infty c_A\gamma^k$. It follows that
\[
\|(S_a)_i\circ M_i^k - (S_a)_i\|_\infty = \sum_{j}\|a_{i-2j}(M_{ij}^k - I)\|_\infty \leq c_M\gamma^k,
\]
where $c_M=\|S_a\|^2_\infty c_A$. Since the constant $c_M$ is independent of the index $i$, we have
\[
\|S_a\circ M^k -S_a\|_\infty \leq c_M\gamma^k.
\]
From this inequality, we have $\sum_{k}\|S_a\circ M^k -S_a\|_\infty < +\infty$. This implies that $\{S_a\circ M^k\}\approx S_a$.
Based on Proposition~\ref{Proposition:asympototically equivalent} we know that the PN subdivision scheme $\{S_a\circ M^k\}$ converges.
\end{proof}

We then prove the higher orders of smoothness of univariate PN subdivision schemes. In the remainder part of this subsection we assume that the linear subdivision scheme $S_a\in C^m$ is given by the symbol $a(z)=\frac{(1+z)^m}{2^m}b(z)$, where $S_b$ has the H\"{o}lder regularity of $C^{0+\alpha}$. Let $a_0(z)=a(z)$, $a_1(z)=\frac{(1+z)^{m-1}}{2^{m-1}}b(z)$, $\ldots$, $a_{m-1}(z)=\frac{1+z}{2}b(z)$ and $a_m(z)=b(z)$. The corresponding subdivision schemes are referred as $S_a$, $S_{a_1}$, $\ldots$, $S_{a_{m-1}}$ and $S_{a_m}$, respectively.
We also introduce partial differences when a sequence has two sub-indexes:
\[
%\begin{array}{ccl}
\Delta_1 M_{ij}^k = M_{ij}^k - M_{i-1,j}^k, \ \
\Delta_2 M_{ij}^k = M_{ij}^k - M_{i,j-1}^k.
%\end{array}
\]
For sequences with only one sub-index, the finite difference $\Delta_1 \mathbf{p}_i^k$ is given by
$
\Delta_1 \mathbf{p}_i^k = \Delta \mathbf{p}_i^k = \mathbf{p}_i^k - \mathbf{p}_{i-1}^k.
$
Similarly, we have $\Delta_2 \mathbf{p}_j^k = \Delta \mathbf{p}_j^k = \mathbf{p}_j^k - \mathbf{p}_{j-1}^k$.
Based on first order partial differences, higher order partial differences will be computed by operators $\Delta_1^m=\Delta_1(\Delta_1^{m-1})$, $\Delta_1\Delta_2=\Delta_1(\Delta_2)$, etc.

Before presenting the main theorem for the smoothness analysis, we introduce a lemma about the norm estimation of the differences of the coefficient matrices given in Equation (\ref{Eqn:PN binary subd reform}).
\begin{lemma}
\label{lemma:difference of matrices}
Assume $S_a$ is a linear binary subdivision scheme with mask defined by the symbol $a(z)=\frac{(1+z)^m}{2^m}b(z)$, where $S_b\in C^0$. Assume the matrices $M_{ij}^k$ are given by Equation (\ref{Eqn:PN binary subd reform}). Then for any nonnegative integers $m_1$, $m_2$ satisfying $1\leq m_1+m_2\leq m$, the following inequality holds
\[
\left\|\frac{\Delta_1^{m_1}\Delta_2^{m_2}M_{ij}^k}{2^{-(m_1+m_2)k}}\right\|_\infty \leq c_m\gamma^k,
\]
where $c_m$ and $0<\gamma<1$ are constants.
\end{lemma}
\begin{proof}
Let $U_{ij}^k=(\ldots;A_{ij}^k-A_{i,l-1}^k;A_{ij}^k-A_{il}^k;A_{ij}^k-A_{i,l+1}^k;\ldots)$. Then the matrix $M_{ij}^k$ can be rewritten as $M_{ij}^k = I+(S_a)_i U_{ij}^k$. The finite differences of the matrices can be computed by the Leibniz rule and Equation (\ref{Eqn:Delta m Sa_i}) as follows
\[
%\begin{array}{ccl}
\begin{aligned}
\frac{\Delta_1^{m_1}\Delta_2^{m_2}M_{ij}^k}{2^{-(m_1+m_2)k}}
&=\frac{\Delta_1^{m_1}\Delta_2^{m_2}[(S_a)_i U_{ij}^k]}{2^{-(m_1+m_2)k}}  \\
&=\frac{1}{2^{-(m_1+m_2)k}} \Delta_1^{m_1}[(S_a)_i \Delta_2^{m_2}U_{ij}^k] \\
&=\frac{1}{2^{-(m_1+m_2)k}} \big[\Delta_1^{m_1}(S_a)_i \Delta_2^{m_2}U_{ij}^k
  +C_{m_1}^1 \Delta_1^{m_1-1}(S_a)_i \Delta_1\Delta_2^{m_2}U_{ij}^k
  +\cdots
  +C_{m_1}^{m_1} (S_a)_i \Delta_1^{m_1}\Delta_2^{m_2}U_{ij}^k\big]     \\
&=\frac{(S_{a_{m_1}})_i}{2^{m_1}} \frac{\Delta_2^{m_1+m_2}U_{ij}^k}{2^{-(m_1+m_2)k}}
  +C_{m_1}^1 \frac{(S_{a_{m_1-1}})_i}{2^{m_1-1}} \frac{\Delta_1\Delta_2^{m_1+m_2-1}U_{ij}^k}{2^{-(m_1+m_2)k}}
  +\cdots
  +C_{m_1}^{m_1} (S_a)_i \frac{\Delta_1^{m_1}\Delta_2^{m_2}U_{ij}^k}{2^{-(m_1+m_2)k}},
%\end{array}
\end{aligned}
\]
where $C_{m_1}^l=\frac{{m_1}!}{l!(m_1-l)!}$.
As discussed in Section \ref{Subsection:subdivision on sphere}, the spherical subdivision curve $\mathbf{n}(t)$ has the H\"{o}lder regularity of $C^{m+\alpha}$ when the linear subdivision scheme $S_a\in C^m$. It follows that the function $A(X_1(t),X_2(t))$ also has the H\"{o}lder regularity of $C^{m+\alpha}$ when $X_1(t)\approx X_2(t)$ are the spherical subdivision curves.
Since the support of the mask of $S_a$ is bounded, we assume that $|j-l|$ is bounded too.
For any $s_1,s_2\in \mathbb{N}_0$, $s_1+s_2=m_1+m_2\leq m$, we have
\[
%\begin{aligned}
\begin{array}{cl}
\left\|\frac{\Delta_1^{s_1}\Delta_2^{s_2}U_{ij}^k}{2^{-(s_1+s_2)k}}\right\|_\infty
&= \left\|\frac{\Delta_1^{s_1}\Delta_2^{s_2}(A_{ij}^k-A_{il}^k)}{2^{-(s_1+s_2)k}}\right\|_\infty   \\
&= \left\|\frac{\Delta_1^{s_1}\Delta_2^{s_2}(A(\mathbf{n}_i^{k+1},\mathbf{n}_j^k)-A(\mathbf{n}_i^{k+1},\mathbf{n}_l^k))}{2^{-(s_1+s_2)k}}\right\|_\infty \\
&\leq c_s\gamma^k,
%\end{aligned}
\end{array}
\]
where $c_s$ and $\gamma\in(0,1)$ are constants.
Let
\[
k_a=\max_{0\leq m_1\leq m} \left\{\frac{|S_{a_{m_1}}|_\infty}{2^{m_1}} + C_{m_1}^1 \frac{|S_{a_{m_1-1}}|_\infty}{2^{m_1-1}}
    +\cdots+C_{m_1}^{m_1}|S_a|_\infty \right\}.
\]
We have
\[
\left\|\frac{\Delta_1^{m_1}\Delta_2^{m_2}M_{ij}^k}{2^{-(m_1+m_2)k}}\right\|_\infty
\leq k_ac_s\gamma^k \doteq c_m\gamma^k.
\]
This completes the proof.
\end{proof}

We now show that the univariate PN subdivision schemes have the same orders of smoothness as linear subdivision schemes.
\begin{theorem}
\label{Theorem:smoothness for univariate PN subd}
Assume $S_a$ is a linear binary subdivision scheme with mask defined by the symbol $a(z)=\frac{(1+z)^m}{2^m}b(z)$, where $S_b\in C^0$. Let $\{S_a\circ M^k\}$ be the PN subdivision scheme originally defined by Equation (\ref{Eqn:PN binary subd}). Then $\{S_a\circ M^k\}\in C^m$.
\end{theorem}
\begin{proof}
We prove the theorem by induction. From Theorem \ref{Theoem:convergence univar_PN subd} we know that $\{S_a\circ M^k\}\in C^0$. We then prove that $\{S_a\circ M^k\}\in C^m$ under the assumption that $\{S_a\circ M^k\}\in C^l$, $l=0,1,\ldots,m-1$.

From Equation (\ref{Eqn:P_i k+1 in matrix form}) we have $\mathbf{p}_i^{k+1}=(S_a)_i ((M_i^k)^{\top_{blk}}\circ P^k)$. We first compute the differences of the point sequence $\{\mathbf{p}_i^{k+1}\}$ using the Leibniz rule and Equation (\ref{Eqn:Delta m Sa_i}):
\[
\begin{aligned}
 \Delta^m \mathbf{p}_i^{k+1}
&= \Delta_1^m \{(S_a)_i [(M_i^k)^{\top_{blk}}\circ P^k]\} \\
&= \Delta_1^m (S_a)_i[(M_i^k)^{\top_{blk}} \circ P^k] + C_m^1 \Delta_1^{m-1} (S_a)_i[\Delta_1(M_i^k)^{\top_{blk}} \circ P^k]
   + \cdots + C_m^m (S_a)_i[\Delta_1^m(M_i^k)^{\top_{blk}} \circ P^k]    \\
&= \frac{(S_b)_i}{2^m}\Delta_2^m[(M_i^k)^{\top_{blk}}\circ P^k]
   + C_m^1\frac{(S_{a_{m-1}})_i}{2^{m-1}} \Delta_2^{m-1}[\Delta_1(M_i^k)^{\top_{blk}}\circ P^k]
   +\cdots + C_m^m (S_a)_i[\Delta_1^m(M_i^k)^{\top_{blk}} \circ P^k] \\
&= \frac{(S_b)_i}{2^m} \big[(M_i^k)^{\top_{blk}} \circ \Delta^m P^k + C_m^1\Delta_2(M_i^k)^{\top_{blk}}\circ \Delta^{m-1} P^k + \cdots
   + C_m^m\Delta_2^m(M_i^k)^{\top_{blk}}\circ P^k \big]  \\
& \ \ \ \ +C_m^1\frac{(S_{a_{m-1}})_i}{2^{m-1}}\big[\Delta_1(M_i^k)^{\top_{blk}}\circ \Delta^{m-1}P^k
   +C_{m-1}^1\Delta_2\Delta_1(M_i^k)^{\top_{blk}}\circ \Delta^{m-2}P^k + \cdots
   + C_{m-1}^{m-1}\Delta_2^{m-1}\Delta_1(M_i^k)^{\top_{blk}} \circ P^k \big] \\
& \ \ \ \ +\cdots \\
& \ \ \ \ +C_m^m (S_a)_i\left[\Delta_1^m(M_i^k)^{\top_{blk}} \circ P^k\right].
\end{aligned}
\]
From this expression, we have
\[
\begin{aligned}
\frac{\Delta^m \mathbf{p}_i^{k+1}}{2^{-m(k+1)}} &= ((S_b)_i\circ M_i^k)\frac{\Delta^m P^k}{2^{-mk}}\\
&  \ \ \ \ + (S_b)_i \bigg[C_m^1\frac{\Delta_2(M_i^k)^{\top_{blk}}}{2^{-k}}\circ \frac{\Delta^{m-1} P^k}{2^{-(m-1)k}} + \cdots
   + C_m^m\frac{\Delta_2^m(M_i^k)^{\top_{blk}}}{2^{-mk}}\circ P^k \bigg]  \\
&  \ \ \ \ + 2C_m^1(S_{a_{m-1}})_i \bigg[\frac{\Delta_1(M_i^k)^{\top_{blk}}}{2^{-k}}\circ \frac{\Delta^{m-1}P^k}{2^{-(m-1)k}}
 +C_{m-1}^1\frac{\Delta_2\Delta_1(M_i^k)^{\top_{blk}}}{2^{-2k}}\circ \frac{\Delta^{m-2}P^k}{2^{-(m-2)k}} + \cdots
 + C_{m-1}^{m-1}\frac{\Delta_2^{m-1}\Delta_1(M_i^k)^{\top_{blk}}}{2^{-mk}} \circ P^k\bigg] \\
&  \ \ \ \ + \cdots \\
&  \ \ \ \ + 2^m C_m^m (S_a)_i\left[\frac{\Delta_1^m(M_i^k)^{\top_{blk}}}{2^{-mk}} \circ P^k\right].
\end{aligned}
\]
Under the assumption that $\{S_a\circ M^k\}\in C^l$, $l=0,1,\ldots,m-1$, we have
\[
\left\|\frac{\Delta^l P^k}{2^{-lk}}\right\|_\infty < K, \ \ \ \ \  0\leq l <m,
\]
where $K$ is the bound of the derivatives of the subdivision curve as well as the bound of the finite differences of the sequence of subdivided points. From Lemma \ref{lemma:difference of matrices} we know that the differences of all element matrices of $(M_i^k)^{\top_{blk}}$ within above equation have a bound $c_m\gamma^k$. Then the above equation can be simplified as
\[
\frac{\Delta^m \mathbf{p}_i^{k+1}}{2^{-m(k+1)}} = ((S_b)_i\circ M_i^k)\frac{\Delta^m P^k}{2^{-mk}} + \varepsilon_i^k,
\]
where $\|\varepsilon_i^k\|<c\gamma^k$ and $\gamma\in(0,1)$.
Since $S_b\in C^0$, and by Theorem \ref{Theoem:convergence univar_PN subd}, we know $\{S_b\circ M^k\}$ converges. Based on Proposition \ref{Proposition:perturb subdivision} we know that the difference sequence $\{\frac{\Delta^m P^k}{2^{-mk}}\}_{k\in\mathbb{N}}$ converges too when $k$ approaches infinity. This implies that $\{S_a\circ M^k\}\in C^m$.
\end{proof}

Besides by subdividing the old normals using scheme $S_a$ and projecting the linearly subdivided normals onto sphere, the normal vectors within Equation (\ref{Eqn:PN binary subd}) can also be generated by masks of schemes other than $S_a$ or sampled directly from a smooth curve on sphere. In the same way as the proof of Theorem {\ref{Theorem:smoothness for univariate PN subd}} we obtain the smoothness orders of this kind of PN subdivision schemes.
\begin{corollary}
\label{Corollary:smoothness for univariate PN subd}
Assume $S_a\in C^m$ and $S_{a'}\in C^{m'}$ are two binary linear subdivision schemes, where $m,m'\in \mathbb{N}_0$. If a PN subdivision scheme is defined by Equation (\ref{Eqn:PN binary subd}) with points computed using mask of $S_a$ and with unit normals computed using mask of $S_{a'}$, then the PN subdivision scheme $\{S_a\circ M^k\}\in C^{\min\{m,m'\}}$.
\end{corollary}

\begin{figure*}[htb]
  \centering
  \subfigure[]{\includegraphics[width=0.28\linewidth]{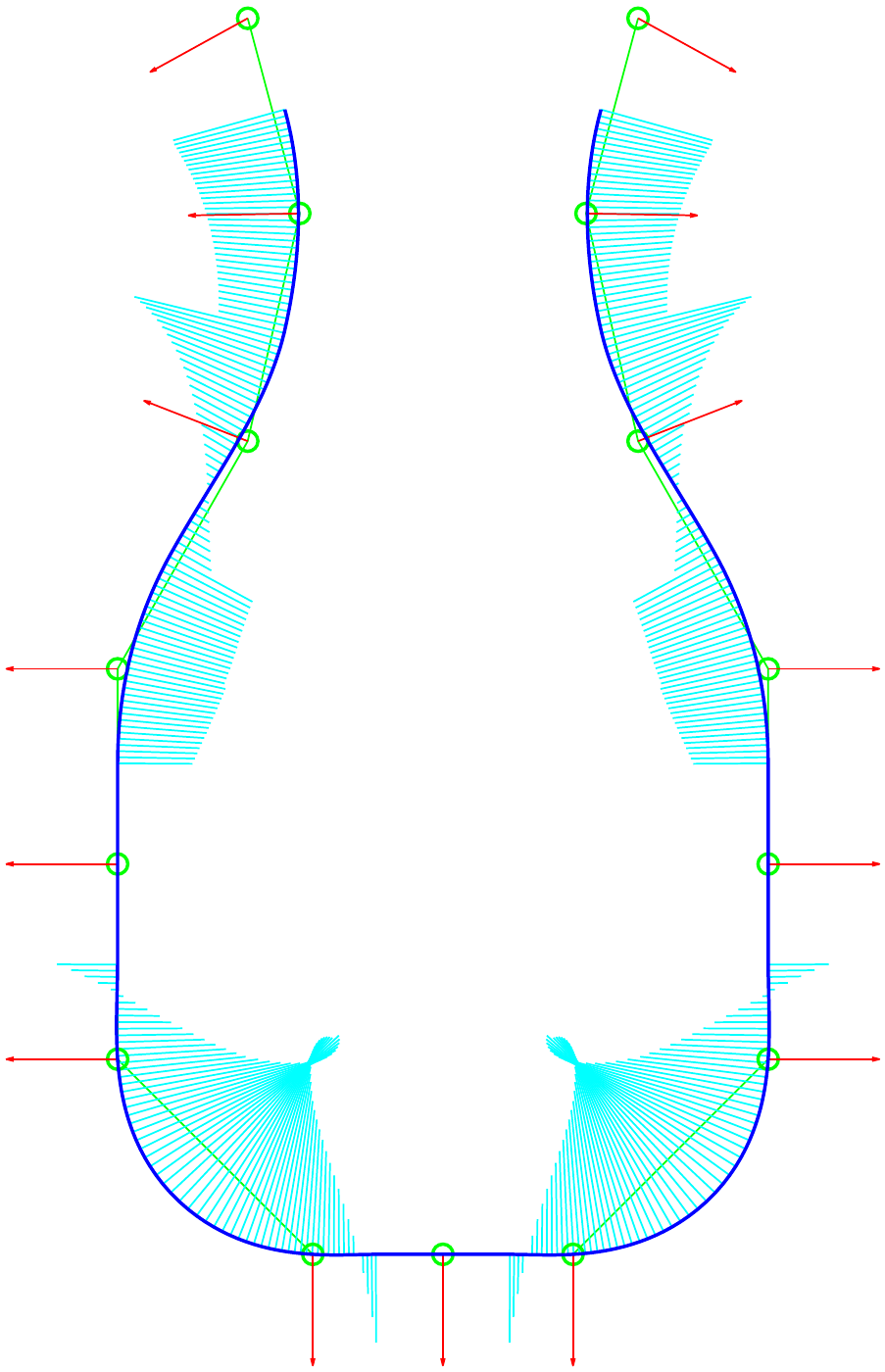}} \ \ \ \ \
  \subfigure[]{\includegraphics[width=0.28\linewidth]{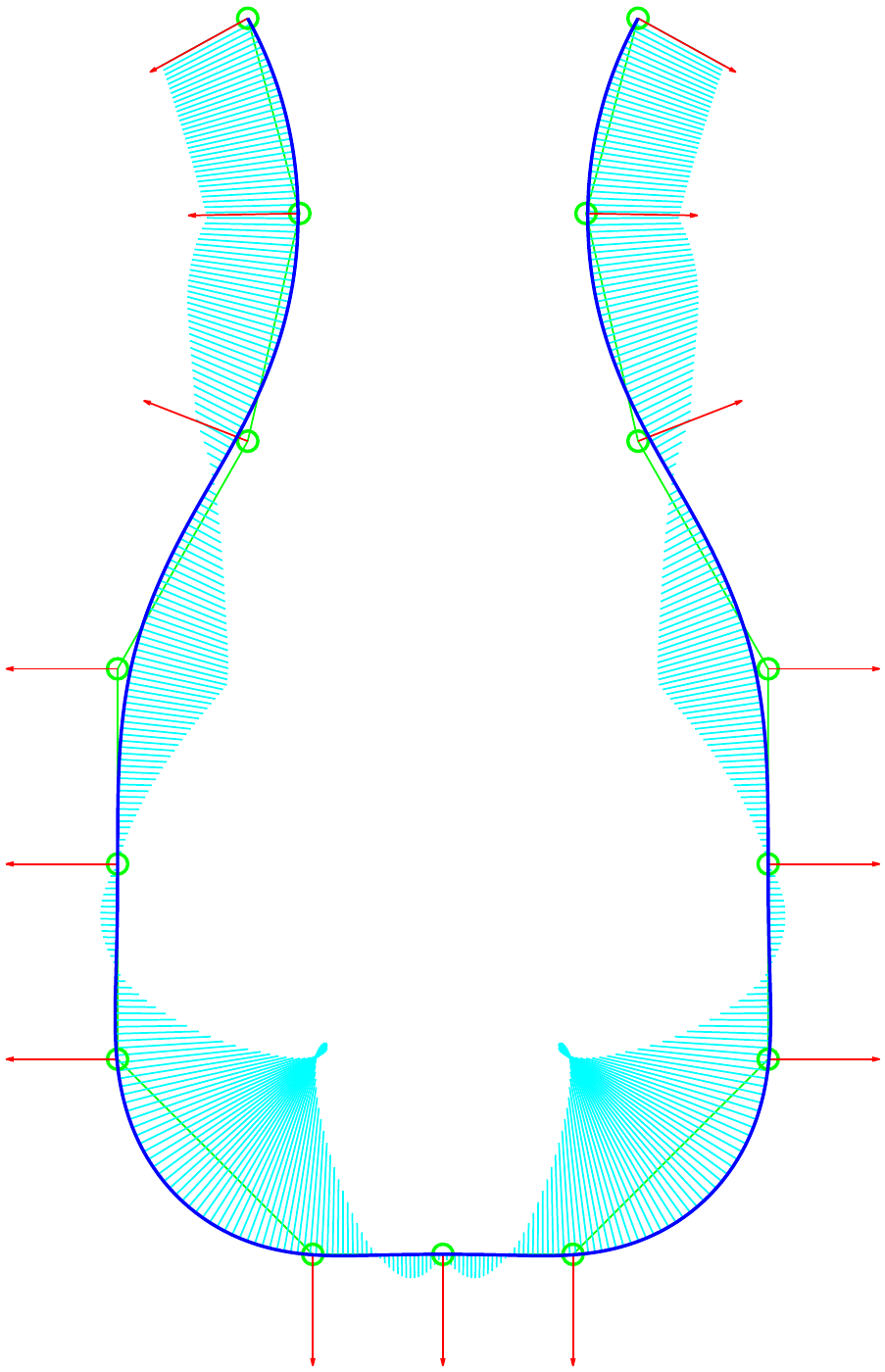}} \ \ \ \ \
  \subfigure[]{\includegraphics[width=0.28\linewidth]{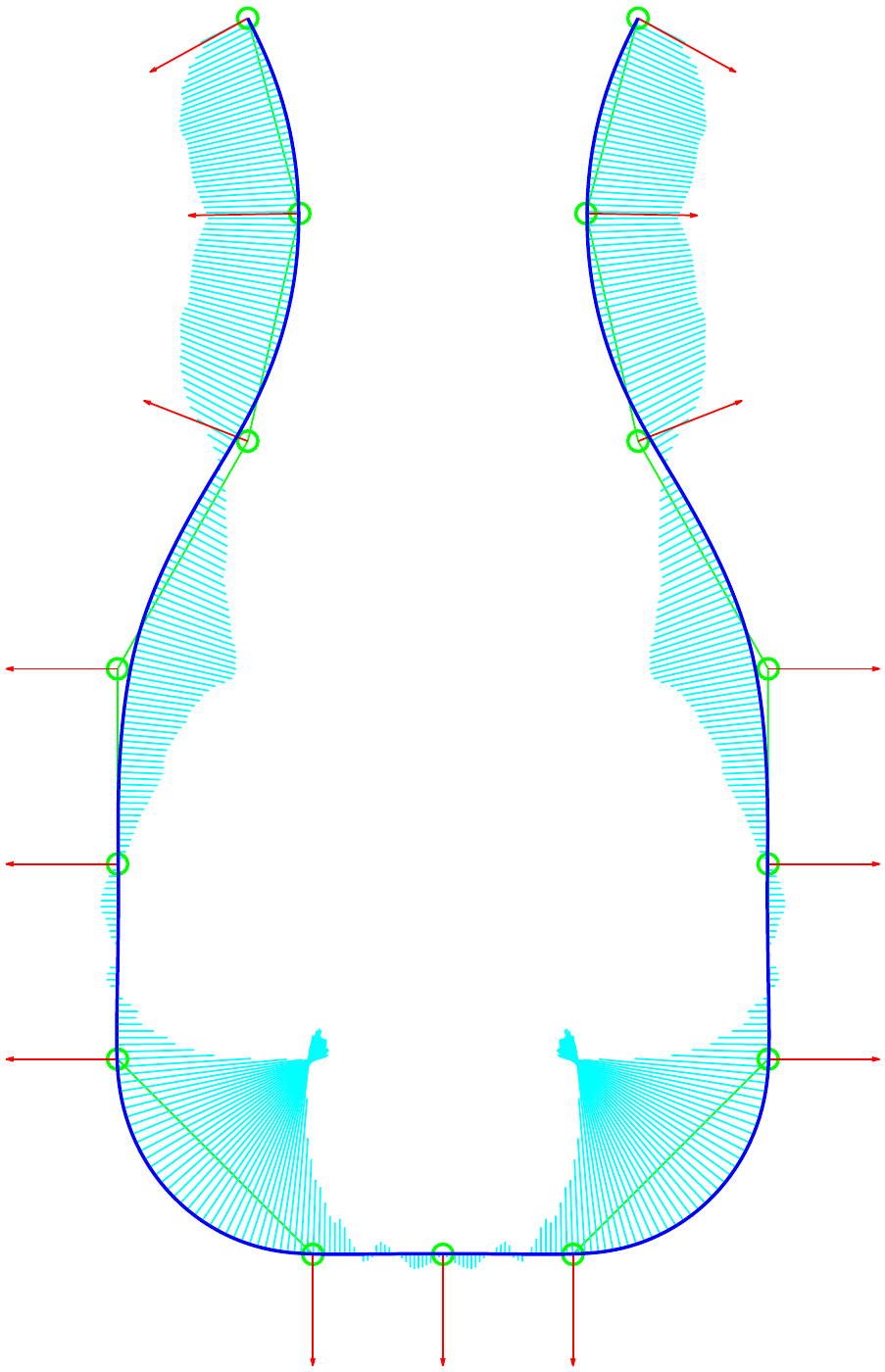}}
  \caption{PN B-spline subdivision curves with curvature combs: (a) PN quadratic B-spline subdivision; (b) PN cubic B-spline subdivision; (c) PN cubic B-spline subdivision using normal field generated by spherical 4-point subdivision.}
  \label{Fig:bottle PN Bspline Subdivision}
\end{figure*}

It is known that a uniform B-spline subdivision curve of degree $m$ has continuity order of $m-1$. From Theorem \ref{Theorem:smoothness for univariate PN subd} we know that a PN B-spline subdivision curve of degree $m$ has also the continuity order of $m-1$. Figure \ref{Fig:bottle PN Bspline Subdivision}(a) illustrates a PN quadratic B-spline subdivision curve. It is clear that the subdivision curve is tangent continuous but not curvature continuous. The PN cubic B-spline subdivision curve illustrated in Figure \ref{Fig:bottle PN Bspline Subdivision}(b) is curvature continuous, just as expected as a cubic B-spline curve.
Figure \ref{Fig:bottle PN Bspline Subdivision}(c) illustrates another PN cubic B-spline subdivision curve, but using normal field generated by spherical 4-point subdivision scheme. As 4-point subdivision has only $C^1$ continuity, the obtained PN subdivision curve is no longer as smooth as that in Figure \ref{Fig:bottle PN Bspline Subdivision}(b).

\subsection{Analysis of PN subdivision on irregular meshes}
\label{Subsection:analysis of PN subdivision surface}

Corresponding to the theoretical analysis of linear subdivision on irregular meshes, convergence and smoothness analysis of PN subdivision on irregular meshes also consists of two parts: analysis of PN subdivision on regular meshes and analysis of PN subdivision on meshes surrounding an extraordinary vertex or face.

Same as univariate subdivision, bivariate linear subdivision on regular quad meshes or triangular meshes can also be generalized to PN subdivision using Equation (\ref{Eqn:PN binary subd}). If the symbol $a(z_1,z_2)=\sum_{(i_1,i_2)\in\mathbb{Z}^2}a_{i_1,i_2}z_1^{i_1}z_2^{i_2}$ for a bivariate subdivision scheme is factorizable, the convergence and smoothness of the obtained PN subdivision scheme can be analyzed in the same way as univariate PN subdivision. Similar to Theorem \ref{Theorem:smoothness for univariate PN subd} and Corollary \ref{Corollary:smoothness for univariate PN subd}, the smoothness order of any bivariate PN subdivision on regular meshes can be derived from the smoothness order of the linear subdivision $S_a$ and the smoothness order of the subdivided normal field. Based on the smoothness equivalence between projection based bivariate subdivision and linear bivariate subdivision for regular control meshes (see Theorem 2.6 and Corollary 2.7 in \citep{Weinmann2012Subdwithgeneraldilation}), we know that the smoothness order of the subdivided normal field and the smoothness order of bivariate PN subdivision on regular meshes are the same as that for the corresponding linear subdivision scheme.

We present here the convergence and normal continuity analysis of PN subdivision of irregular quad meshes surrounding an isolated extraordinary vertex, the same result holds for subdivision of irregular quad meshes surrounding an extraordinary face or irregular triangle meshes surrounding an extraordinary vertex.
By taking the notations used in Section \ref{Subsection:subdivision of irregular meshes}, we assume $Q=(\mathbf{q}_0;\ldots;\mathbf{q}_{\bar{l}})$ be a set of control points surrounding an isolated extraordinary vertex and $N=(\mathbf{n}_0;\ldots;\mathbf{n}_{\bar{l}})$ be the initial control normals at the control points. Let $S=(s_{ij})_{0\leq i,j\leq \bar{l}}$ be the subdivision matrix and $Q_m=(\mathbf{q}_0^m;\ldots;\mathbf{q}_{\bar{l}}^m)$ be the control points for the $m$th surface ring. Assume the normal vectors $\mathbf{n}_i^k$, $0\leq i\leq \bar{l}$, at the control points are refined by Equation (\ref{Eqn:refinement of normals at extra vertex}). Let
\begin{equation}
\label{Eqn:M_k square}
\tilde{M}_k=\left(
      \begin{array}{ccc}
        \tilde{M}_{0,0}^k & \cdots & \tilde{M}_{0,\bar{l}}^k  \\
        \vdots   & \ddots   & \vdots    \\
        \tilde{M}_{\bar{l},0}^k & \cdots & \tilde{M}_{\bar{l},\bar{l}}^k  \\
      \end{array}
    \right),
\end{equation}
where $\tilde{M}_{ij}^k = I + \sum_{l=0}^{\bar{l}}s_{il}(A_{ij}^k-A_{il}^k)$ and $A_{ij}^k=A(\mathbf{n}_i^{k+1},\mathbf{n}_j^k)$ using Equation (\ref{Eqn:A(X1,X2)}).
Then the control points for the surface ring $\mathbf{\tilde{x}}_m$ by PN subdivision are computed by
\begin{equation}
\label{Eqn:Q_m = Sm...S1Q}
Q_m = S_{m-1}\cdots S_1 S_0 Q,
\end{equation}
where $S_k=S\circ \tilde{M}_k$, $k=0,\ldots,m-1$. We denote the subdivision scheme as $\{S_k\}$.
It is verified that $\sum_{j=0}^{\bar{l}}s_{ij}\tilde{M}_{ij}^k=I$ for $i=0,1,\ldots,\bar{l}$. Then we have
\begin{equation}
\label{Eqn:S_k I1}
S_k(I;I;\ldots;I) = S_k(I\mathbbm{1}) = I\mathbbm{1}.
\end{equation}

For convenience of comparison between $S_kQ_m$ and $SQ_m$ in the following text, we introduce matrix $E$ as
\[
E=\left(
      \begin{array}{ccc}
        I & \cdots & I  \\
        \vdots   & \ddots   & \vdots    \\
        I & \cdots & I  \\
      \end{array}
    \right)_{(\bar{l}+1)\times(\bar{l}+1)}
\]
such that $SQ_m=(S\circ E)Q_m$.
From the control points $Q_m$ and based on Equation (\ref{Eqn:x^m(s)}), a surface ring is obtained as $\mathbf{x}_m(\mathbf{s})=G(2^m\mathbf{s})Q_m$, where $G$ is the vector of scalar valued generating functions. On the other hand, $\mathbf{x}_m(\mathbf{s})$ can also be generated from the control mesh by linear subdivision directly. Similar to uniform refinement of curves \citep{MicchelliPrautzsch:UniformRefinement}, for any coordinates $\mathbf{s}\in\mathbf{S}_n^m$, the point $\mathbf{x}_m(\mathbf{s})$  can be computed recursively as follows
\begin{equation}
\label{Eqn:x_m(s)1}
\lim_{j\rightarrow +\infty}B_{m+j}(\mathbf{s})\cdots B_{m+1}(\mathbf{s})B_m(\mathbf{s})Q_m = \mathbf{x}_m(\mathbf{s})\mathbbm{1},
\end{equation}
where $B_{m+j}(\mathbf{s})$, $j=0,1,\ldots$, are the matrices for binary subdivision for regular control meshes with a fixed size. Correspondingly, the point on the surface ring by PN subdivision is obtained as
\begin{equation}
\label{Eqn:x_tilde m(s)1}
\lim_{j\rightarrow +\infty}\tilde{B}_{m+j}(\mathbf{s})\cdots \tilde{B}_{m+1}(\mathbf{s})\tilde{B}_m(\mathbf{s})Q_m = \mathbf{\tilde{x}}_m(\mathbf{s})\mathbbm{1},
\end{equation}
where $\tilde{B}_{m+j}(\mathbf{s})=B_{m+j}(\mathbf{s})\circ \tilde{M}_{m+j}(\mathbf{s})$, $j=0,1,\ldots$, and the matrices $\tilde{M}_{m+j}(\mathbf{s})$ are defined in a similar way as Equation (\ref{Eqn:M_k square}) using the refined control normals at the subdivided points. See Figure~\ref{Fig:PNsubd_mesh_rings} for the surface rings computed by PN subdivision from control points and control normals or by linear subdivision from the same sequence of control meshes.

Let $\mathbf{e}_0=(I;0;\cdots;0)$, $\mathbf{e}_1=(0;I;\cdots;0)$, $\ldots$, $\mathbf{e}_{\bar{l}}=(0;0;\cdots;I)$. Assume $\mathbf{\tilde{g}}_l^m(\mathbf{s})$, $l=0,1,\ldots,\bar{l}$, are the generating functions computed by Equation (\ref{Eqn:x_tilde m(s)1}) with $Q_m$ replaced by $\mathbf{e}_l$, $l=0,1,\ldots,\bar{l}$. By the same reason as Equation (\ref{Eqn:S_k I1}), we have $\tilde{B}_{m+j}(\mathbf{s})(I\mathbbm{1})=I\mathbbm{1}$, $j=0,1,\ldots$. It follows
\[
\begin{aligned}
    I\mathbbm{1}&=\lim_{j\rightarrow +\infty}\tilde{B}_{m+j}(\mathbf{s})\cdots \tilde{B}_{m+1}(\mathbf{s})\tilde{B}_m(\mathbf{s})(I\mathbbm{1}) \\
    &=\lim_{j\rightarrow +\infty}\tilde{B}_{m+j}(\mathbf{s})\cdots \tilde{B}_{m+1}(\mathbf{s})\tilde{B}_m(\mathbf{s})\left(\sum_{l=0}^{\bar{l}}\mathbf{e}_l\right)  \\
    &=\sum_{l=0}^{\bar{l}}\lim_{j\rightarrow +\infty}\tilde{B}_{m+j}(\mathbf{s})\cdots \tilde{B}_{m+1}(\mathbf{s})\tilde{B}_m(\mathbf{s})\mathbf{e}_l  \\
    &=\sum_{l=0}^{\bar{l}}\mathbf{\tilde{g}}_l^m(\mathbf{s})\mathbbm{1}.
\end{aligned}
\]
Therefore, the generating functions satisfy $\sum_{l=0}^{\bar{l}}\mathbf{\tilde{g}}_l^m(\mathbf{s})=I$. Obviously, these generating functions are no longer scalar valued but matrix valued.
Representing the control points as $Q_m=\sum_{l=0}^{\bar{l}}\mathbf{e}_l\mathbf{q}_l^m$, the PN subdivision ring is obtained as
$\mathbf{\tilde{x}}_m=\sum_{l=0}^{\bar{l}}\mathbf{\tilde{g}}_l^m(\mathbf{s})\mathbf{q}_l^m=\tilde{G}_m(\mathbf{s})Q_m$, where  $\tilde{G}_m(\mathbf{s})=(\mathbf{\tilde{g}}_0^m(\mathbf{s}),\mathbf{\tilde{g}}_1^m(\mathbf{s}),\ldots,\mathbf{\tilde{g}}_{\bar{l}}^m(\mathbf{s}))$.

\begin{figure}[htb]
  \centering
  \subfigure[]{\includegraphics[width=0.255\linewidth]{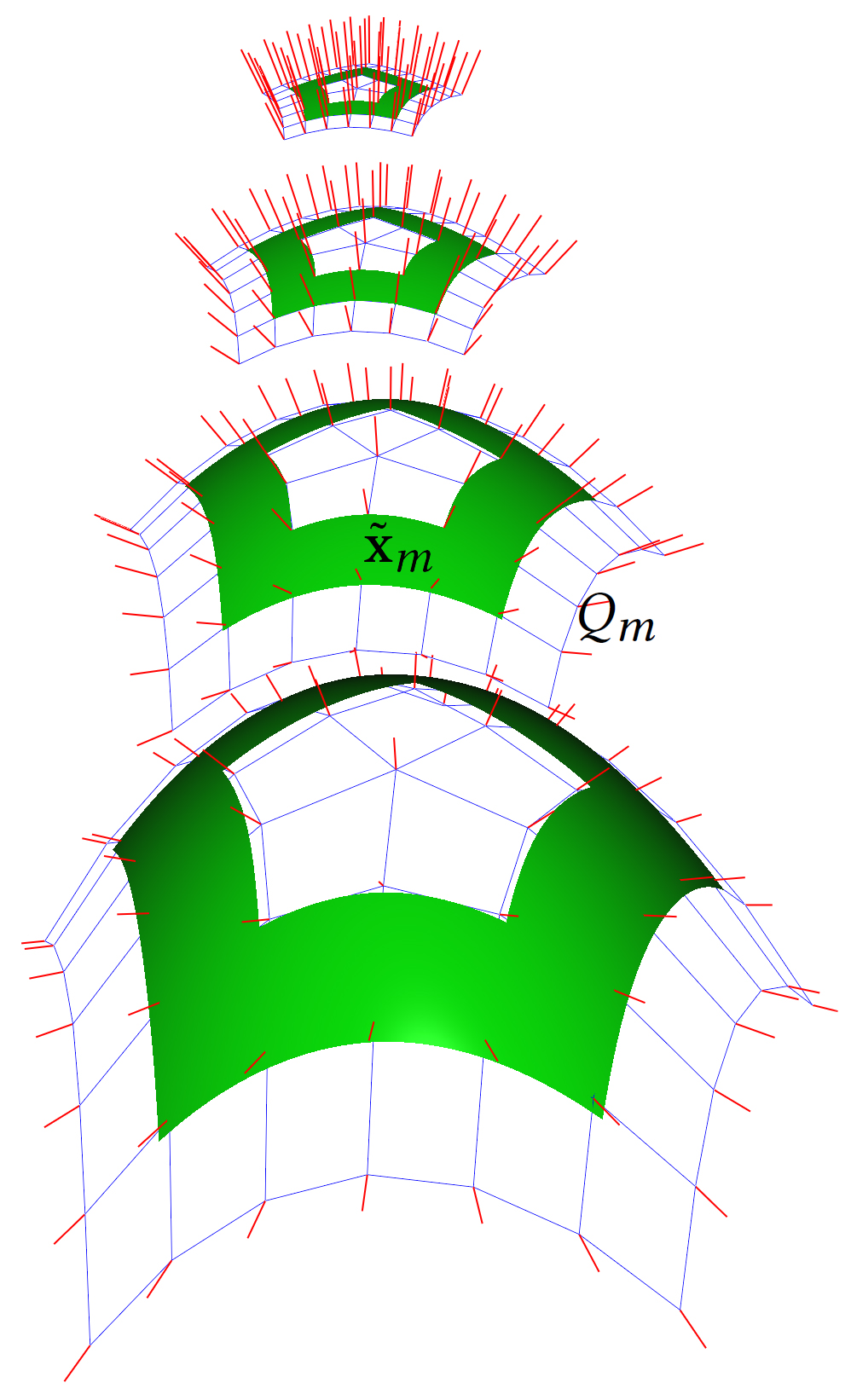}} \ \ \ \ \ \
  \subfigure[]{\includegraphics[width=0.255\linewidth]{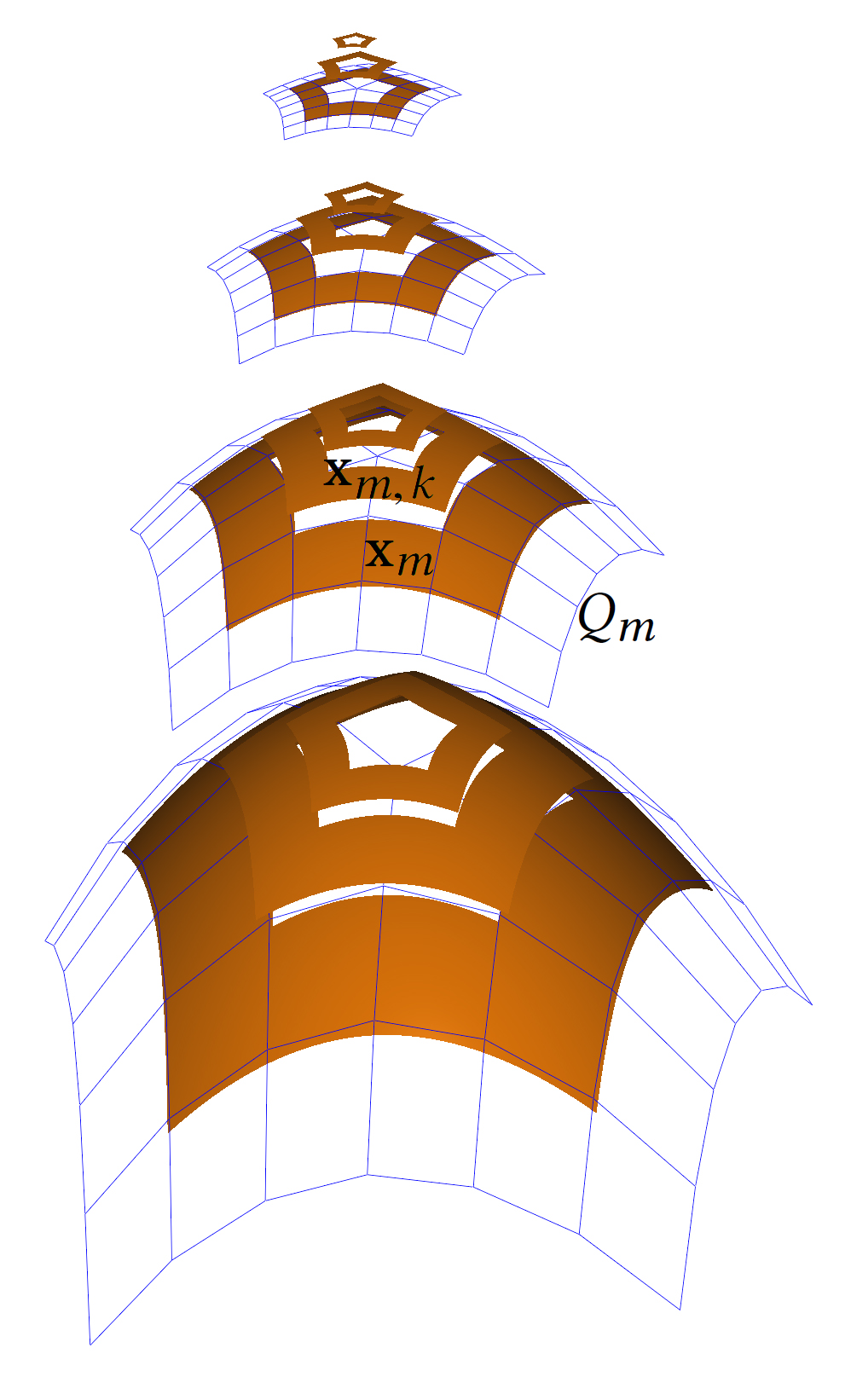}}
  \caption{(a) The sequence of surface rings around an extraordinary vertex and their control points and control normals obtained by PN subdivision; (b) the sequence of ring sequences by linear subdivision from the meshes computed by PN subdivision. The surface rings are shifted for clarity. }
  \label{Fig:PNsubd_mesh_rings}
\end{figure}

Before proving that the sequence of surface rings $\mathbf{\tilde{x}}_m$ converge to a limit point, we show that all block matrices $\tilde{M}_{ij}^k$ within Equation (\ref{Eqn:M_k square}) converge to $I$, which implies that $\tilde{M}_k$ converge to $E$, when the normal vectors $\mathbf{n}_i^k$ converge.
\begin{lemma}
\label{lemma:norm for M_ij^k - I}
Assume $S$ is a standard subdivision scheme and normal vectors $\mathbf{n}_i^k$ are refined by Equation (\ref{Eqn:refinement of normals at extra vertex}). Let $\tilde{M}_{ij}^k$ be the matrices as defined in Equation (\ref{Eqn:M_k square}). Then
\[
\|\tilde{M}_{ij}^k-I\|_\infty \leq K_M\gamma^k,
\]
where $K_M$ and $\gamma\in(0,1)$ are constants.
\end{lemma}
\begin{proof}
Based on Proposition \ref{Proposition:distancebetweenrefinednormals} we know that the subdivided normals satisfy
\[
\|\mathbf{n}_j^k - \mathbf{n}_l^k\|\leq c_2\gamma^k, \ \ \ j,l\in\{0,1,\ldots,\bar{l}\}
\]
where $c_2$ and $\gamma\in(0,1)$ are constants.
In the same way as the proof of Theorem \ref{Theoem:convergence univar_PN subd} we have
$\|A_{ij}^k - A_{il}^k\|_\infty \leq c_A\gamma^k$.
As $\tilde{M}_{ij}^k = I + \sum_{l=0}^{\bar{l}}s_{il}(A_{ij}^k-A_{il}^k)$, it follows that
\[
\|\tilde{M}_{ij}^k - I\|_\infty \leq c_Sc_A \gamma^k,
\]
where $c_S=\|S\|_{\infty}=\max_{0\leq i\leq\bar{l}}\sum_{l=0}^{\bar{l}} |s_{il}|$. The lemma is proven by choosing $K_M=c_Sc_A$.
\end{proof}

Now, we show that the PN subdivision scheme $\{S_j\}$ is stable and convergent and the obtained surface is $C^0$ continuous at isolated extraordinary points.
\begin{theorem}
\label{Theorem:stable+convergent for {S_j}}
Assume $S$ is a standard subdivision scheme. Assume $Q_m$ are the subdivided points and $S_j$, $j=0,1,\ldots$ are the subdivision matrices as defined in Equation (\ref{Eqn:Q_m = Sm...S1Q}). Then the PN subdivision scheme $\{S_j\}$ is stable and convergent.
\end{theorem}
\begin{proof}
We prove the stability and convergence of $\{S_j\}$ by comparing with the stationary subdivision scheme $S$.
Based on Lemma \ref{lemma:norm for M_ij^k - I}, we have
\[
\|S_j-S\circ E\|_\infty = \|S\circ(\tilde{M}_j-E)\|_\infty \leq \|S\|_\infty K_M\gamma^j
\]
for $j\in \mathbb{Z}_+$. It follows that
\[
\sum_{j\in \mathbb{Z}_+}\|S_j-S\circ E\|_\infty < +\infty,
\]
which implies $\{S_j\}\approx S$. Based on Theorem 6 in \citep{DynLevin1995AsymptoticallyEquivalent}, we conclude that the PN subdivision scheme $\{S_j\}$ is convergent and stable.
\end{proof}

\begin{theorem}
\label{Theorem:C0 PN subdivision surface}
Assume $S$ is a standard subdivision scheme and $\mathbf{\tilde{x}}_m$ are PN subdivision rings with control points $Q_m$ which are computed by Equation (\ref{Eqn:Q_m = Sm...S1Q}) and control normals $\mathbf{n}_l^m$, $l=0,1,\ldots,\bar{l}$. Then $\mathbf{\tilde{x}}_m$ converge to a point as $m$ approaches infinity.
\end{theorem}
\begin{proof}
Based on Theorem \ref{Theorem:stable+convergent for {S_j}} and Equation (\ref{Eqn:Q_m = Sm...S1Q}) we know that the mesh sequence $\{Q_m\}_{m=1}^{+\infty}$ converges. It follows that $Q_m$, $m=1,2,\ldots$, are bounded. To prove the theorem, we first prove that the mesh sequence converges to a central point, then we show that the PN subdivision rings $\mathbf{\tilde{x}}_m$ also converge to the central point. Assume $\omega_0^\top$ is the left eigenvector of $S$.
Using Equation (\ref{Eqn:S^m Q ->p_01}), we have
\[
\lim_{k\rightarrow+\infty}S^k Q_m = \mathbf{p}_{m,0}\mathbbm{1}
\]
and
\[
\lim_{k\rightarrow+\infty}S^k Q_{m-1} = \lim_{k\rightarrow+\infty}S^{k-1} SQ_{m-1} = \mathbf{p}_{m-1,0}\mathbbm{1},
\]
where $\mathbf{p}_{m,0}=\omega_0^\top Q_m$ and $\mathbf{p}_{m-1,0}=\omega_0^\top SQ_{m-1}$.
By Lemma \ref{lemma:norm for M_ij^k - I} and because the sequence $\{Q_m\}_{m=1}^\infty$ are bounded, we have
\[
\begin{aligned}
    \|\mathbf{p}_{m,0} - \mathbf{p}_{m-1,0}\| &= \|\omega_0^\top Q_m - \omega_0^\top SQ_{m-1}\|  \\
    &= \|\omega_0^\top (S_{m-1}-S)Q_{m-1}\|  \\
    &= \|\omega_0^\top S\circ(\tilde{M}_{m-1}-E)Q_{m-1}\|   \\
    &\leq \|\omega_0^\top\|_1 \|S\|_{\infty} \|Q_{m-1}\|_{\infty} K_M \gamma^{m-1}    \\
    &\leq K_p \gamma^{m-1},
\end{aligned}
\]
where $K_p$ and $\gamma\in(0,1)$ are constants with $\|\omega_0^\top\|_1$ the $l_1$ norm of the eigenvector. This implies that $\{\mathbf{p}_{m,0}\}_{m=1}^\infty$ is a Cauchy sequence.
Therefore, we have
\[
\lim_{m\to +\infty}\mathbf{p}_{m,0}=\mathbf{p}_c.
\]
To prove the surface rings $\mathbf{\tilde{x}}_m(\mathbf{s})$ converge to $\mathbf{p}_c$, we prove all points within mesh $Q_m$ converge to $\mathbf{p}_c$. We write
\[
\begin{aligned}
    Q_{m+k}-S^k Q_m &= (S_{m+k-1}\cdots S_m-S^k)Q_m \\
    &=\sum_{j=0}^{k-1} S_{m+k-1}\cdots S_{m+j+1}(S_{m+j}-S)S^jQ_m \\
    &=\sum_{j=0}^{k-1} S_{m+k-1}\cdots S_{m+j+1}(S\circ(\tilde{M}_{m+j}-E))S^jQ_m.
\end{aligned}
\]
By applying Lemma \ref{lemma:norm for M_ij^k - I} and because $\{S_j\}$ is stable, we have
\begin{equation}
\label{Eqn:Q_m+l-S^lQ_m}
\|Q_{m+k}-S^k Q_m\|_\infty \leq K_q \sum_{j=0}^{k-1} \gamma^{m+j} \leq \frac{K_q}{1-\gamma}\gamma^m,
\end{equation}
where $\gamma\in(0,1)$.
Based on the identity
\[
Q_{m+k} -  \mathbf{p}_c\mathbbm{1} = (Q_{m+k}-S^kQ_m) + (S^kQ_m-\mathbf{p}_{m,0}\mathbbm{1}) + (\mathbf{p}_{m,0}\mathbbm{1}-\mathbf{p}_c\mathbbm{1})
\]
as well as the definitions of $\mathbf{p}_{m,0}$ and $\mathbf{p}_c$, we have
\[
\lim_{m\rightarrow+\infty \atop k\to+\infty}{Q_{m+k}}= \mathbf{p}_c\mathbbm{1}.
\]
Since $\lim_{m\rightarrow+\infty}Q_m=\mathbf{p}_c\mathbbm{1}$, and because the generating functions of $\mathbf{\tilde{x}}_m(\mathbf{s})$ sum up to $I$, it yields that
\[
\lim_{m\rightarrow+\infty}\mathbf{\tilde{x}}_m(\mathbf{s})= \lim_{m\rightarrow+\infty}\tilde{G}_m(\mathbf{s})Q_m = \mathbf{p}_c.
\]
This proves the theorem.
\end{proof}

Besides being $C^0$ continuous, the PN subdivision surfaces can also be $C^1$ continuous at the extraordinary points. We prove that the normals of the sequence of surface rings $\mathbf{\tilde{x}}_m(\mathbf{s})$ by PN subdivision converge by comparing with a sequence of surface rings obtained by linear subdivision using the same set of control nets. We present a lemma before proving the theorem for $C^1$ continuity.

\begin{lemma}
\label{lemma:normal continuity}
Assume $S$ is a standard subdivision scheme and the characteristic map $\Psi$ is regular. Assume $Q_m$ be the control points given by Equation (\ref{Eqn:Q_m = Sm...S1Q}) and $G$ is the vector of scalar valued generating functions. Then the normals of surface rings $\mathbf{x}_m=GQ_m$ converge for almost all initial control nets.
\end{lemma}
\begin{proof}
We first show that a limit vector exists and then we show that the normals of surface rings $\mathbf{x}_m(\mathbf{s})$ converge to the limit vector.

Let $\mathbf{x}_{m,k}=GS^kQ_m$. See the surfaces illustrated in Figure~\ref{Fig:PNsubd_mesh_rings}(b) for reference. Similar to Equation (\ref{Eqn:x^m expanded}), we have
\[
\mathbf{x}_{m,k} \cong \mathbf{p}_{m,0} + \lambda^k\Psi(\mathbf{p}_{m,1}; \mathbf{p}_{m,2}),
\]
where $\lambda$ is the second large eigenvalue with multiplicity 2, $\Psi$ is the characteristic map and $\mathbf{p}_{m,i} = \omega_i^\top Q_m$, $i=0,1,2$, with $\omega_i^\top$ the left eigenvector of the matrix $S$.
Let $\mathbf{n}_{m,k}$ be the normal vector of the surface $\mathbf{x}_{m,k}$. Under the assumption that the characteristic map $\Psi$ is regular, by Proposition \ref{Propostion:normalcontinuityforsteadysubdivision}, we have
\[
\mathbf{n}_m^c := \lim_{k\to+\infty}\mathbf{n}_{m,k} = sign({^\times}D\Psi)\frac{\mathbf{p}_{m,1}\times \mathbf{p}_{m,2}}{\|\mathbf{p}_{m,1}\times \mathbf{p}_{m,2}\|}.
\]
We show the central normal sequence $\{\mathbf{n}_m^c\}_{m=1}^\infty$ converges to a limit vector. Similar to the asymptotic expansion of $\mathbf{x}_{m,k}$, by expanding $\mathbf{x}_{m-1,k}=GS^kQ_{m-1}=GS^{k-1}SQ_{m-1}$, we have $\mathbf{p}_{m-1,i} = \omega_i^\top Q_{m-1} = \omega_i^\top SQ_{m-1}$, $i=1,2$.
By the same reason for $\{\mathbf{p}_{m,0}\}_{m=1}^\infty$ within the proof of Theorem \ref{Theorem:C0 PN subdivision surface},
we know that $\{\mathbf{p}_{m,i}\}_{m=1}^\infty$, $i=1,2$, are also Cauchy sequences.
Therefore, we have
\[
\lim_{m\to +\infty}\mathbf{p}_{m,i}=\mathbf{t}_i, \ \ \ \ i=1,2.
\]
It follows that
\[
\mathbf{n}_c := \lim_{m\to+\infty}\mathbf{n}_m^c = sign({^\times}D\Psi)\frac{\mathbf{t}_1\times \mathbf{t}_2}{\|\mathbf{t}_1\times \mathbf{t}_2\|}.
\]

Let $\mathbf{n}_m$ be the normal vector of surface ring $\mathbf{x}_m=GQ_m$. We prove that the normal vectors $\mathbf{n}_m$ converge to $\mathbf{n}_c$.
Based on Equation (\ref{Eqn:Q_m+l-S^lQ_m}), we know that the surface difference
\[
\mathbf{x}_{m+k}(\mathbf{s}) - \mathbf{x}_{m,k}(\mathbf{s})=G(2^{m+k}\mathbf{s})(Q_{m+k}-S^kQ_m)
\]
as well as the differences between partial derivatives of the two surfaces $\mathbf{x}_{m+k}$ and $\mathbf{x}_{m,k}$ approach zero when $m$ goes to infinity.
By direct computation of normals for the two surfaces, we have
\[
\lim_{m\to +\infty} (\mathbf{n}_{m+k} - \mathbf{n}_{m,k}) = \mathbf{0}.
\]
Based on the identity
\[
\mathbf{n}_{m+k} -  \mathbf{n}_c = (\mathbf{n}_{m+k}-\mathbf{n}_{m,k}) + (\mathbf{n}_{m,k}-\mathbf{n}_m^c) + (\mathbf{n}_m^c-\mathbf{n}_c)
\]
as well as the definitions of $\mathbf{n}_m^c$ and $\mathbf{n}_c$, we have
\[
\lim_{m\rightarrow+\infty \atop k\to+\infty}{\mathbf{n}_{m+k}}= \mathbf{n}_c.
\]
This completes the proof.
\end{proof}

\begin{theorem}
\label{Theorem:C1 PN subdivision surface}
Assume $S$ is a standard subdivision scheme and the characteristic map $\Psi$ is regular. If the control normals at the mesh vertices are refined by Equation (\ref{Eqn:refinement of normals at extra vertex}), then the PN subdivision surface is normal continuous at the extraordinary point for almost all initial control nets.
\end{theorem}
\begin{proof}
Assume $Q_m$ are the control points computed by Equation (\ref{Eqn:Q_m = Sm...S1Q}) and $N_m$ are the control normals at the control points. Let $\mathbf{n}_m(\mathbf{s})$ and $\mathbf{\tilde{n}}_m(\mathbf{s})$, $\mathbf{s}\in\mathbf{S}_n^m$, be the unit normals of surface rings $\mathbf{x}_m(\mathbf{s})$, $\mathbf{\tilde{x}}_m(\mathbf{s})$ that are generated from the control points and control normals by linear subdivision or PN subdivision, respectively. We prove the theorem by showing that the normals $\mathbf{\tilde{n}}_m(\mathbf{s})$ and $\mathbf{n}_m(\mathbf{s})$ converge to the same limit vector when $m$ goes to infinity.

\begin{figure}[hb]
  \centering
  \subfigure[]{\includegraphics[width=0.25\linewidth]{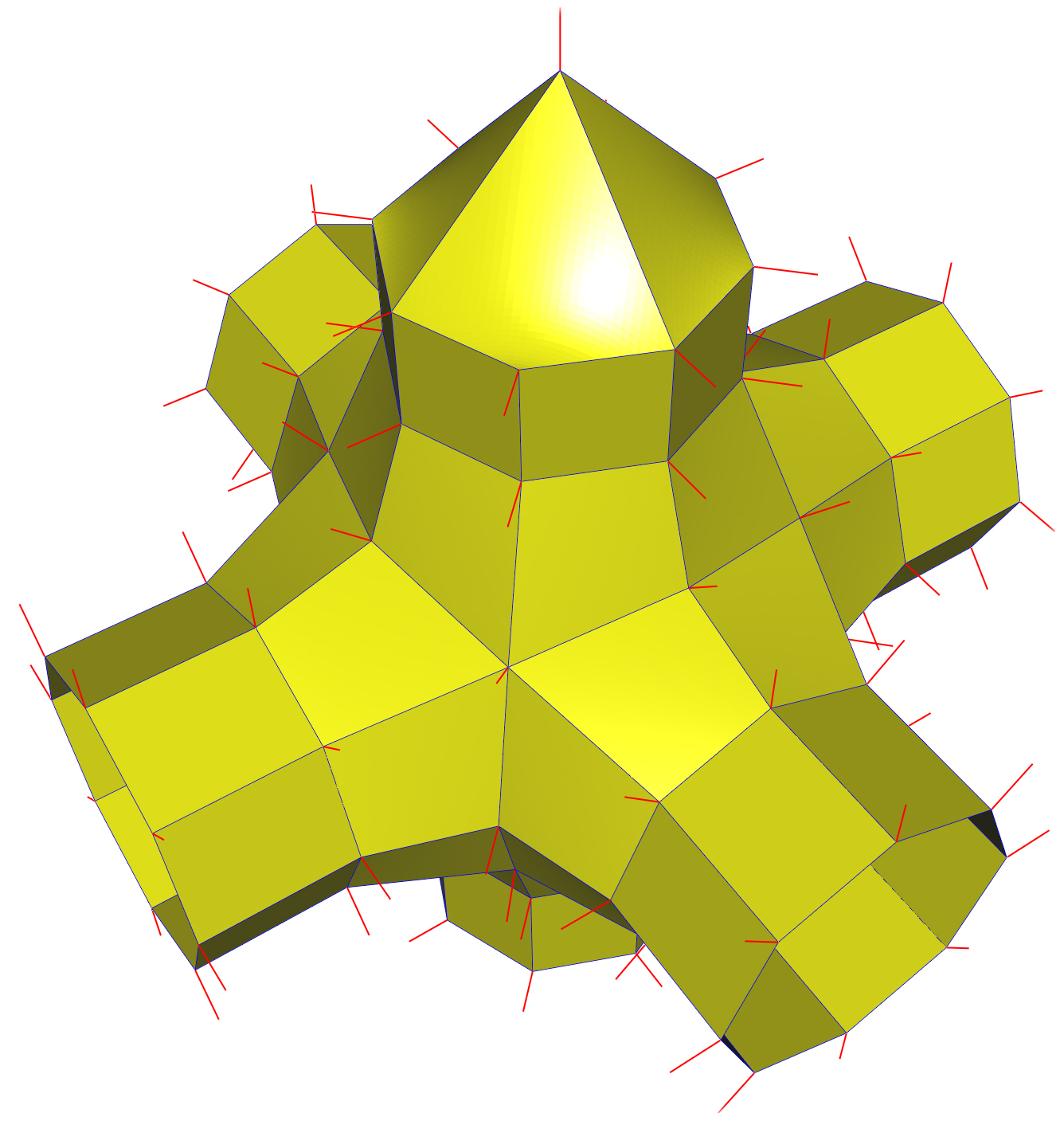}} \ \ \ \ \ \
  \subfigure[]{\includegraphics[width=0.25\linewidth]{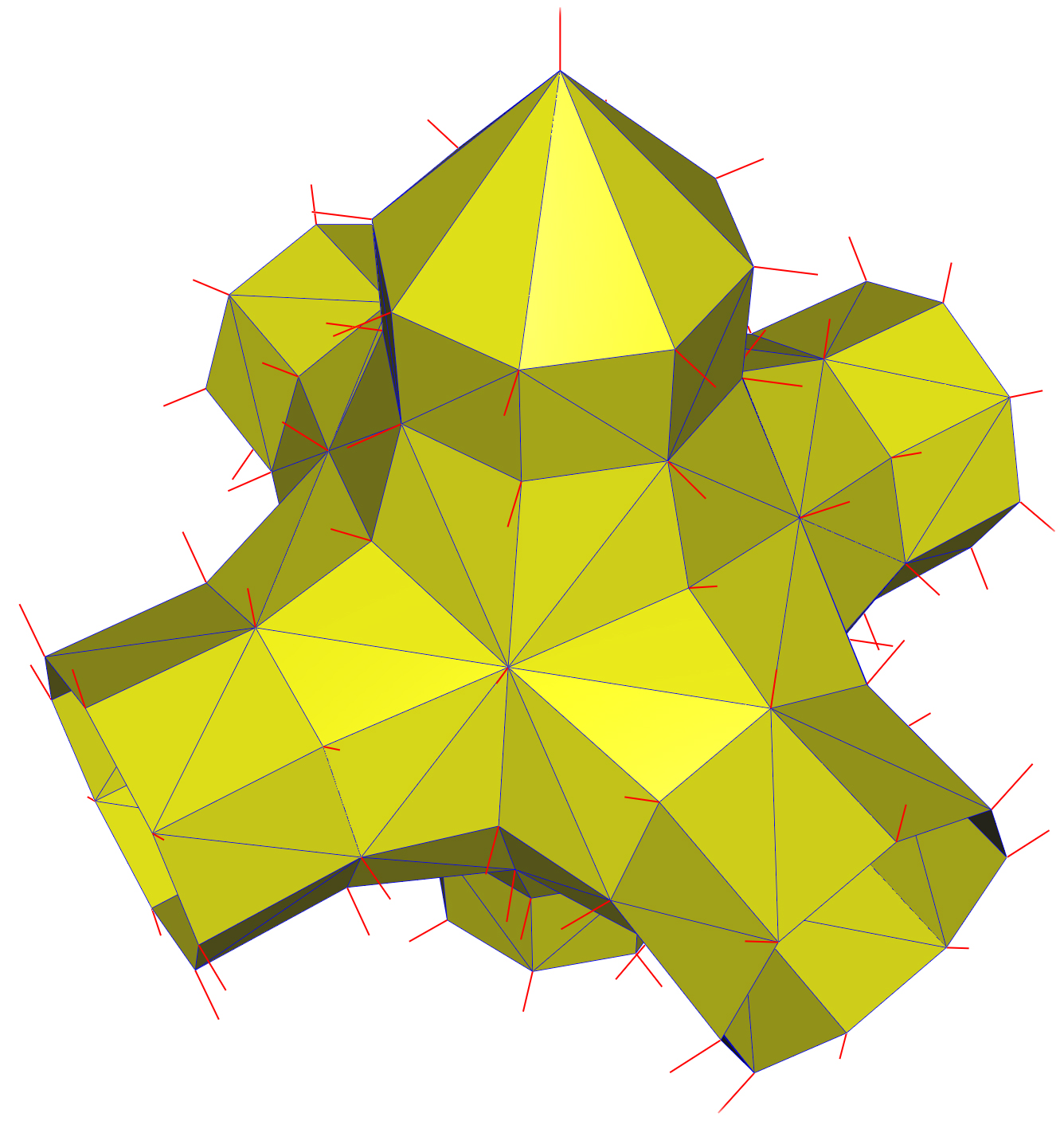}}
  \caption{A quad mesh and its triangulation together with pre-computed unit normal vectors at the vertices.}
  \label{Fig:Six+ Control mesh}
\end{figure}

Let $B_{m+j}(\mathbf{s})$ and $\tilde{B}_{m+j}(\mathbf{s})=B_{m+j}(\mathbf{s})\circ \tilde{M}_{m+j}(\mathbf{s})$, $j=0,1,\ldots$, be the subdivision matrices given in Equation (\ref{Eqn:x_m(s)1}) and Equation (\ref{Eqn:x_tilde m(s)1}). By the same technique as Lemma \ref{lemma:norm for M_ij^k - I} and Lemma \ref{lemma:difference of matrices}, we have
\[
\begin{aligned}
\|\tilde{M}_{m+j}(\mathbf{s})-E\|_\infty &\leq \tilde{c} \gamma^{m+j},  \ \ \ \
\left\|\frac{\partial \tilde{M}_{m+j}(\mathbf{s})}{\partial u}\right\|_\infty &\leq \tilde{c}_{u} \gamma^{m+j}, \ \ j=0,1,\ldots
\end{aligned}
\]
where $\tilde{c}$, $\tilde{c}_{u}$ and $\gamma\in(0,1)$ are constants. It follows that
\[
\tilde{B}_{m+j}(\mathbf{s}) = B_{m+j}(\mathbf{s})\circ(E+\tilde{M}_{m+j}(\mathbf{s})-E) = B_{m+j}(\mathbf{s})\circ E + O(\gamma^{m+j})
\]
and
\[
\begin{aligned}
\frac{\partial \tilde{B}_{m+j}(\mathbf{s})}{\partial u} &= \frac{\partial B_{m+j}(\mathbf{s})}{\partial u}\circ \tilde{M}_{m+j}(\mathbf{s})
        + B_{m+j}(\mathbf{s}) \circ \frac{\partial \tilde{M}_{m+j}(\mathbf{s})}{\partial u} \\
        &= \frac{\partial B_{m+j}(\mathbf{s})}{\partial u}\circ(E + O(\gamma^{m+j})) + B_{m+j}(\mathbf{s}) \circ \frac{\partial \tilde{M}_{m+j}(\mathbf{s})}{\partial u} \\
        &= \frac{\partial B_{m+j}(\mathbf{s})}{\partial u}\circ E + O(\gamma^{m+j}).
\end{aligned}
\]
By substituting above two equalities, we compute the partial derivatives of $\mathbf{\tilde{x}}_m(\mathbf{s})$ as follows:
\[
\begin{aligned}
& \frac{\partial}{\partial u}(\tilde{B}_{m+j}(\mathbf{s})\cdots \tilde{B}_m(\mathbf{s})Q_m)   \\
& = \sum_{l=0}^j \tilde{B}_{m+j}(\mathbf{s})\cdots \frac{\partial \tilde{B}_{m+l}(\mathbf{s})}{\partial u} \cdots \tilde{B}_m(\mathbf{s})Q_m \\
& = \sum_{l=0}^j ((B_{m+j}(\mathbf{s})\cdots \frac{\partial B_{m+l}(\mathbf{s})}{\partial u} \cdots B_m(\mathbf{s}))\circ E) Q_m + O(\gamma^m)  \\
& = \frac{\partial}{\partial u}(B_{m+j}(\mathbf{s})\cdots B_m(\mathbf{s})Q_m) + O(\gamma^m).
\end{aligned}
\]
When $j$ goes to infinity, we have
\[
\frac{\partial \mathbf{\tilde{x}}_m(\mathbf{s})}{\partial u} = \frac{\partial \mathbf{x}_m(\mathbf{s})}{\partial u} + O(\gamma^m).
\]
Similarly, we have
\[
\frac{\partial \mathbf{\tilde{x}}_m(\mathbf{s})}{\partial v} = \frac{\partial \mathbf{x}_m(\mathbf{s})}{\partial v} + O(\gamma^m).
\]
Since $\mathbf{\tilde{n}}_m(\mathbf{s}) // \frac{\partial \mathbf{\tilde{x}}_m(\mathbf{s})}{\partial u} \times \frac{\partial \mathbf{\tilde{x}}_m(\mathbf{s})}{\partial v}$
and $\mathbf{n}_m(\mathbf{s}) // \frac{\partial \mathbf{x}_m(\mathbf{s})}{\partial u} \times \frac{\partial \mathbf{x}_m(\mathbf{s})}{\partial v}$, we have
\[
\lim_{m\rightarrow+\infty}  (\mathbf{\tilde{n}}_m(\mathbf{s}) - \mathbf{n}_m(\mathbf{s})) = \mathbf{0}.
\]
By applying the result of Lemma \ref{lemma:normal continuity}, we have
\[
\lim_{m\rightarrow+\infty}{\mathbf{\tilde{n}}_m(\mathbf{s})} = \lim_{m\rightarrow+\infty}(\mathbf{\tilde{n}}_m(\mathbf{s})-\mathbf{n}_m(\mathbf{s}))
+\lim_{m\rightarrow+\infty}\mathbf{n}_m(\mathbf{s}) = \mathbf{n}_c.
\]
This proves the theorem.
\end{proof}

\begin{figure*}[htb]
  \centering
  \subfigure[PN-Catmull-Clark subdivision vs. Catmull-Clark subdivision]{\includegraphics[width=0.24\linewidth]{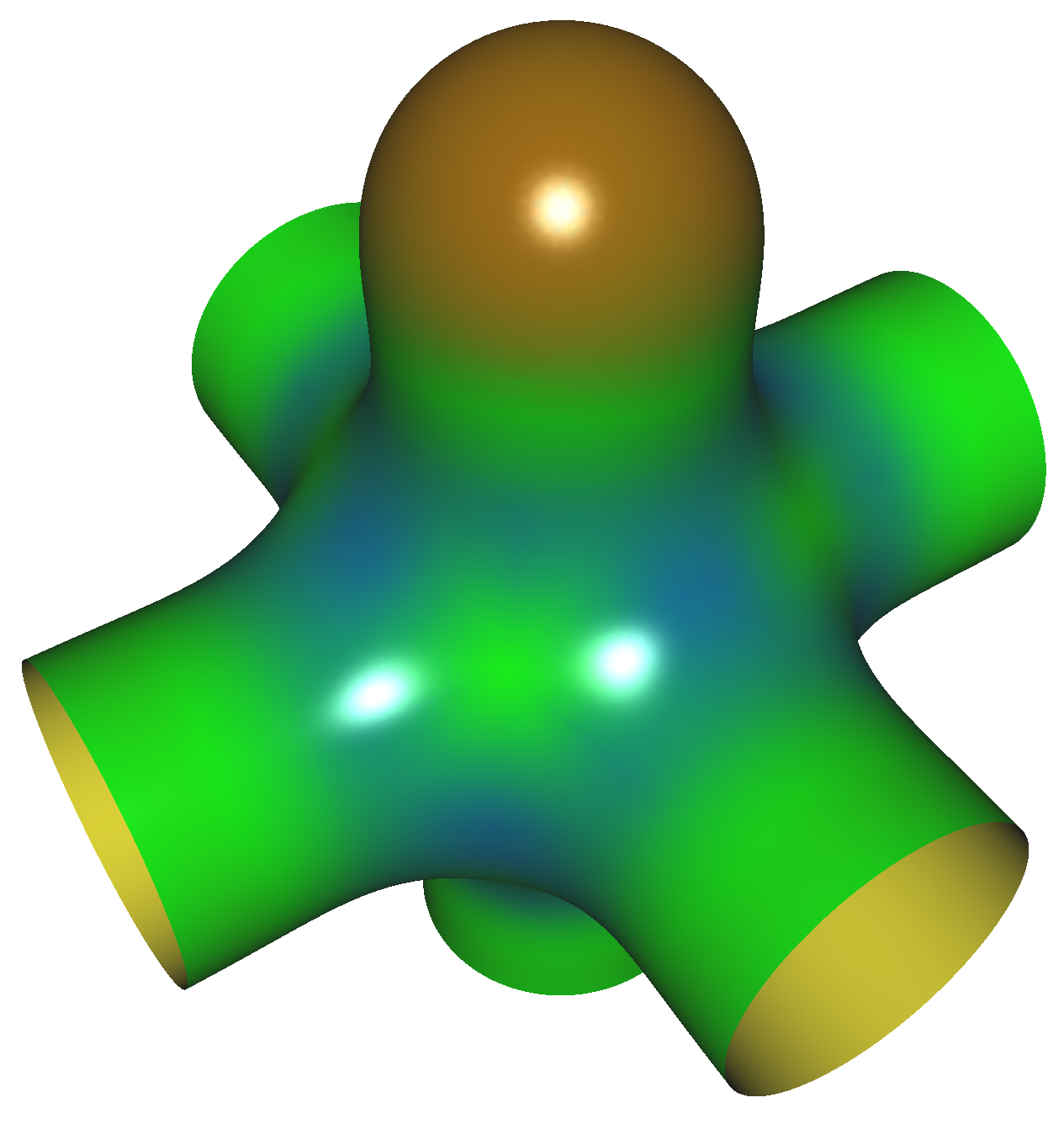}
  \includegraphics[width=0.245\linewidth]{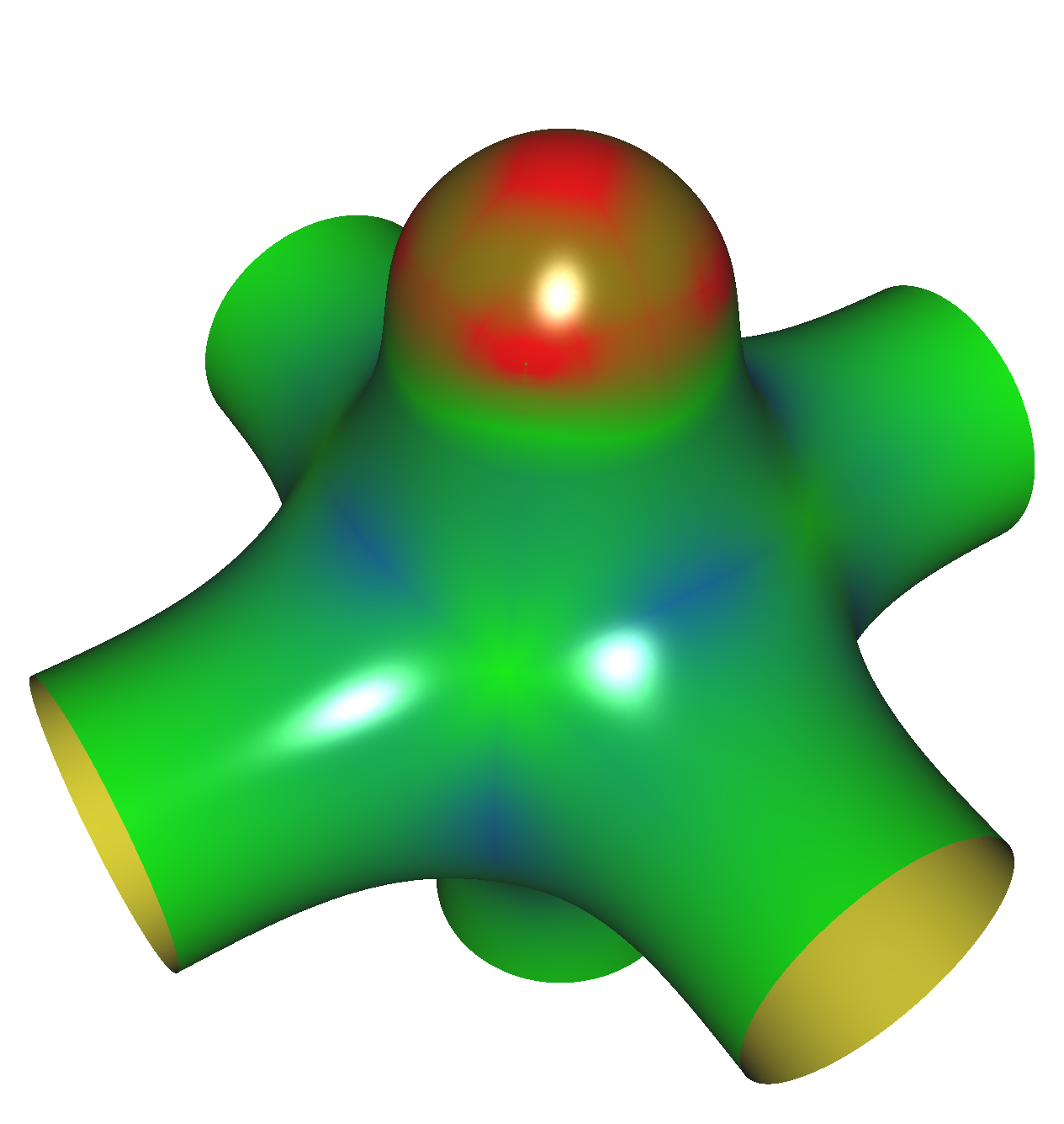}}
  \subfigure[PN-Doo-Sabin subdivision vs. Doo-Sabin subdivision]{\includegraphics[width=0.24\linewidth]{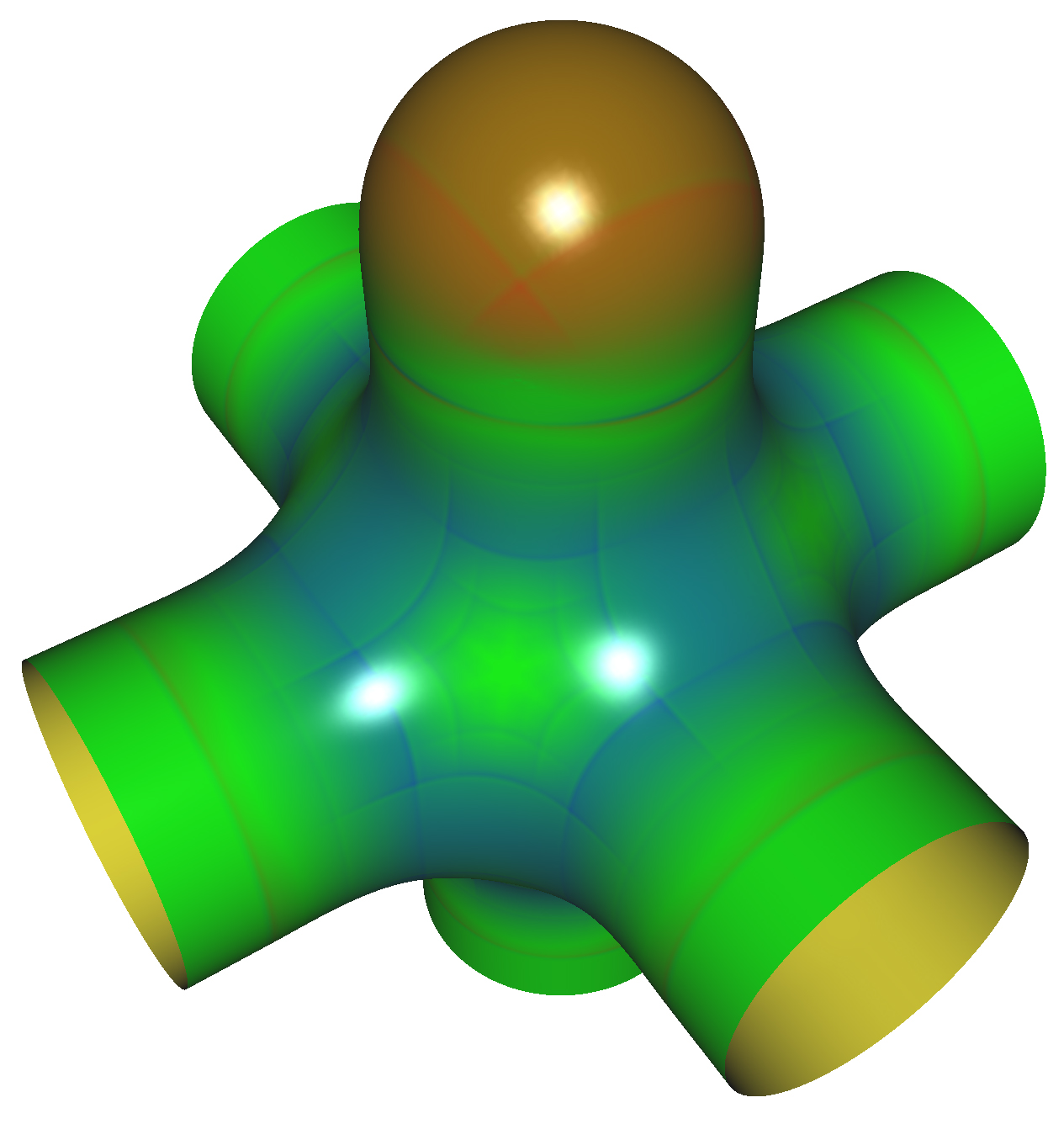}
  \includegraphics[width=0.245\linewidth]{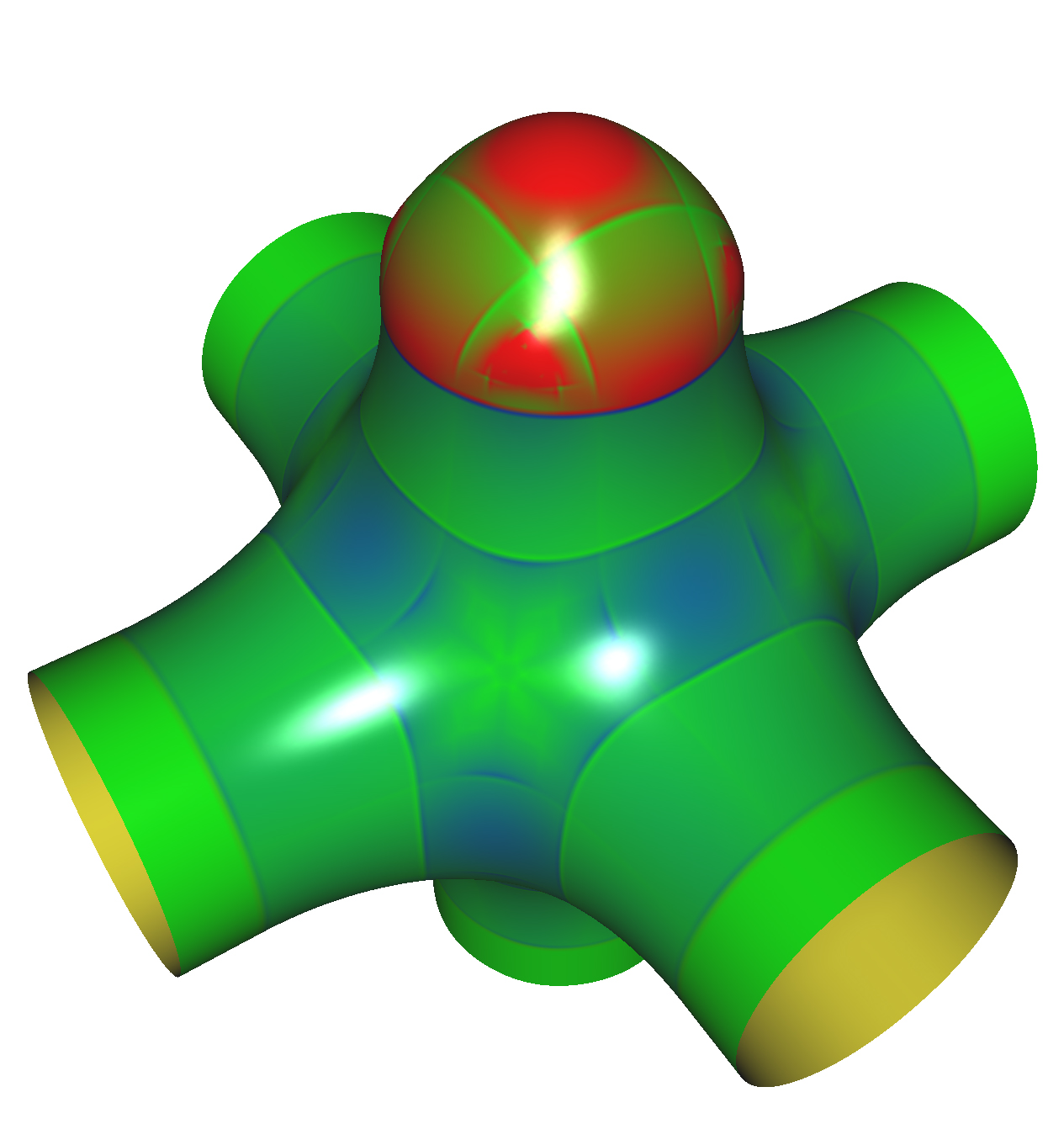}} \\
  \subfigure[PN-Kobbelt subdivision vs. Kobbelt subdivision]{\includegraphics[width=0.24\linewidth]{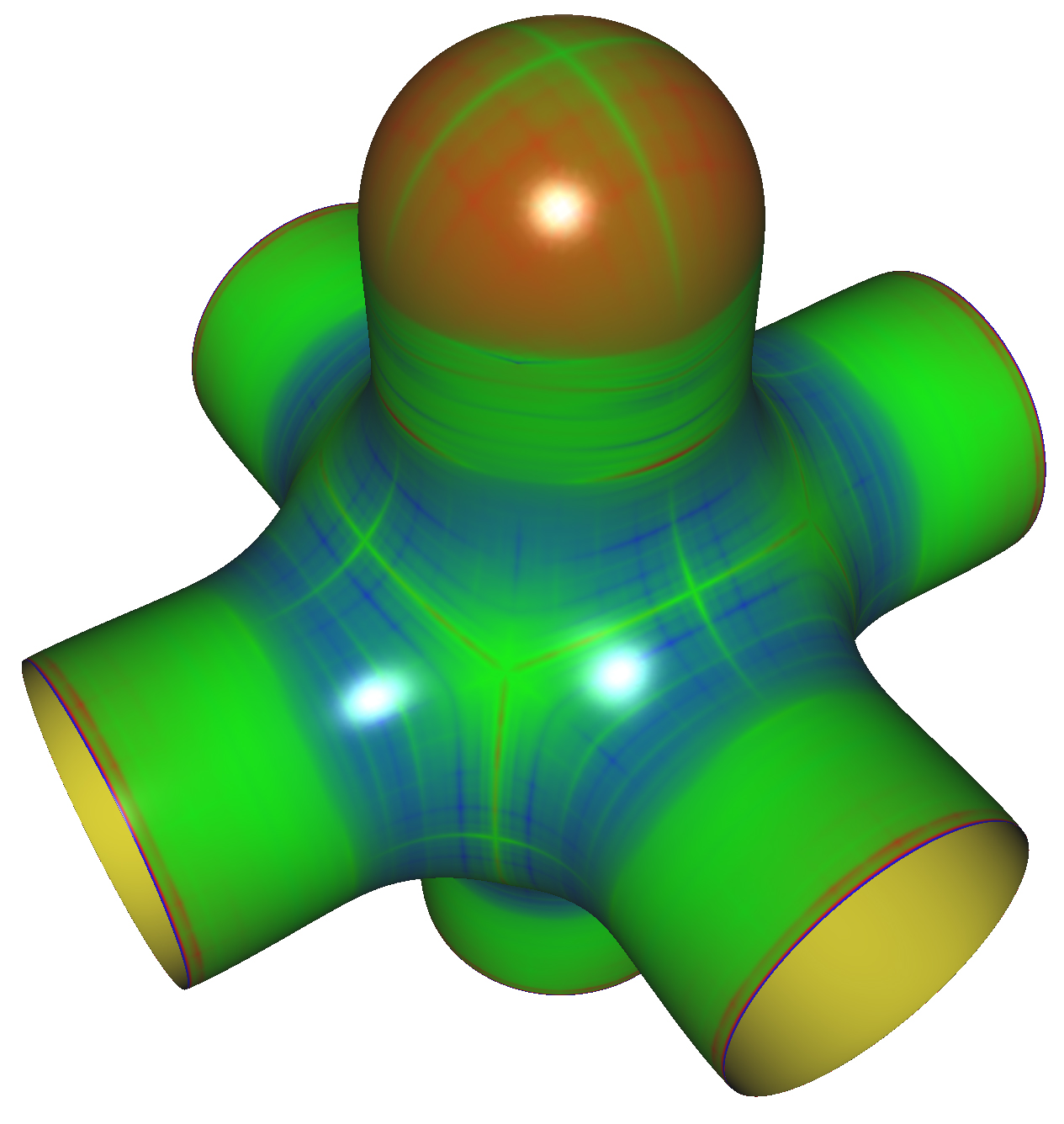}
  \includegraphics[width=0.245\linewidth]{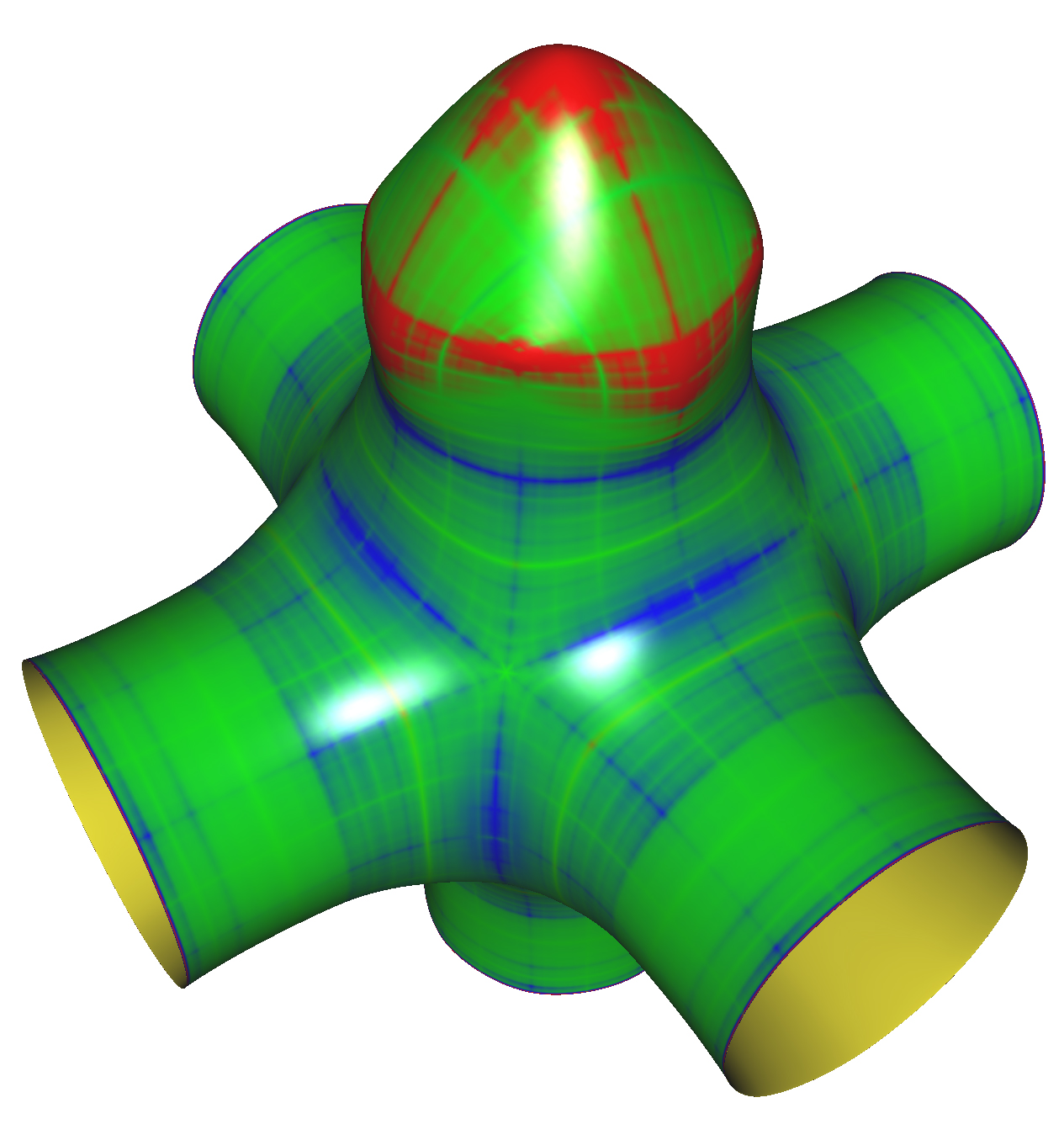}}
  \subfigure[PN-Loop subdivision vs. Loop subdivision]{\includegraphics[width=0.24\linewidth]{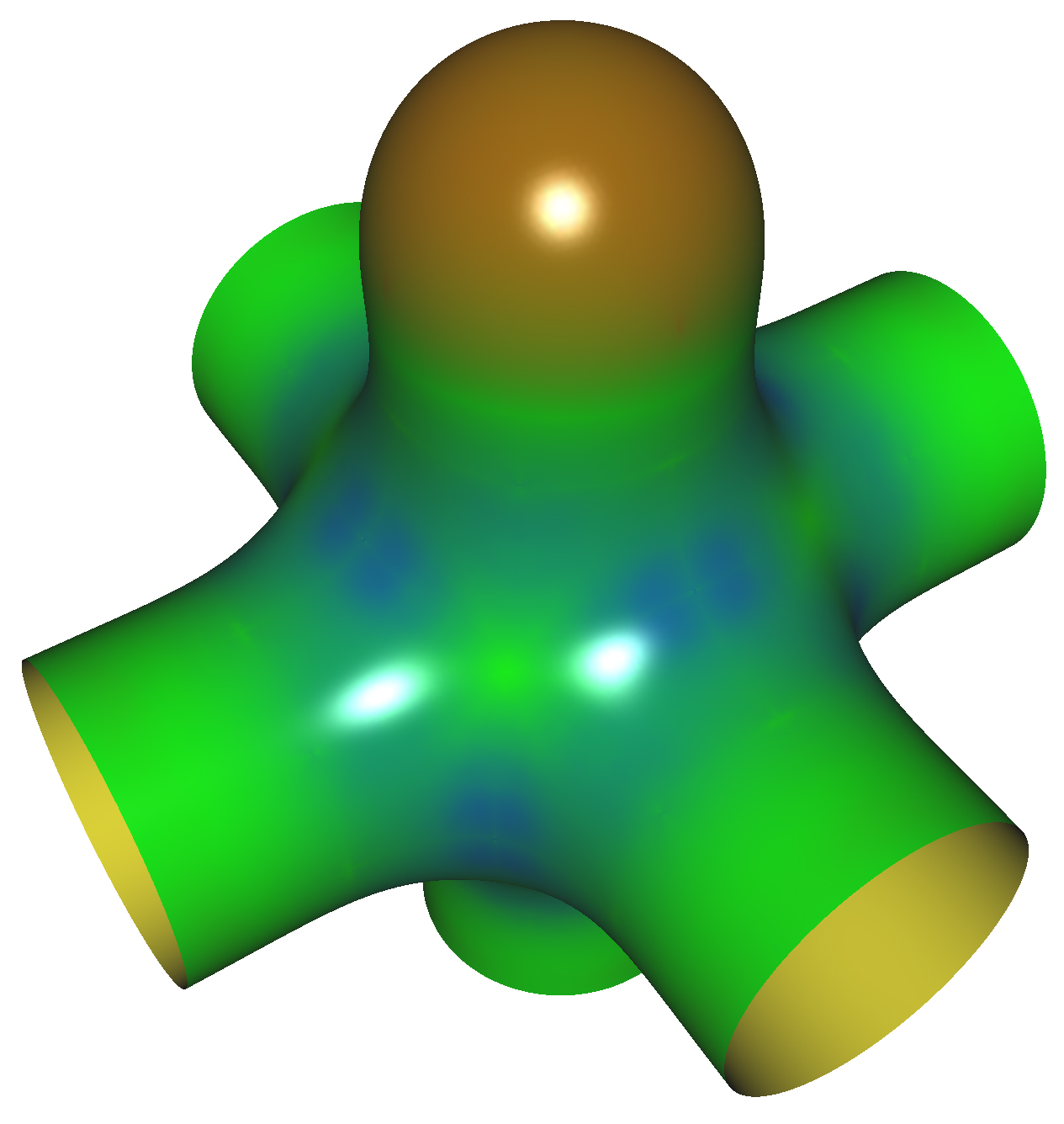}
  \includegraphics[width=0.245\linewidth]{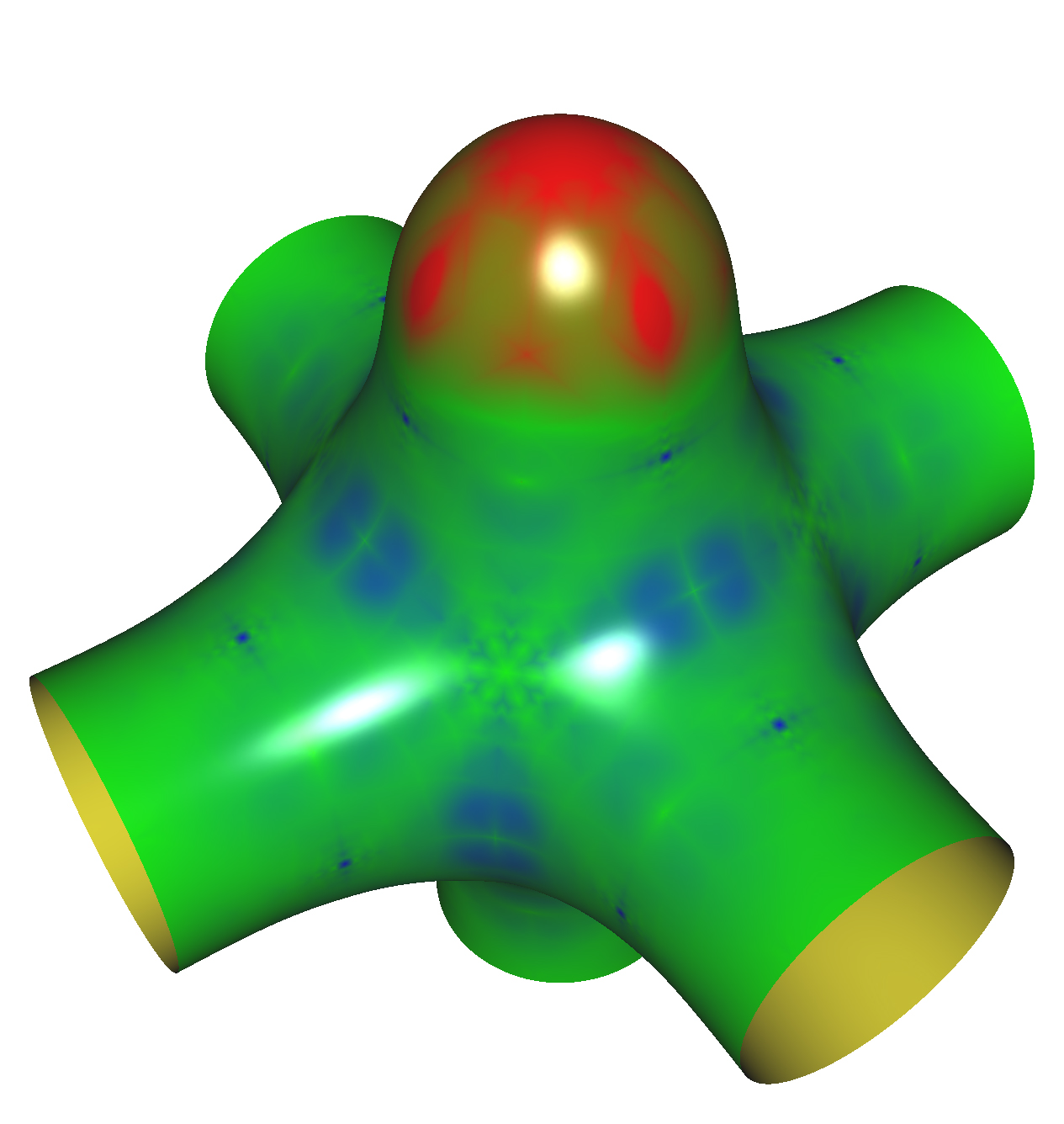}}
  \caption{Surface modeling by PN subdivision schemes or linear subdivision schemes. Gaussian curvatures of the subdivision surfaces change from (positive) high values through zero to (negative) low values when the colors change from red through green to blue.}
  \label{Fig:Six PN subdivision_Gauss}
\end{figure*}

Figure \ref{Fig:Six+ Control mesh} illustrates a quad mesh and its triangulation. The control normals at all control points are estimated from the input mesh. We subdivide the quad mesh by PN-Catmull-Clark, PN-Doo-Sabin or PN-Kobbelt subdivision schemes. A PN-Loop subdivision surface has been constructed from the triangulated mesh. For comparison purposes, the initial meshes are also subdivided by the corresponding linear subdivision schemes. To check the smoothness of all subdivision surfaces, the Gaussian curvatures of the surfaces have been computed. All surfaces illustrated in Figure \ref{Fig:Six PN subdivision_Gauss} are rendered by meshes after 5 iterations of subdivision. Particularly, the curvature plots are computed discretely by employing a high accuracy algorithm presented in \citep{YangZheng13:CurvatureTensor}. From the figures we see that the PN subdivision schemes and the linear subdivision schemes can achieve the same smoothness orders, over regular regions as well as regions near extraordinary points. Due to the properties of preserving circles, cylinders and spheres of the proposed subdivision schemes, the PN subdivision surfaces have exact circular boundaries, circular cylinder parts or approximate hemispheres on the top parts defined by the control points and control normals.

%%%%%%%%%%%%%%%%%%%%%%%%%%%%%%%%%%%%%%%%%%%%%%
%%%% Section 5
%%%%%%%%%%%%%%%%%%%%%%%%%%%%%%%%%%%%%%%%%%%%%%

\section{PN $C^2$ subdivision surfaces}
\label{Sec:PN-C2-subdivision surface}

In addition to generalizing linear subdivision surfaces that have $C^1$ continuity at the extraordinary points to PN subdivision surfaces,
we are also interested in generalizing modified Catmull-Clark subdivision \citep{Prautzsch1998G2subdivisionsurface} or modified Loop subdivision \citep{Prautzsch2000G2Loopsurface} to PN subdivision schemes. These two modified schemes are simple to implement and can generate $C^2$ subdivision surfaces with flat extraordinary points. It is found that the generalized PN $C^2$ subdivision surfaces are curvature continuous too but the extraordinary points can be no longer flat.

Assume $S$ is the subdivision matrix for control points surrounding an isolated extraordinary vertex within a control mesh using Catmull-Clark subdivision. To improve the smoothness order at the extraordinary point, \citet{Prautzsch1998G2subdivisionsurface} proposed to modify the Catmull-Clark subdivision scheme by tuning the eigenvalues of the subdivision matrix. Let $V$ be the matrix of which the columns represent the right eigenvectors of $S$, the subdivision matrix is decomposed into $S=V\Lambda V^{-1}$, where $\Lambda=\textrm{diag}(1,\lambda,\lambda,\mu,\ldots,\zeta)$ and $1>\lambda>|\mu|\geq \ldots \geq |\zeta|$ are the eigenvalues of the matrix. When the matrix $\Lambda$ has been changed into $\Lambda'=\textrm{diag}(1,\lambda,\lambda,\mu',\ldots,\zeta')$, a modified subdivision scheme is obtained by using stencils given in the modified subdivision matrix $S'=V\Lambda' V^{-1}$.
According to Proposition \ref{Proposition:C2 continuity at extraordinary point}, if the prescribed eigenvalues satisfy $|\mu'|<\lambda^2$, $\dots$, $|\zeta'|<\lambda^2$, the modified Catmull-Clark subdivision surface is $C^2$ continuous with vanishing principal curvatures at the extraordinary point. Similarly, conventional Loop subdivision can also be modified to produce $C^2$ subdivision surfaces with flat extraordinary points \citep{Prautzsch2000G2Loopsurface}.

Even though the subdivision surfaces obtained by the modified Catmull-Clark subdivision or the modified Loop subdivision are curvature continuous, they may suffer the unfairness or concentric undulations due to the restricted zero curvature at the extraordinary points. These restrictions make the modified subdivision schemes less practical in high quality surface modeling.

By utilizing control points together with control normals, we propose to construct high quality subdivision surfaces using PN modified $C^2$ subdivision schemes. We just explain the steps of PN modified Catmull-Clark subdivision, PN modified Loop subdivision can be implemented similarly. An arbitrary topology control mesh together with given or estimated control normals are first subdivided by PN-Catmull-Clark subdivision. From the second round of subdivision, all faces within the meshes are quadrangles. The positions and control normals at the refined vertices corresponding to old irregular vertices, their abutting edges or their abutting faces are computed by Equation (\ref{Eqn:PN binary subd}) using stencils for the modified Catmull-Clark subdivision scheme. The remaining parts of the meshes are still subdivided by PN-Catmull-Clark subdivision.

\begin{figure*}[htb]
  \centering
  \includegraphics[width=0.23\linewidth]{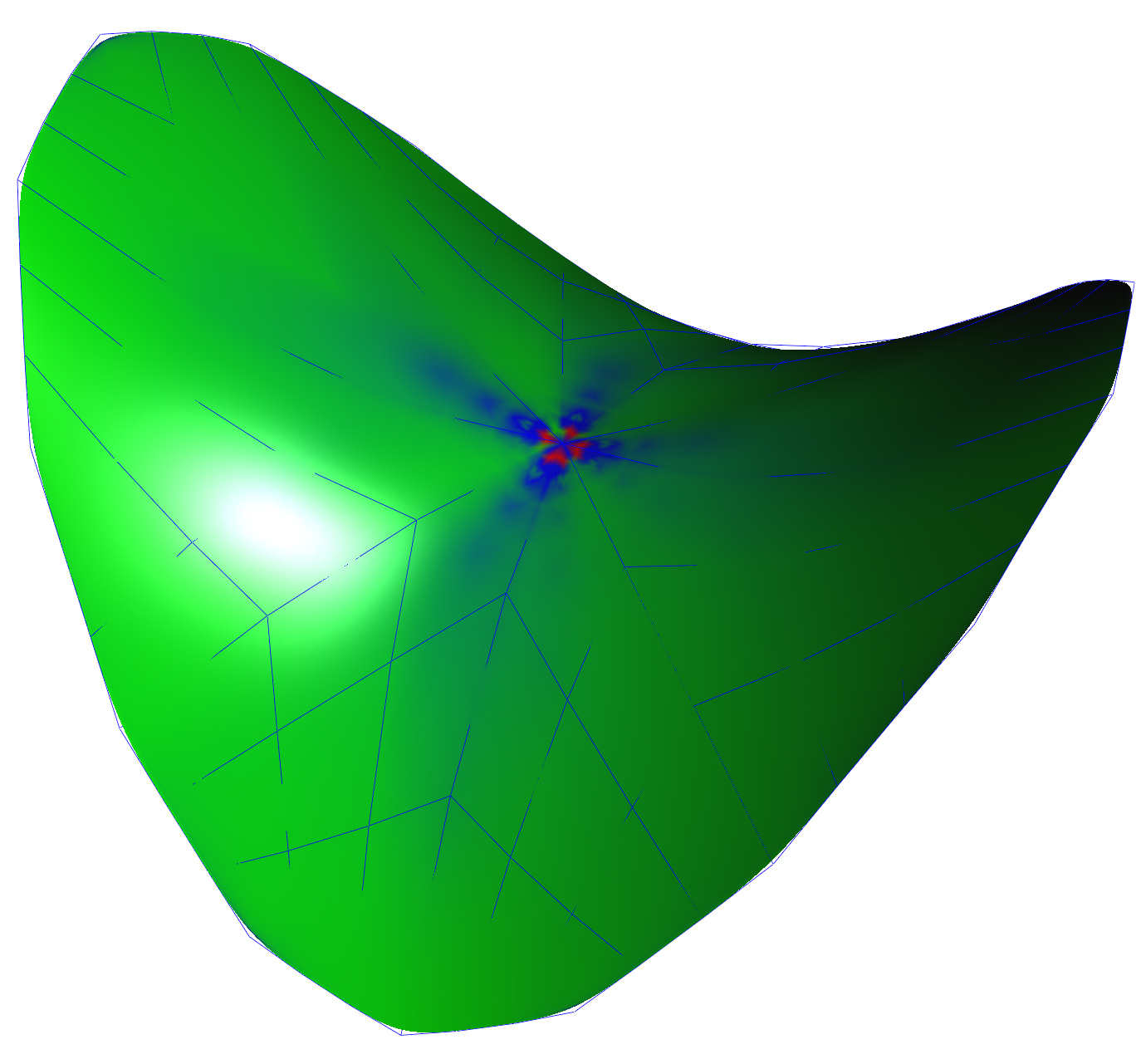}
  \includegraphics[width=0.23\linewidth]{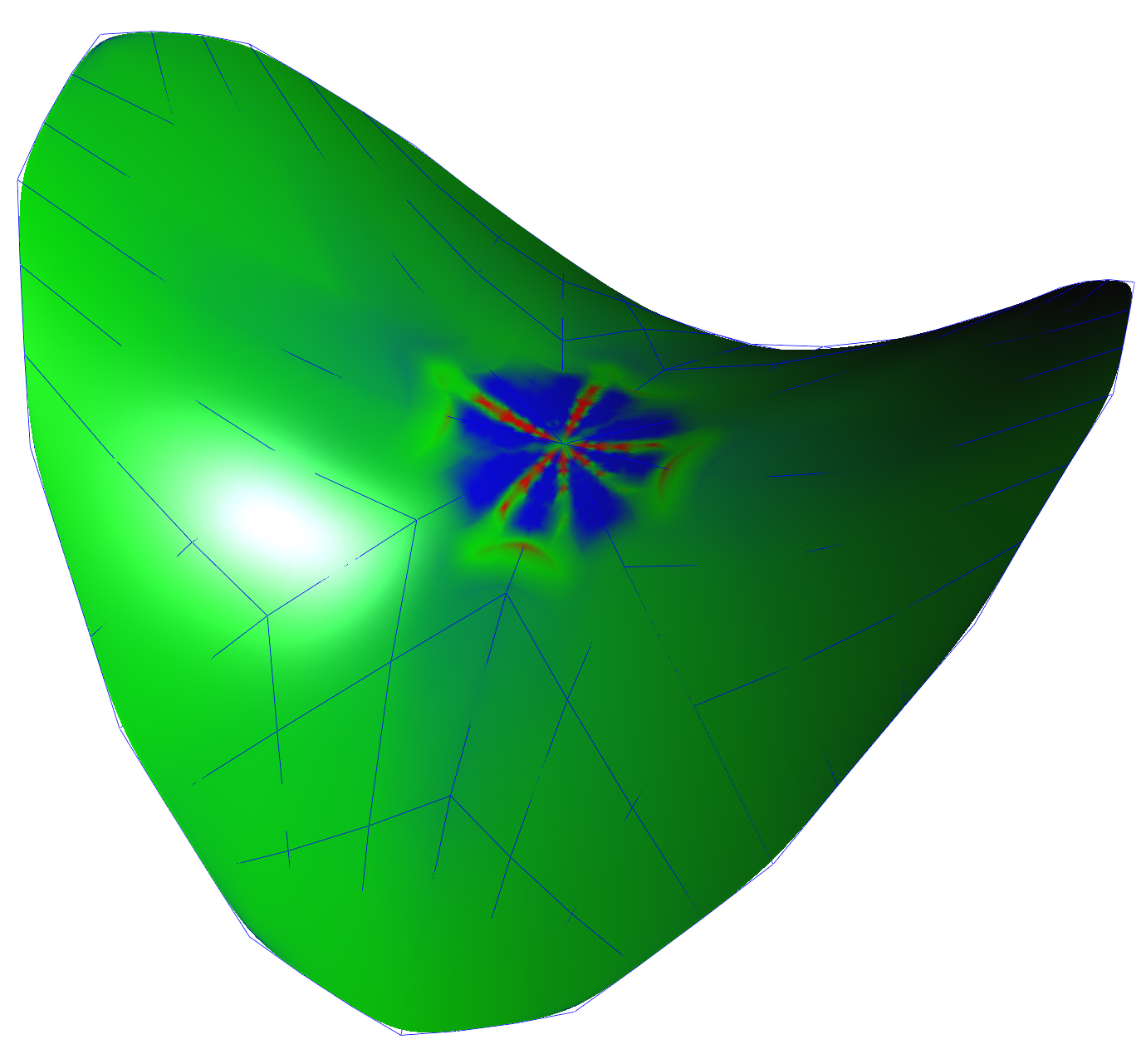}
  \includegraphics[width=0.23\linewidth]{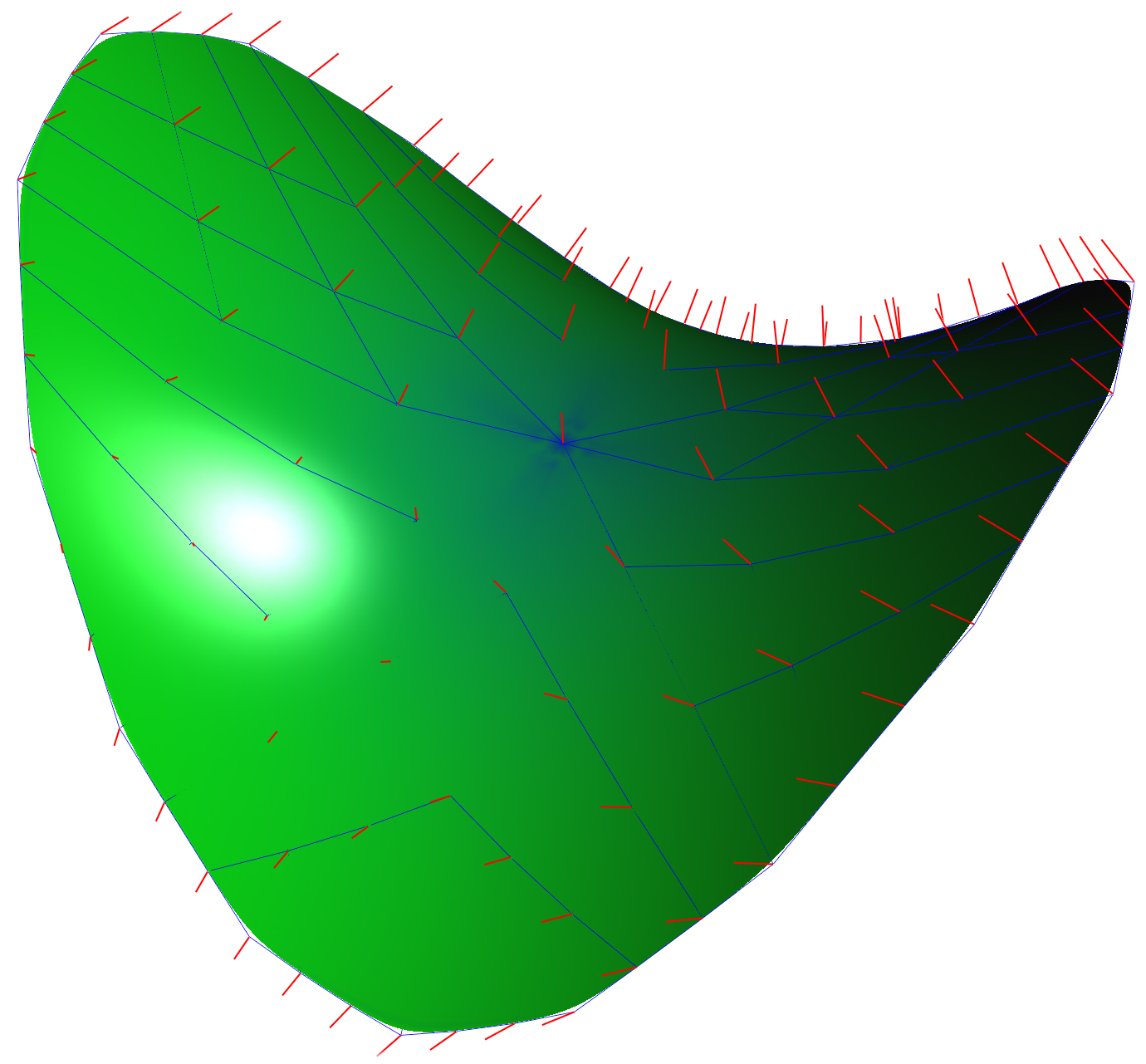}
  \includegraphics[width=0.23\linewidth]{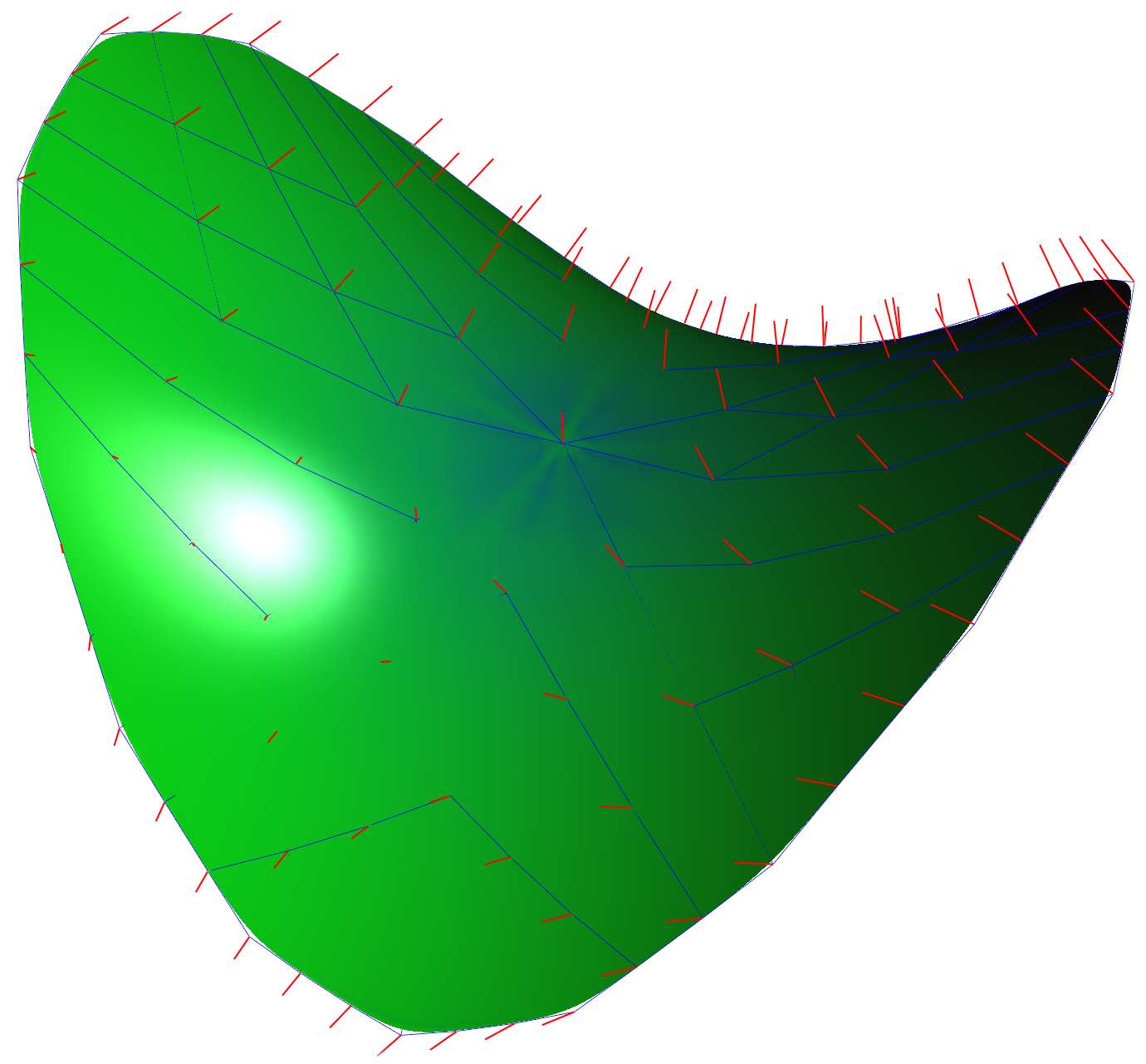} \\

  \includegraphics[width=0.23\linewidth]{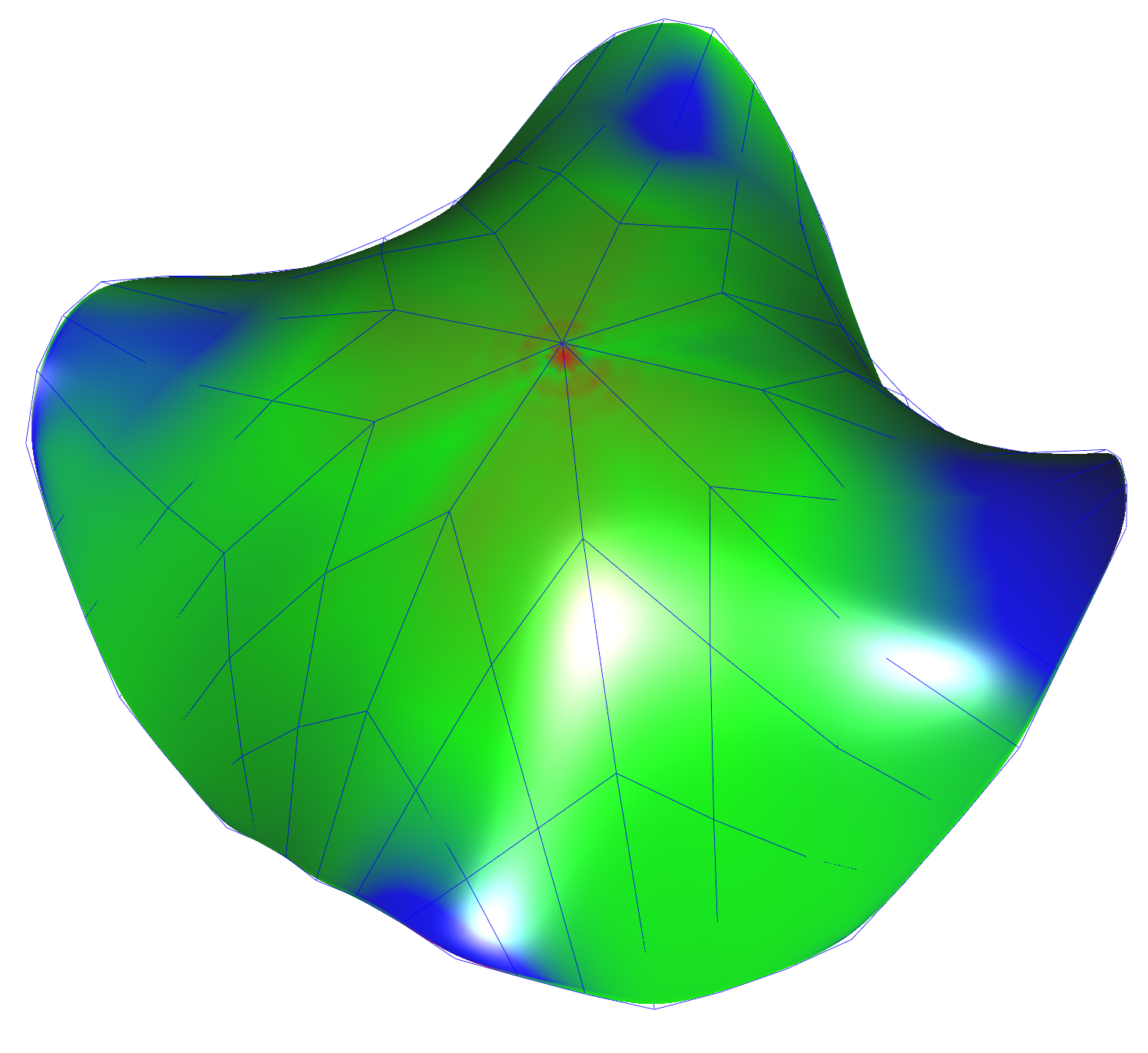}
  \includegraphics[width=0.23\linewidth]{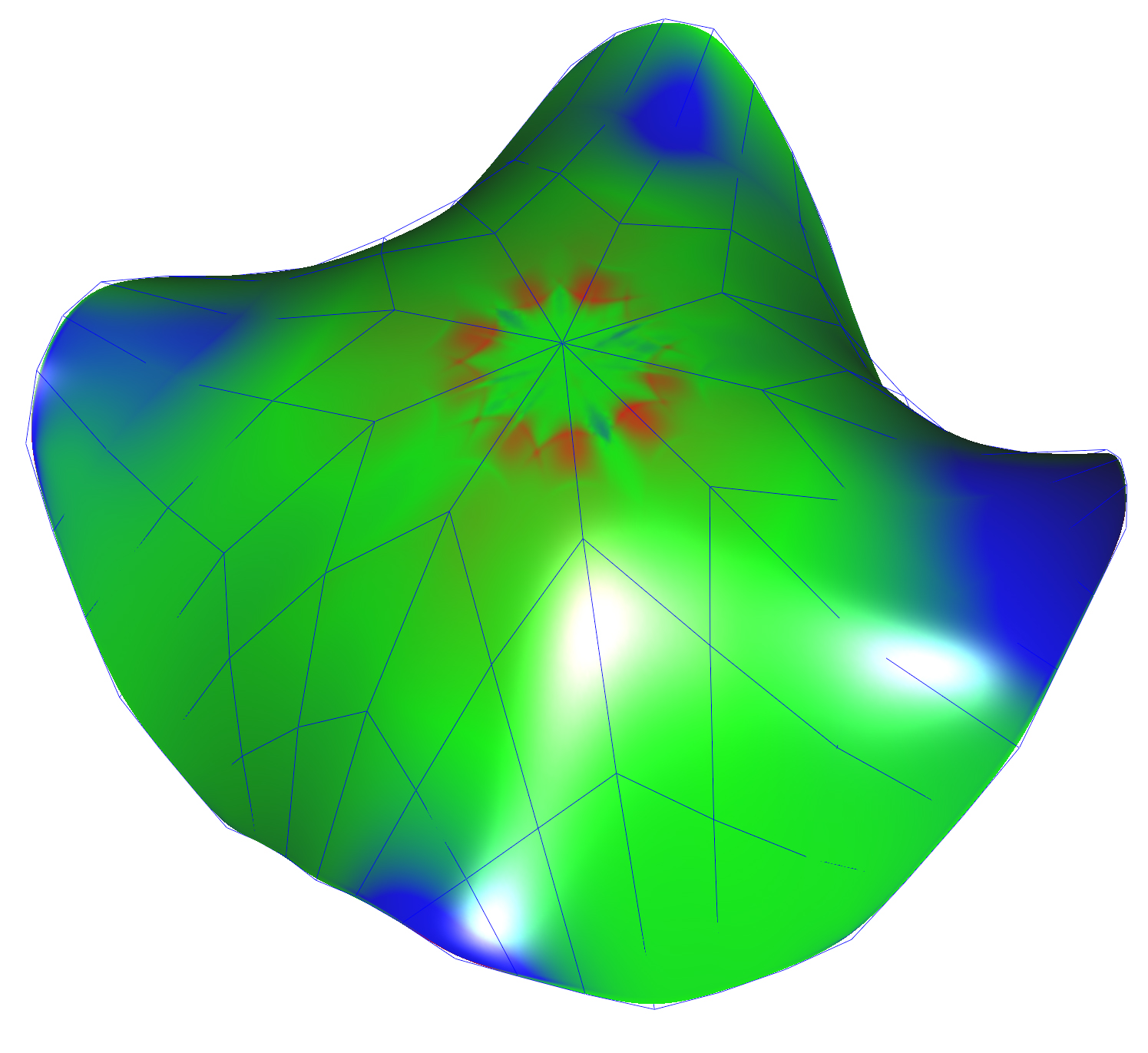}
  \includegraphics[width=0.23\linewidth]{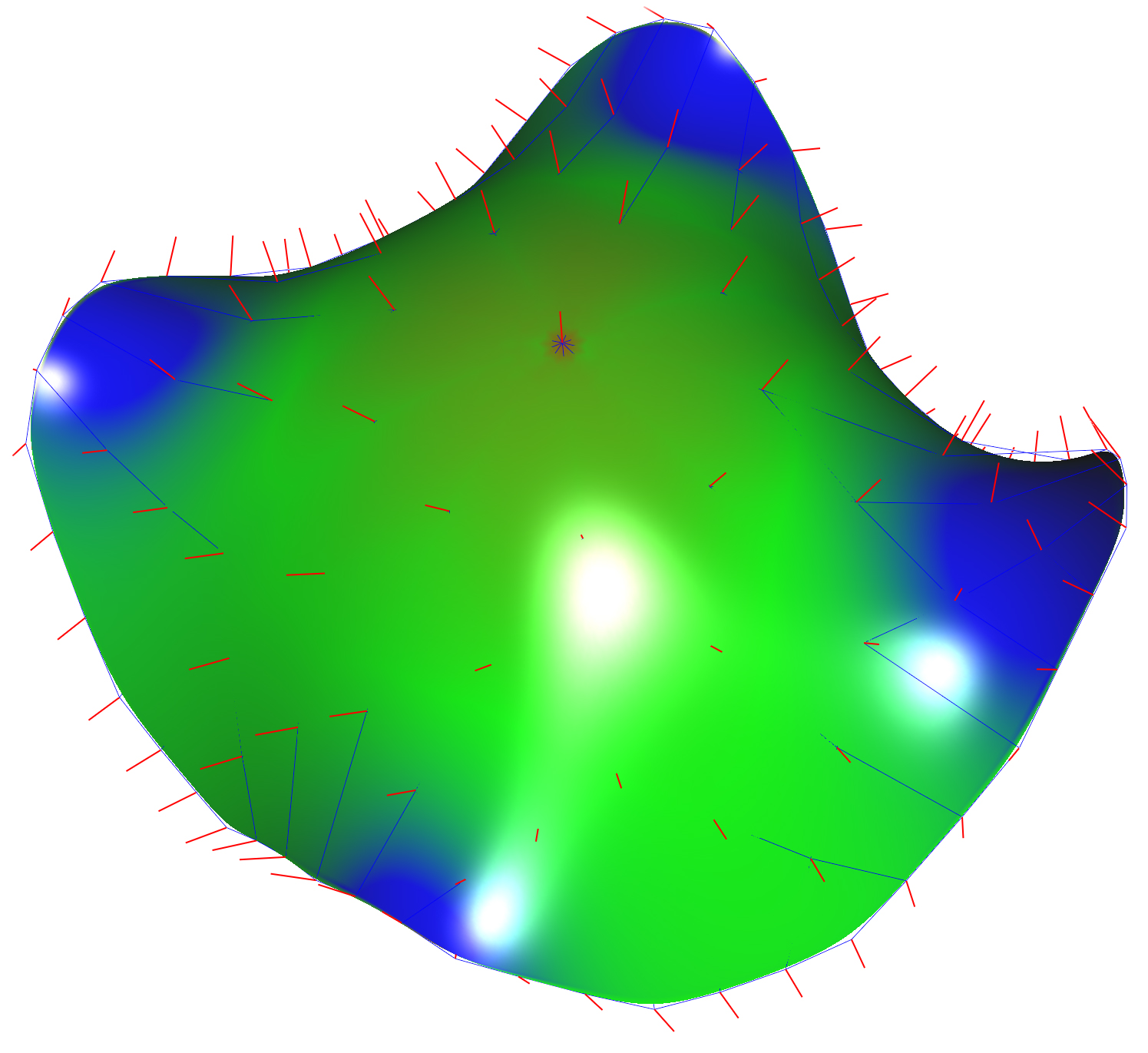}
  \includegraphics[width=0.23\linewidth]{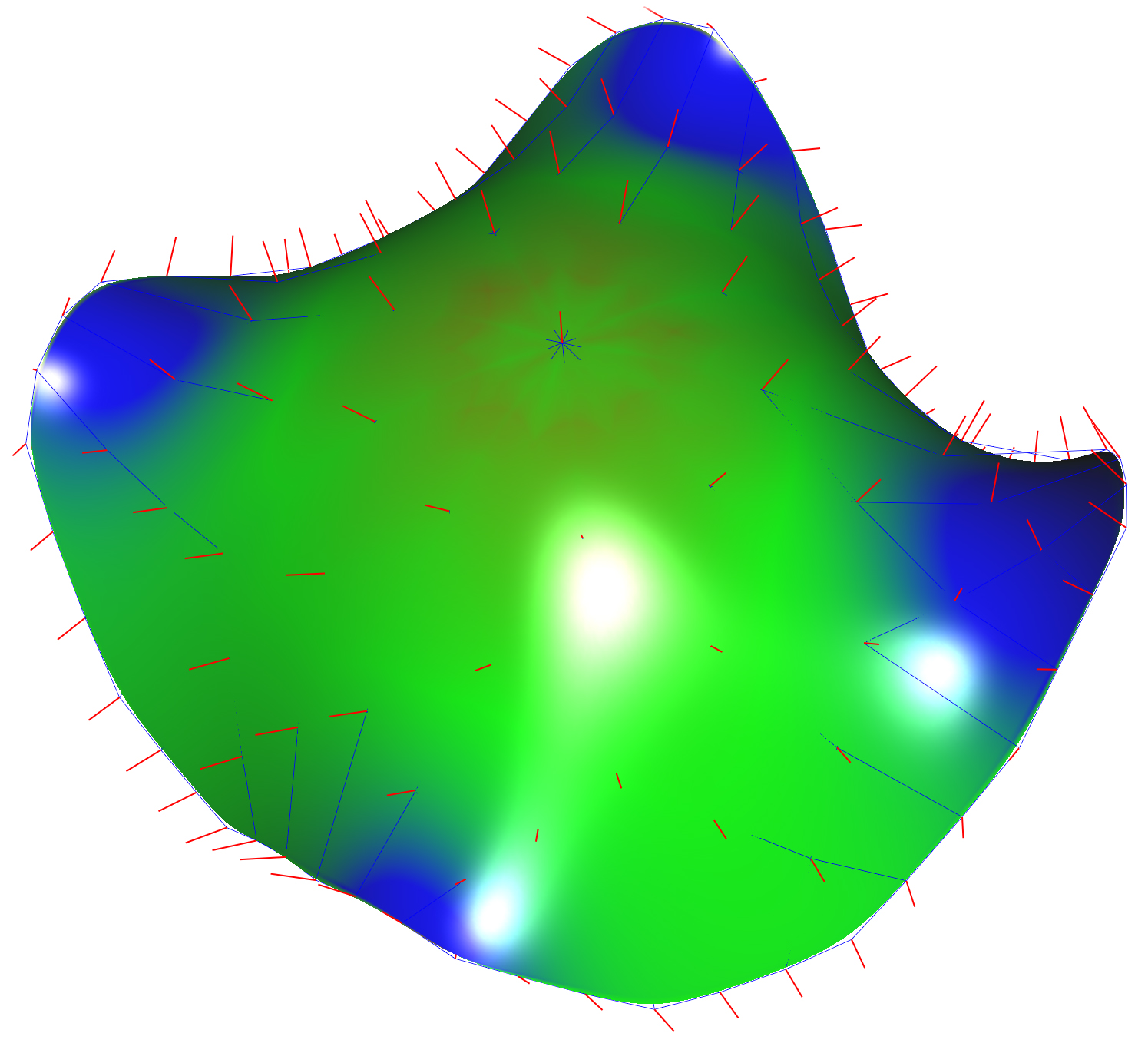} \\

  \subfigure[]{\includegraphics[width=0.23\linewidth]{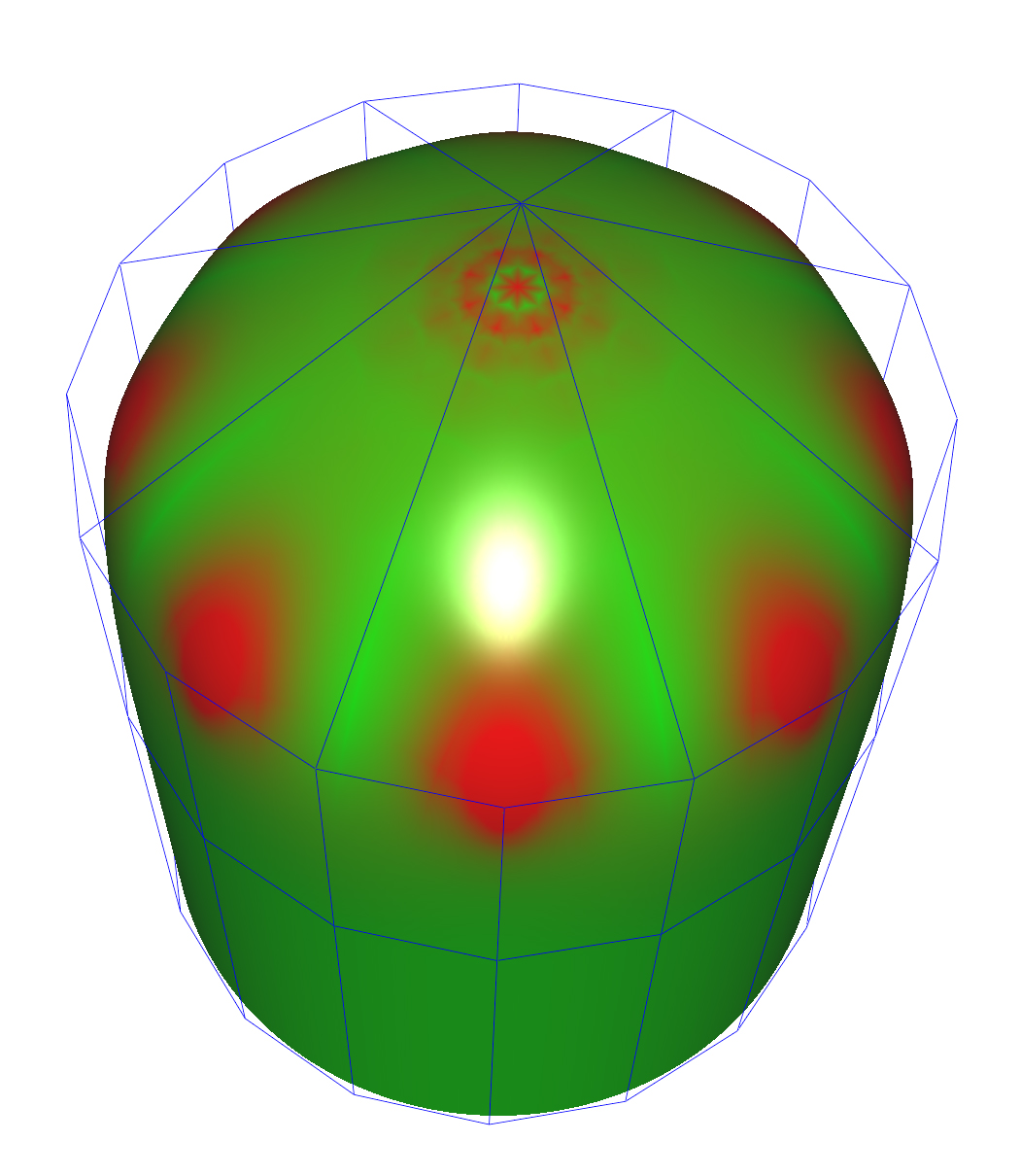}}
  \subfigure[]{\includegraphics[width=0.23\linewidth]{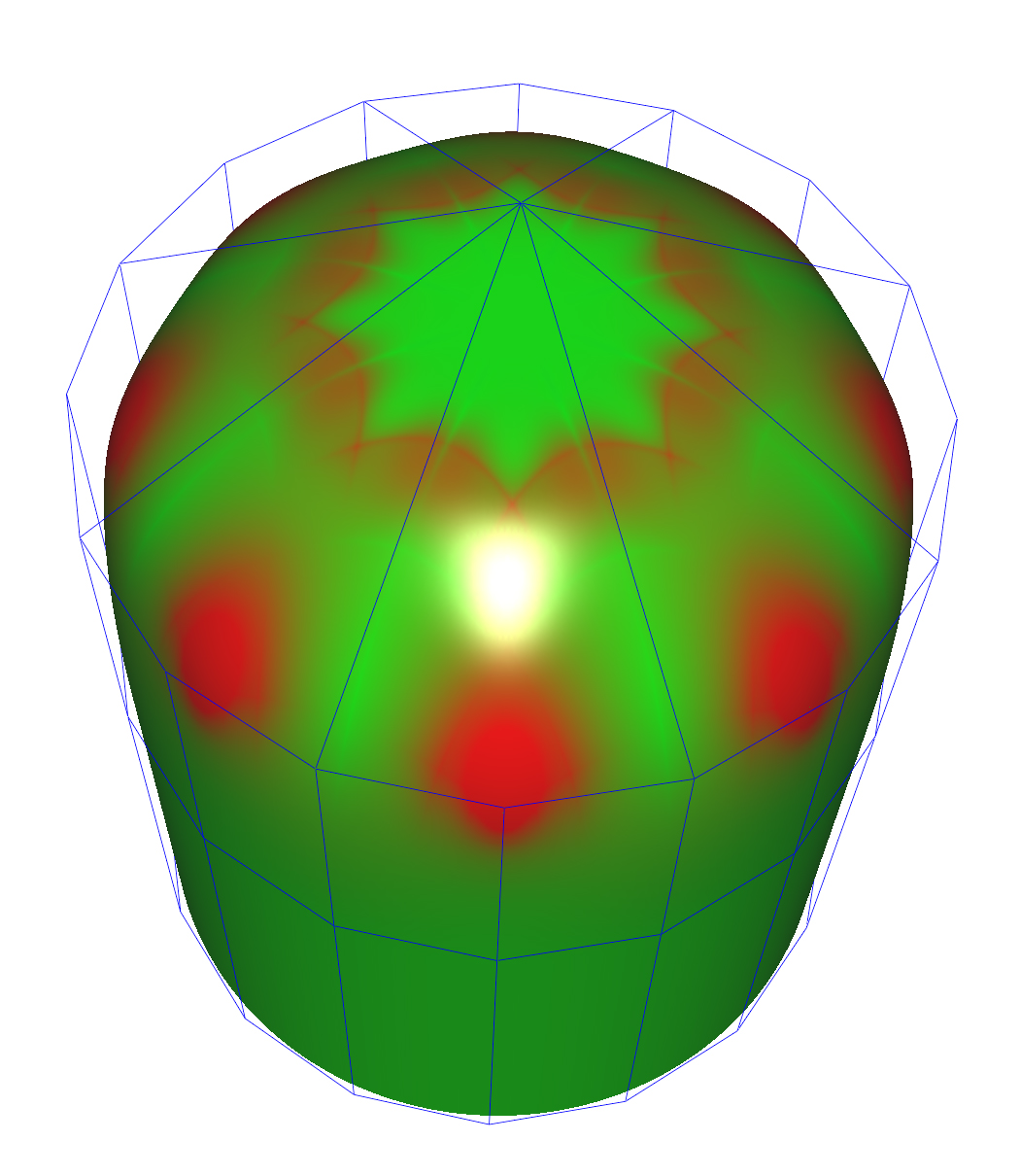}}
  \subfigure[]{\includegraphics[width=0.23\linewidth]{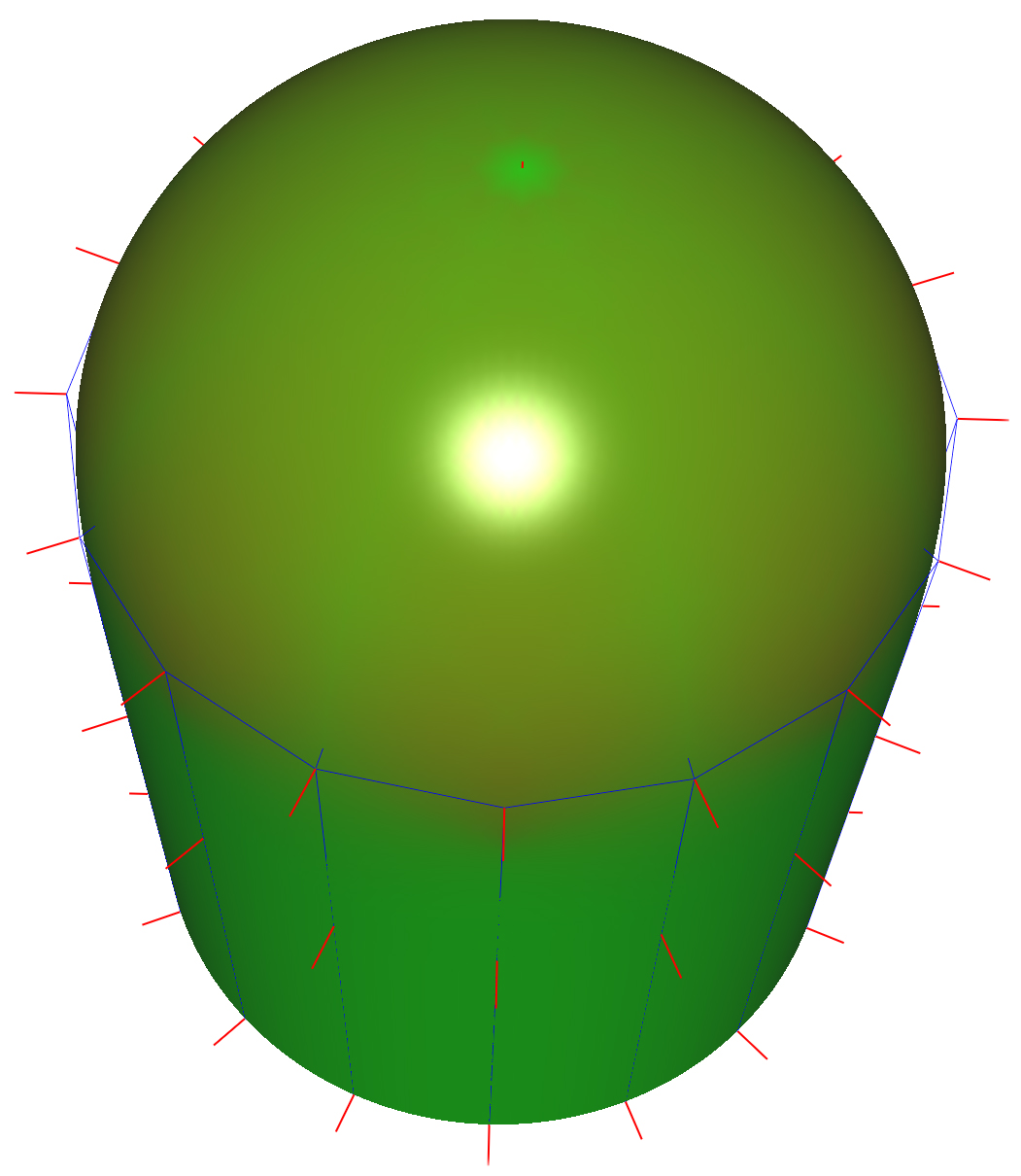}}
  \subfigure[]{\includegraphics[width=0.23\linewidth]{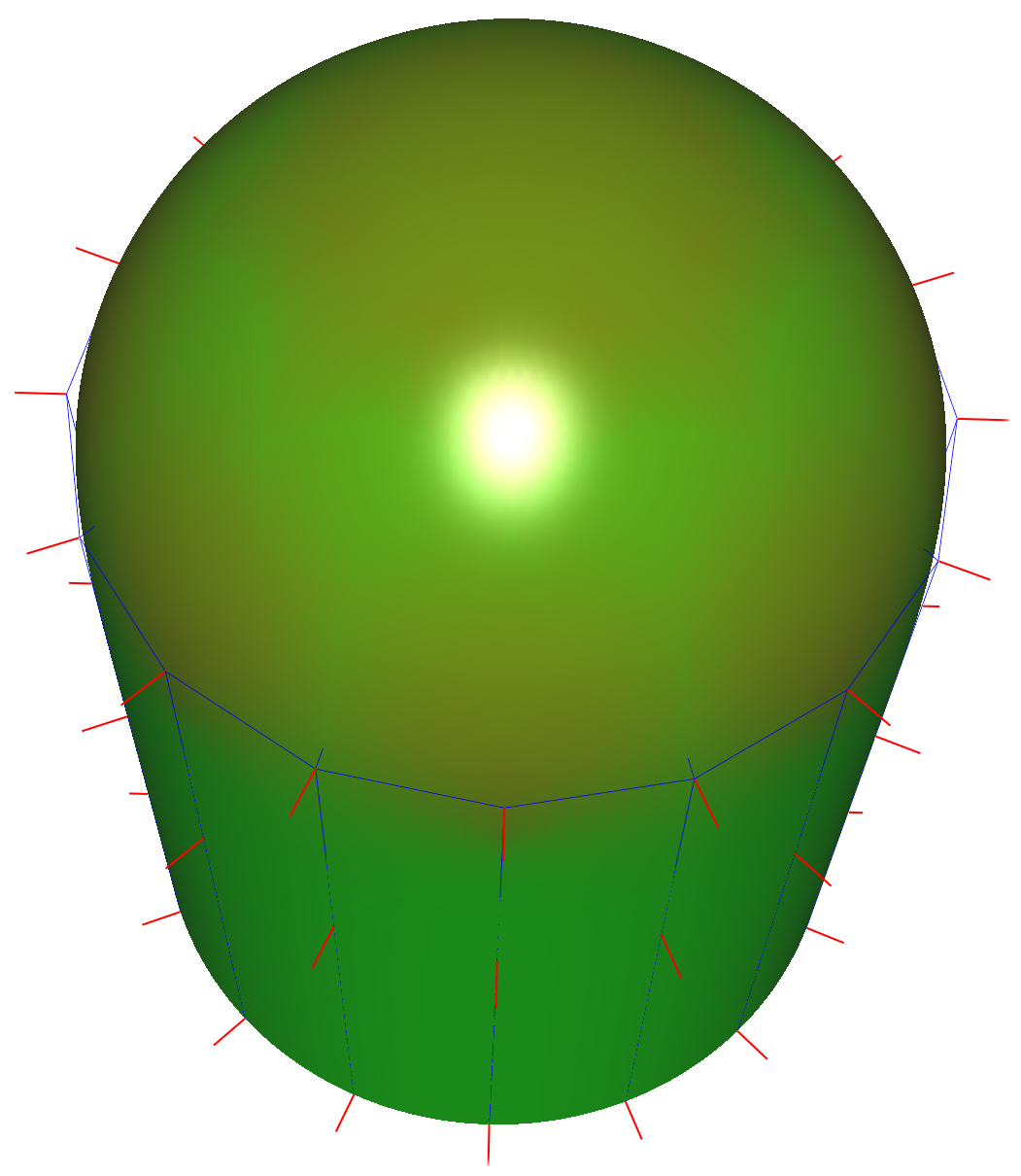}}

  \caption{Gaussian curvature plots of subdivision surfaces obtained by (a) Catmull-Clark subdivision; (b) modified Catmull-Clark subdivision~\citep{Prautzsch1998G2subdivisionsurface}; (c) PN-Catmull-Clark subdivision; (d) PN-modified Catmull-Clark subdivision. }
  \label{Fig:PN C2 subdivision_surfaces}
\end{figure*}

Since the modified Catmull-Clark subdivision is $C^2$ continuous, it is also $C^1$ continuous. Based on Theorem \ref{Theorem:C0 PN subdivision surface} and Theorem \ref{Theorem:C1 PN subdivision surface}, we know that the PN-modified Catmull-Clark subdivision converges and the obtained subdivision surfaces are at least normal continuous at the extraordinary points. It is observed that the surfaces generated by PN-modified Catmull-Clark subdivision are $C^2$ continuous too and the curvatures at the extraordinary points can be no longer vanishing. However, the theoretical proof of $C^2$ continuity of PN-modified Catmull-Clark subdivision is not available at present. We present the assertion as a conjecture.
\begin{conjecture}
The PN-modified Catmull-Clark subdivision can generate curvature continuous subdivision surfaces and the extraordinary points of the surfaces can be no longer flat when the control normals are not a constant vector nor vanish.
\end{conjecture}

Figure \ref{Fig:PN C2 subdivision_surfaces} illustrates examples of surface modeling by Catmull-Clark type subdivision schemes or their adapted PN subdivision schemes. The control points and control normals for the control mesh in the top row are sampled from a hyperbolic surface $z=2xy$ while the control points and control normals for the control mesh in the middle row are sampled from a bicubic B\'{e}zier surface, both with one extraordinary vertex in the center. The control points and control normals for the control mesh in the bottom row are partially sampled from a circular cylinder with radius 15. An irregular vertex of valence 8 lies above the center of the upper base of the cylinder with height 10 and the control normal at the point is chosen the unit upright vector. Since the eigenvalues of subdivision matrices for meshes containing single irregular vertices of valence 3 already satisfy the $G^2$ condition stated in Proposition \ref{Proposition:C2 continuity at extraordinary point}, we only modify subdivision stencils for meshes surrounding irregular vertices of valences greater than 4 for the modified Catmull-Clark subdivision or PN modified Catmull-Clark subdivision. Figures \ref{Fig:PN C2 subdivision_surfaces}(a) and \ref{Fig:PN C2 subdivision_surfaces}(b) show clearly that Catmull-Clark subdivision surfaces are not curvature continuous at the extraordinary points while the surfaces obtained by the modified Catmull-Clark subdivision scheme have flat extraordinary points. Though the PN-Catmull-Clark subdivision scheme can generate much fairer subdivision surfaces than Catmull-Clark subdivision, they still suffer the curvature discontinuities at the extraordinary points; see Figure \ref{Fig:PN C2 subdivision_surfaces}(c). The pictures in Figure \ref{Fig:PN C2 subdivision_surfaces}(d) show that the surfaces obtained by PN-modified Catmull-Clark subdivision are visually curvature continuous and the curvatures at the extraordinary points are not vanishing.

%%%%%%%%%%%%%%%%%%%%%%%%%%%%%%%%%%%%%%%%%%%%%%
%%%% Section 6
%%%%%%%%%%%%%%%%%%%%%%%%%%%%%%%%%%%%%%%%%%%%%%

\section{Experimental examples}
\label{Sec:Examples}
In this section we present several interesting examples to show the modeling effects of PN subdivision schemes, comparisons with some linear subdivision schemes are also given.

\begin{figure*}[htb]
  \centering
  \subfigure[]{\includegraphics[width=0.36\linewidth]{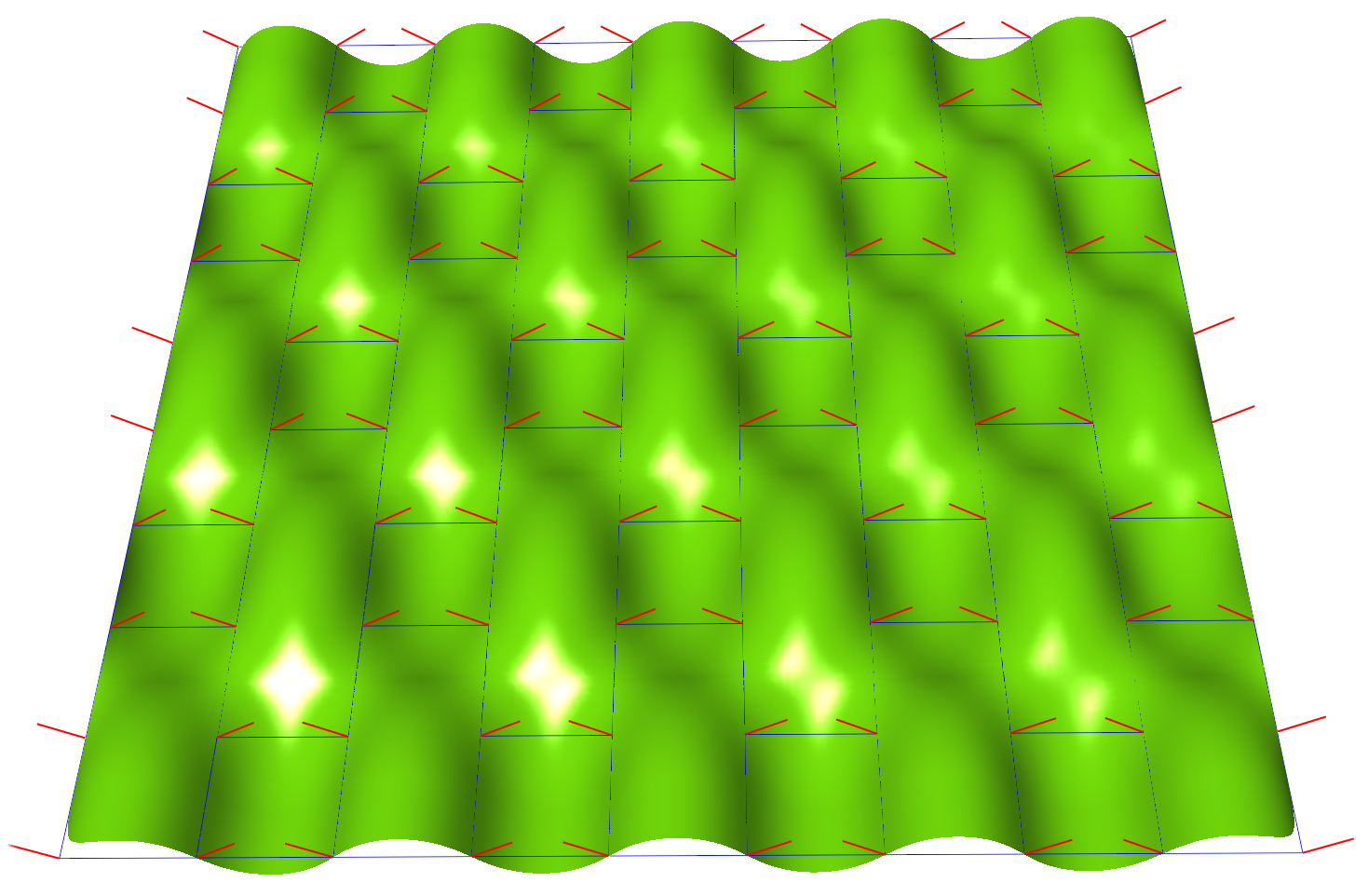}} \ \ \ \
  \subfigure[]{\includegraphics[width=0.36\linewidth]{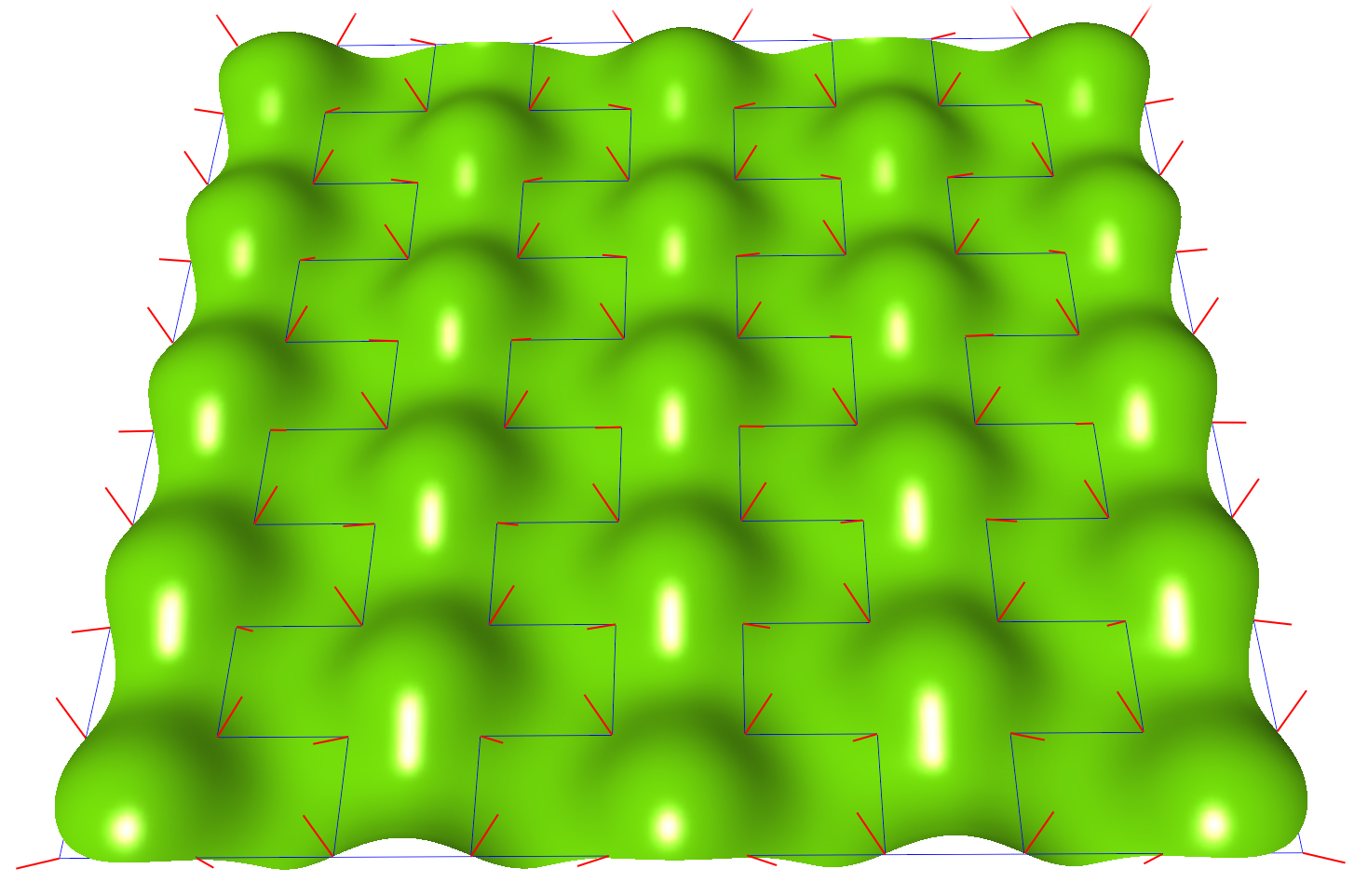}}
  \caption{PN-Doo-Sabin subdivision surfaces with planar uniform control grid and edited control normals: (a) wave like shape; (b) circular bumps. }
  \label{Fig:plane-balls-waves}
\end{figure*}

Figure \ref{Fig:plane-balls-waves} illustrates two examples of surface detail modeling by PN subdivision. Given a planar uniform control grid, obviously, any linear subdivision scheme can only yield a planar patch. We edit surface details by editing control normals at the vertices. Firstly, the control normals at the vertices are chosen from two given vectors alternately in the horizontal direction and every two control normals are parallel with each other in the vertical direction. A wave-like shape following the control normals is obtained by PN-Doo-Sabin subdivision; see Figure \ref{Fig:plane-balls-waves}(a). Besides wave-like shape, we can also model bumps on the subdivision surface by editing control normals. Assume four unit vectors are uniformly chosen from a hemisphere. We line up the vertices of the uniform control grid row by row and set control normals for the vertices from the four vectors repeatedly. As a result, a surface with regular distributed circular bumps is obtained by PN-Doo-Sabin subdivision; see Figure \ref{Fig:plane-balls-waves}(b) for the subdivision surface.

\begin{figure*}[htb]
  \centering
  \subfigure[]{\includegraphics[width=0.3\linewidth]{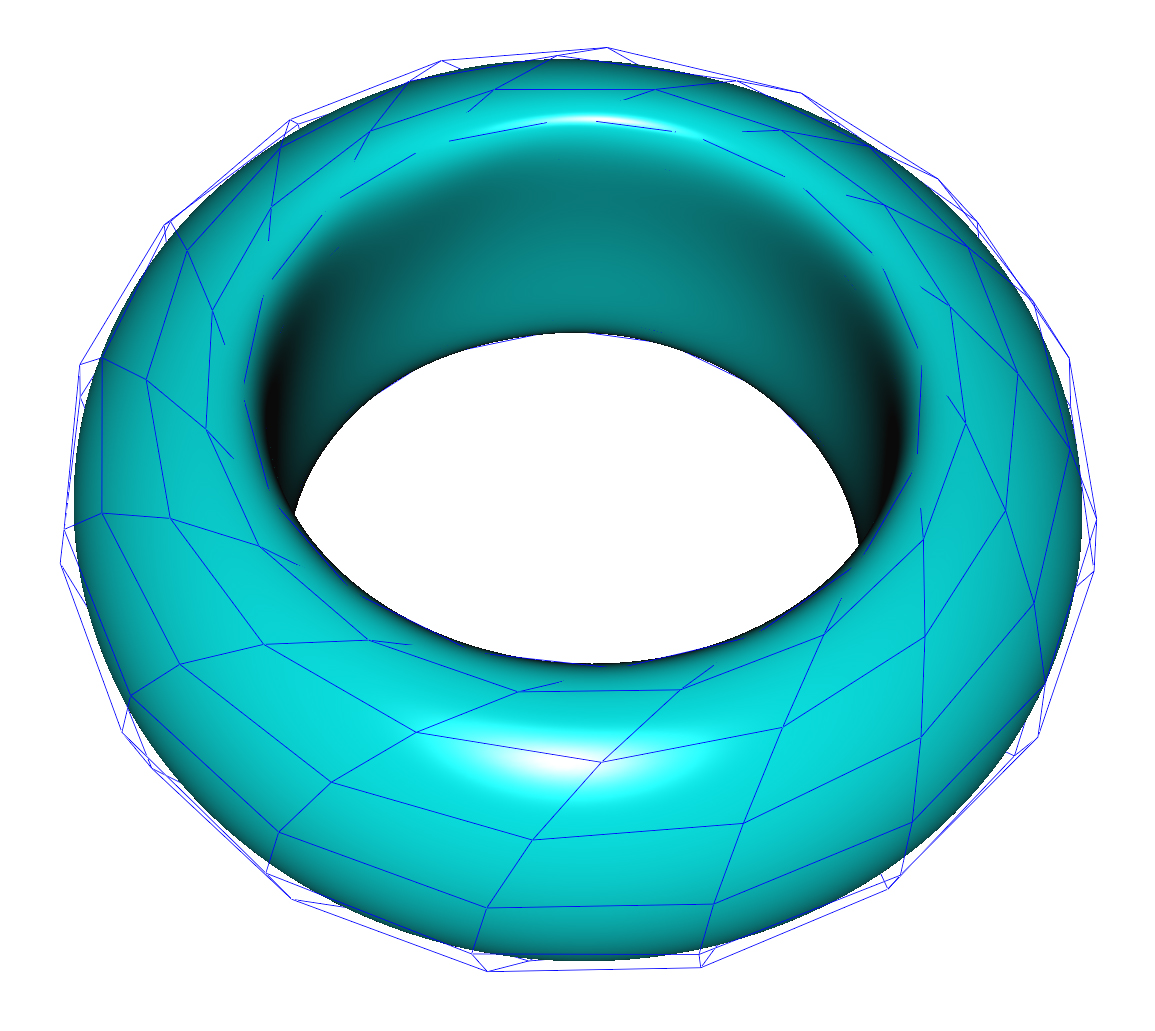}}
  \subfigure[]{\includegraphics[width=0.3\linewidth]{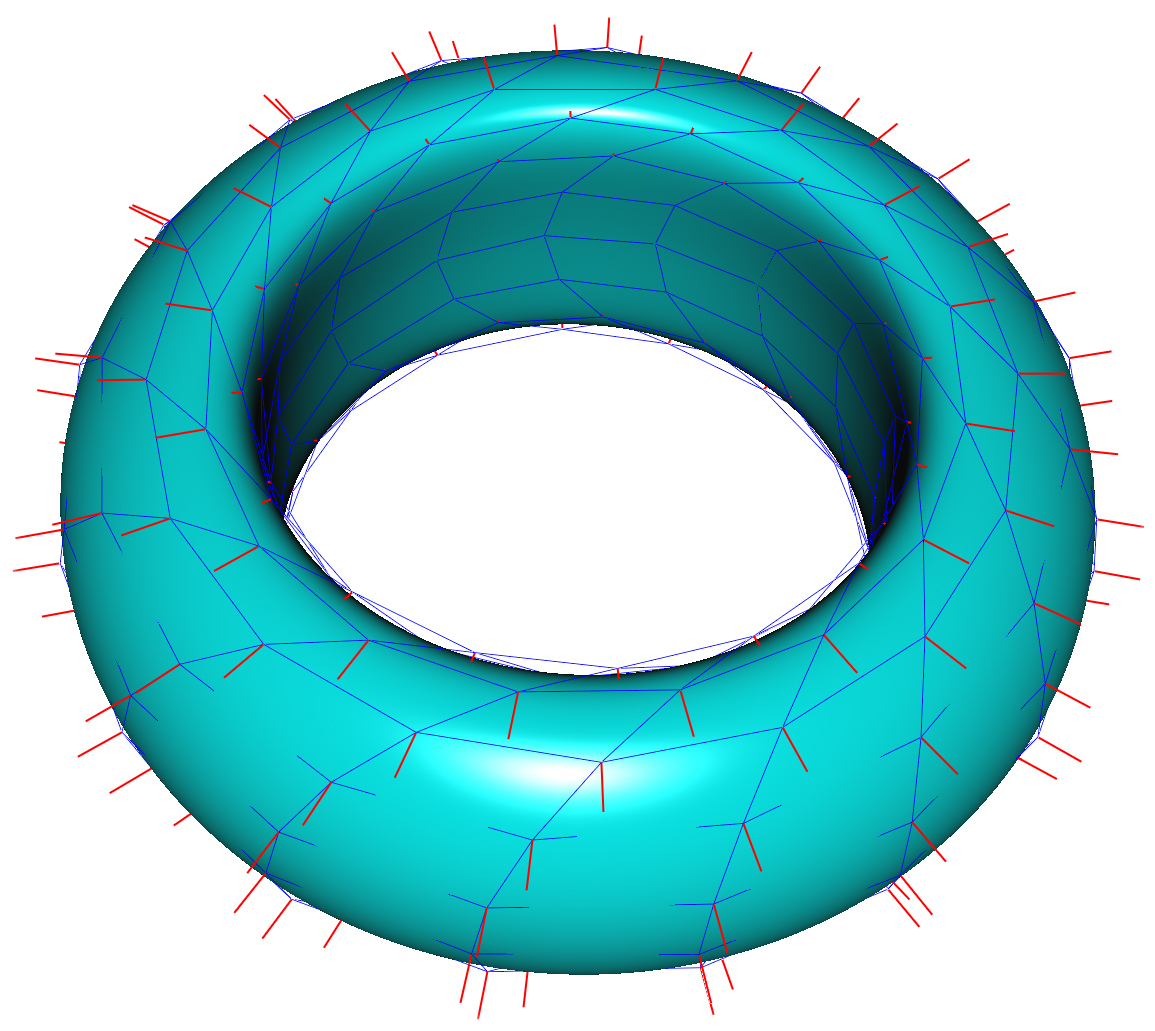}}
  \subfigure[]{\includegraphics[width=0.3\linewidth]{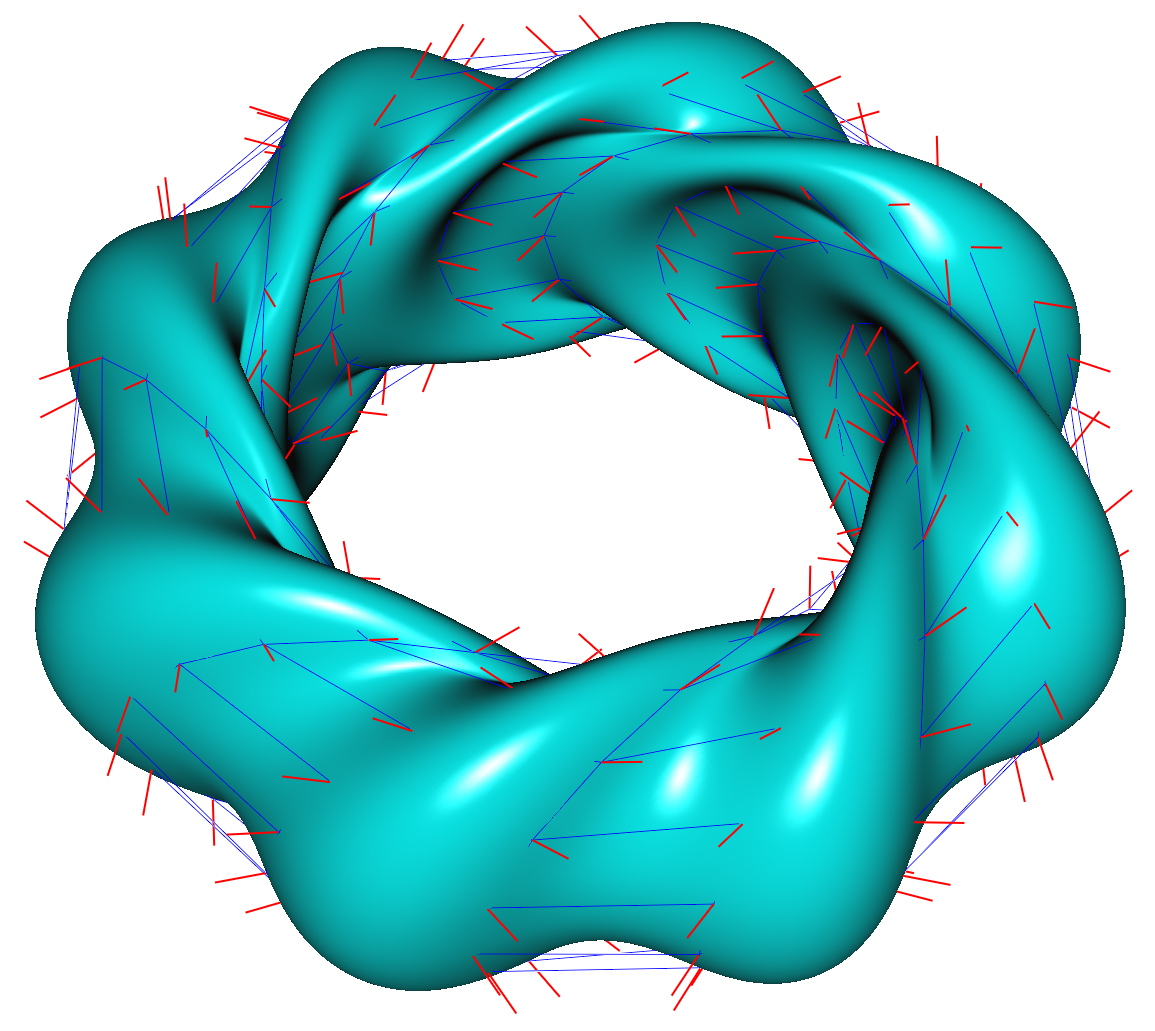}}
  \caption{Ring shape modeling by (a) Catmull-Clark subdivision; (b)\&(c) PN-Catmull-Clark subdivision.}
  \label{Fig:RingShape}
\end{figure*}

Figure \ref{Fig:RingShape}(a) illustrates a ring shape surface by Catmull-Clark subdivision. The control mesh for the surface is constructed by rotating a closed regular polygon along an axis that does not lie on the same plane with the polygon. Since the Catmull-Clark subdivision surface with regular control mesh is actually a bicubic B-spline surface, it is not exactly a rotating surface. By choosing all control normals pointing outwards and being parallel to the bottom plane, an exact rotating surface is obtained by PN-Catmull-Clark subdivision; see Figure \ref{Fig:RingShape}(b). If the control normals at the vertices of the control mesh are edited further, a ring shape surface with complex details is obtained by PN-Catmull-Clark subdivision; see Figure \ref{Fig:RingShape}(c).

\begin{figure*}[htb]
  \centering
  \subfigure[]{\includegraphics[width=0.285\linewidth]{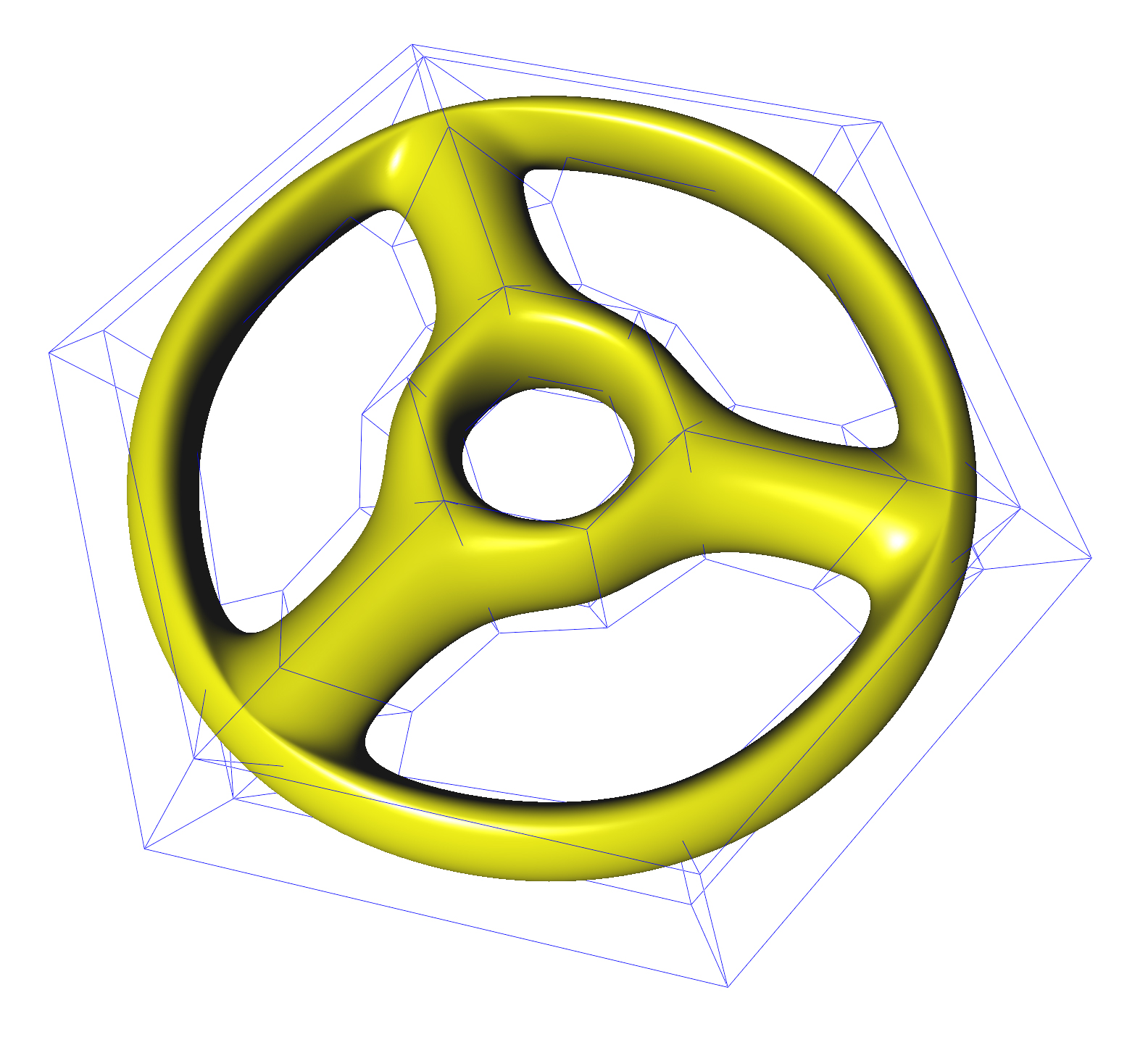}}
  \subfigure[]{\includegraphics[width=0.285\linewidth]{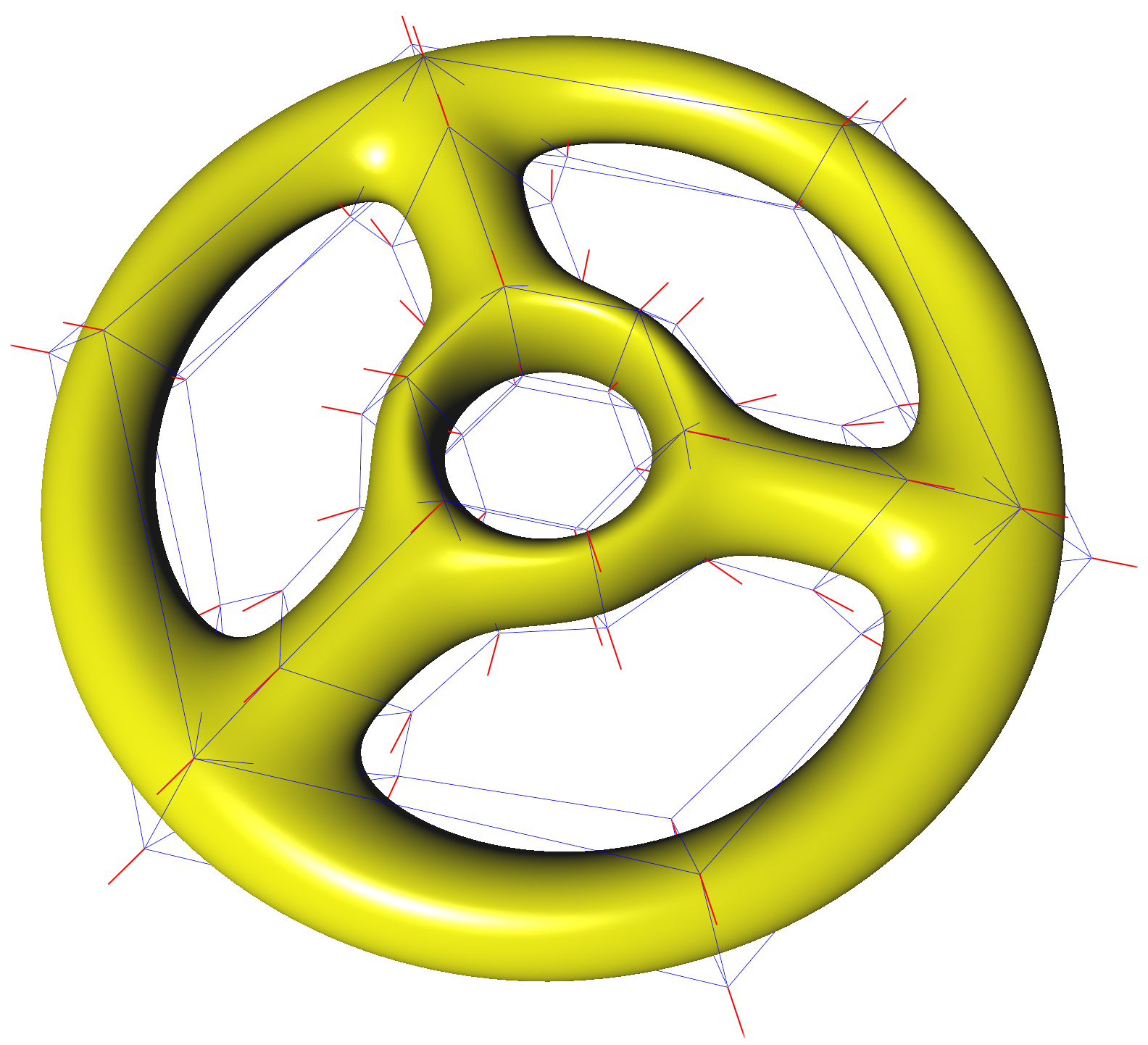}}
  \subfigure[]{\includegraphics[width=0.285\linewidth]{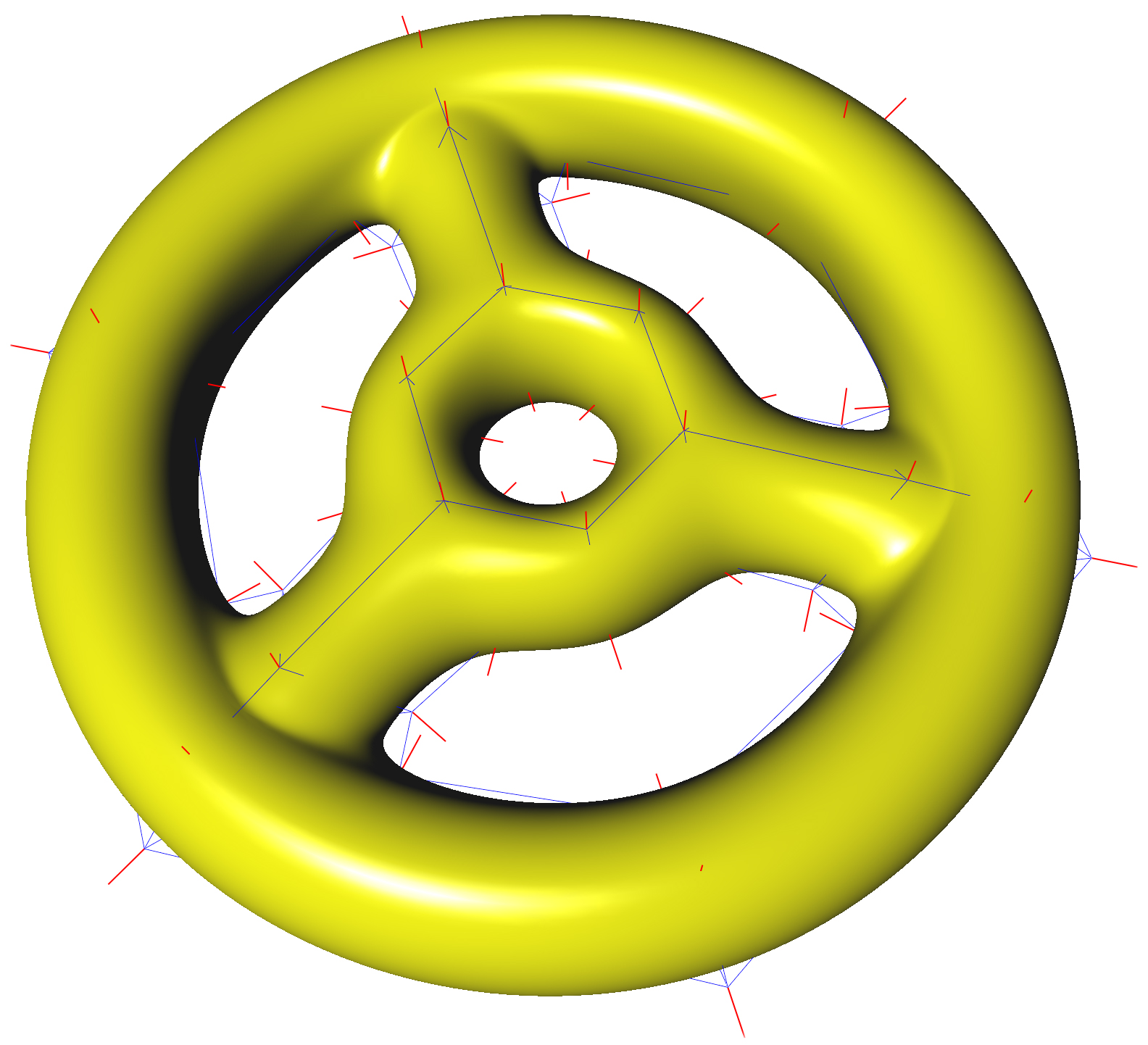}}
  \caption{Wheel shape modeling by (a) Catmull-Clark subdivision; (b)\&(c) PN-Catmull-Clark subdivision. }
  \label{Fig:wheel modeling}
\end{figure*}

Figure \ref{Fig:wheel modeling} illustrates examples of wheel shape modeling by Catmull-Clark subdivision or PN-Catmull-Clark subdivision. Given a control mesh as in Figure \ref{Fig:wheel modeling}(a), a wheel like shape is obtained by Catmull-Clark subdivision. Though the outer part and the inner part of the control mesh are regular, neither the outer contour profile nor the inner one is exactly circular because the subdivision surfaces under regular control meshes are just bicubic B-spline surfaces. Assume the center of the control mesh lies at the origin of a Cartesian coordinates system and the plane on which the control mesh lies on is parallel to the $xy$-plane. We first choose control normal at each control point $\mathbf{p}_i=(x_i,y_i,z_i)^\top$ as $\mathbf{n}_i=\mathrm{normalize}(x_i,y_i,0)^\top$. A wheel like shape that has exact circular contour profiles is obtained by PN-Catmull-Clark subdivision; see Figure \ref{Fig:wheel modeling}(b). Since the control normals are all parallel to the $xy$-plane, the subdivision surface in Figure\ref{Fig:wheel modeling}(b) and the subdivision surface in Figure\ref{Fig:wheel modeling}(a) have the same $z$-coordinates. If the control normals have been changed as in Figure \ref{Fig:wheel modeling}(c), the two ring parts within the PN-Catmull-Clark subdivision surface resemble two toruses very well.

\begin{figure*}[htb]
  \centering
  \subfigure[]{\includegraphics[width=0.285\linewidth]{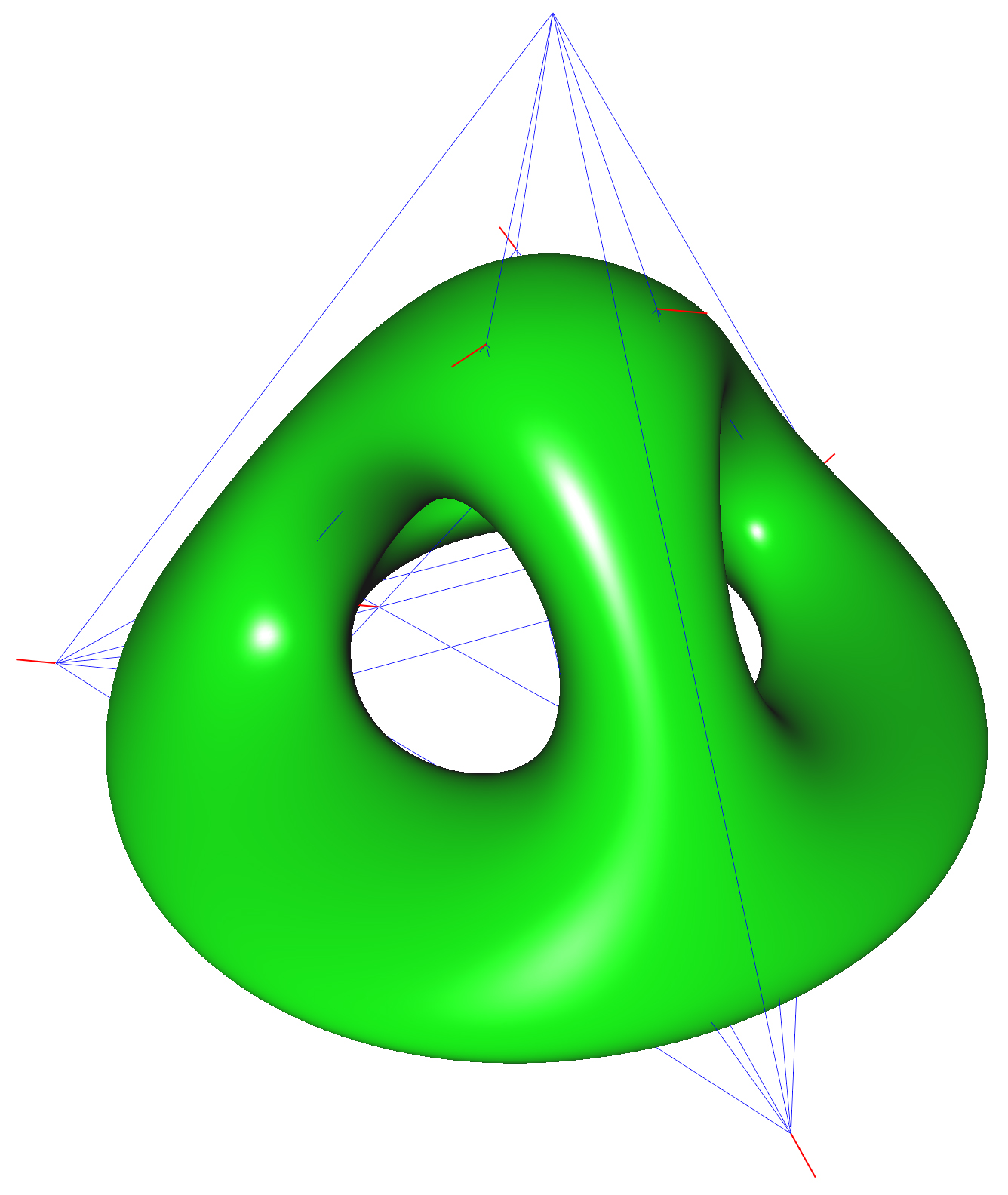}}
  \
  \subfigure[]{\includegraphics[width=0.285\linewidth]{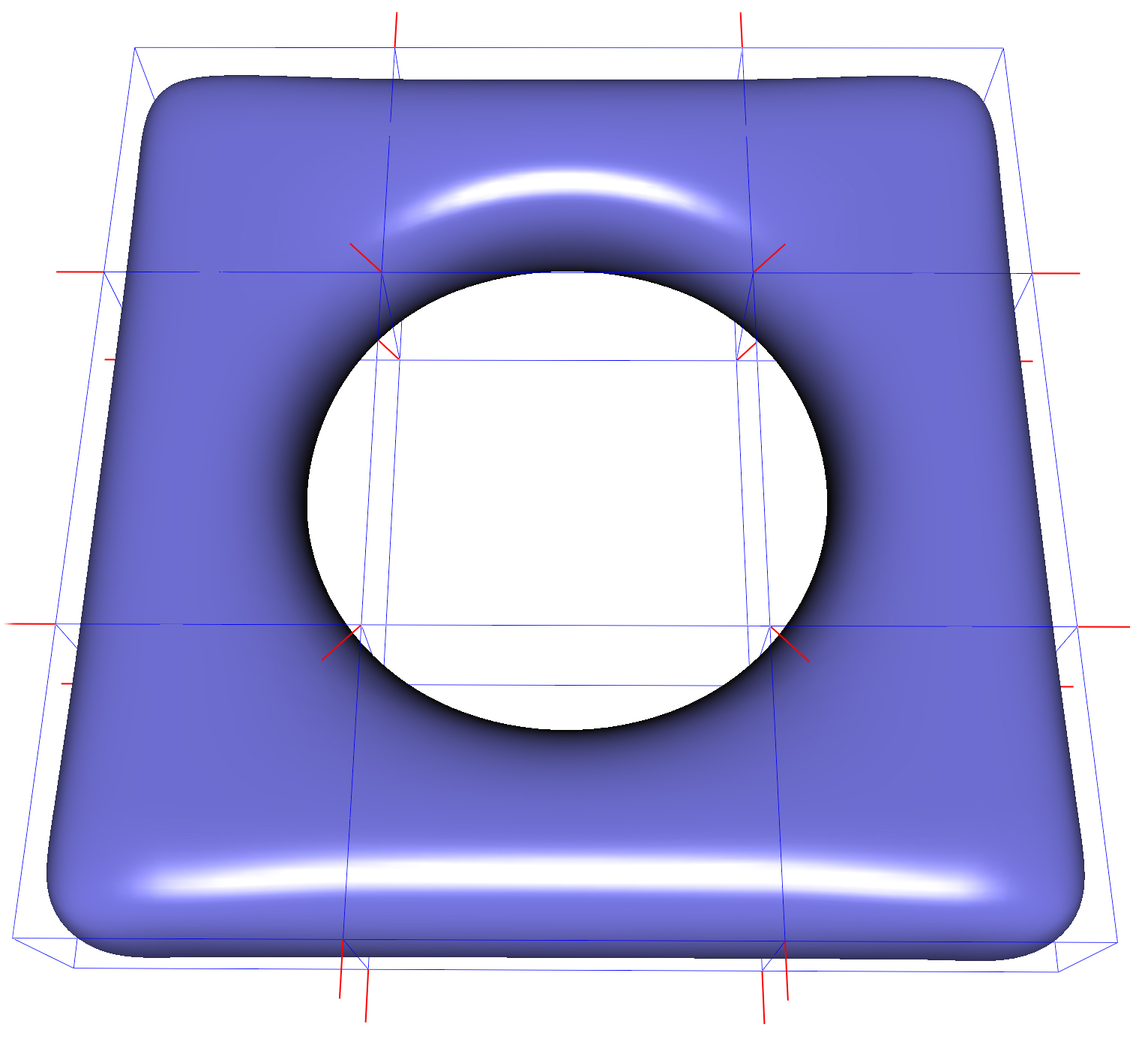}}
  \
  \subfigure[]{\includegraphics[width=0.285\linewidth]{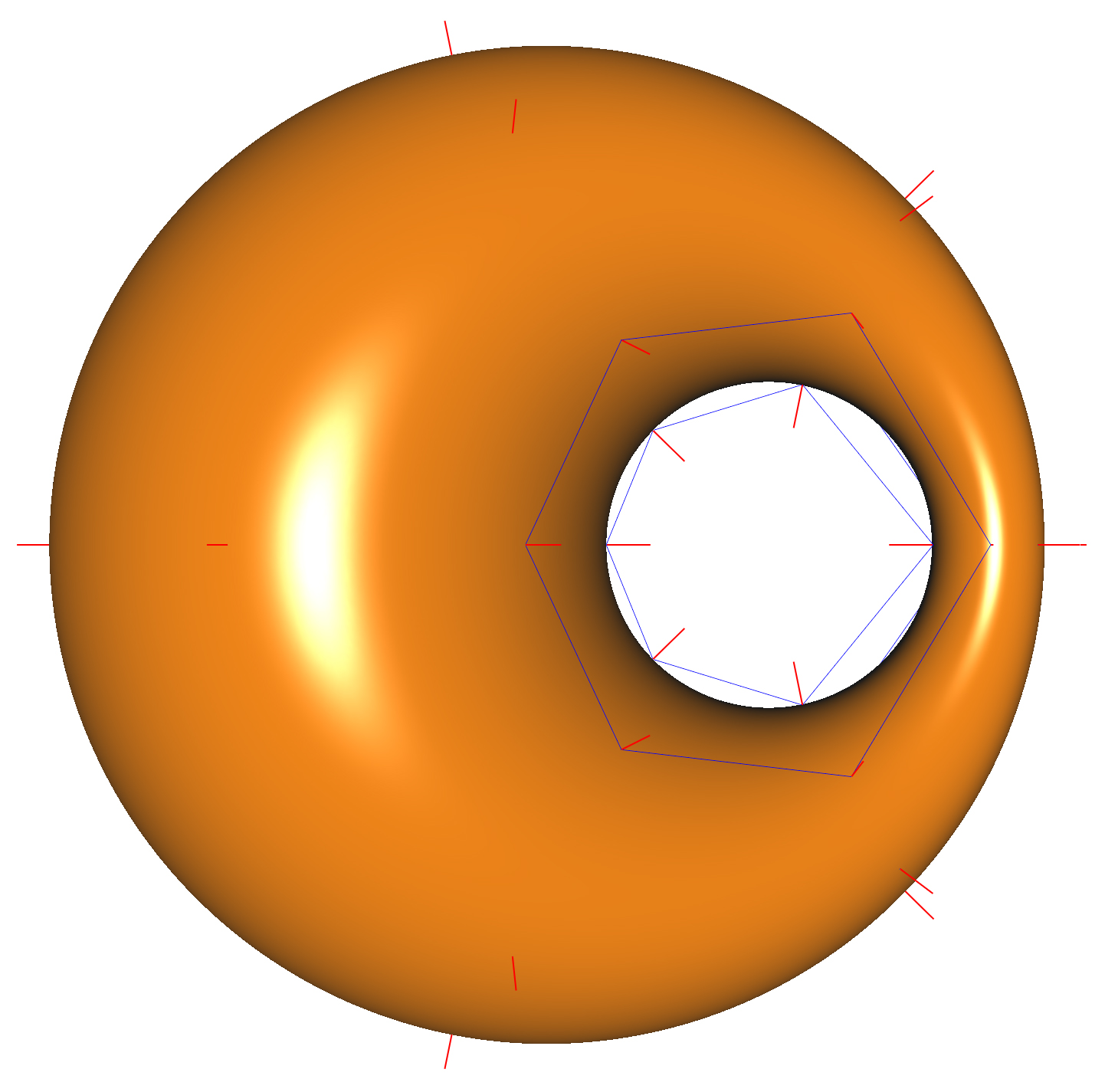}}
  \caption{PN subdivision surface modeling: (a) PN-Catmull-Clark subdivision; (b) PN-Doo-Sabin subdivision; (c) PN-Kobbelt subdivision.}
  \label{Fig:Complex PN subdivision surfce modeling}
\end{figure*}

Figure \ref{Fig:Complex PN subdivision surfce modeling} presents examples of modeling surfaces with complex topology or salient geometric features by PN subdivision schemes.
Figure \ref{Fig:Complex PN subdivision surfce modeling}(a) illustrates a PN-Catmull-Clark subdivision surface using control points and control normals. Except for the top vertex that has no control normal, the control normals at all other control points are parallel to the bottom plane and pointing outwards. As a result, the contour profile of the PN subdivision surface from the top view is circular.
In Figure \ref{Fig:Complex PN subdivision surfce modeling}(b) all vertices of the control mesh are sampled from a cuboid with square bottom while all assigned control normals are parallel to the bottom plane of the cuboid. Particularly, the control normals at the inner control points are pointing outwards and the control normals at points on outside edges are perpendicular to the edges while no control normals are assigned at the corner vertices. A square shaped surface with a circular hole is obtained by PN-Doo-Sabin subdivision.
Figure \ref{Fig:Complex PN subdivision surfce modeling}(c) illustrates an interpolatory PN subdivision surface. A $6\times6$ quad mesh is constructed by points and normals sampled from a Dupin cyclide. Due to the property of circle preserving, the outer silhouette circle, the inner silhouette circle and the six sampled circles across these two silhouette circles are preserved very well by PN-Kobbelt subdivision.

\begin{figure*}[htb]
  \centering
  \subfigure[]{\includegraphics[width=0.25\linewidth]{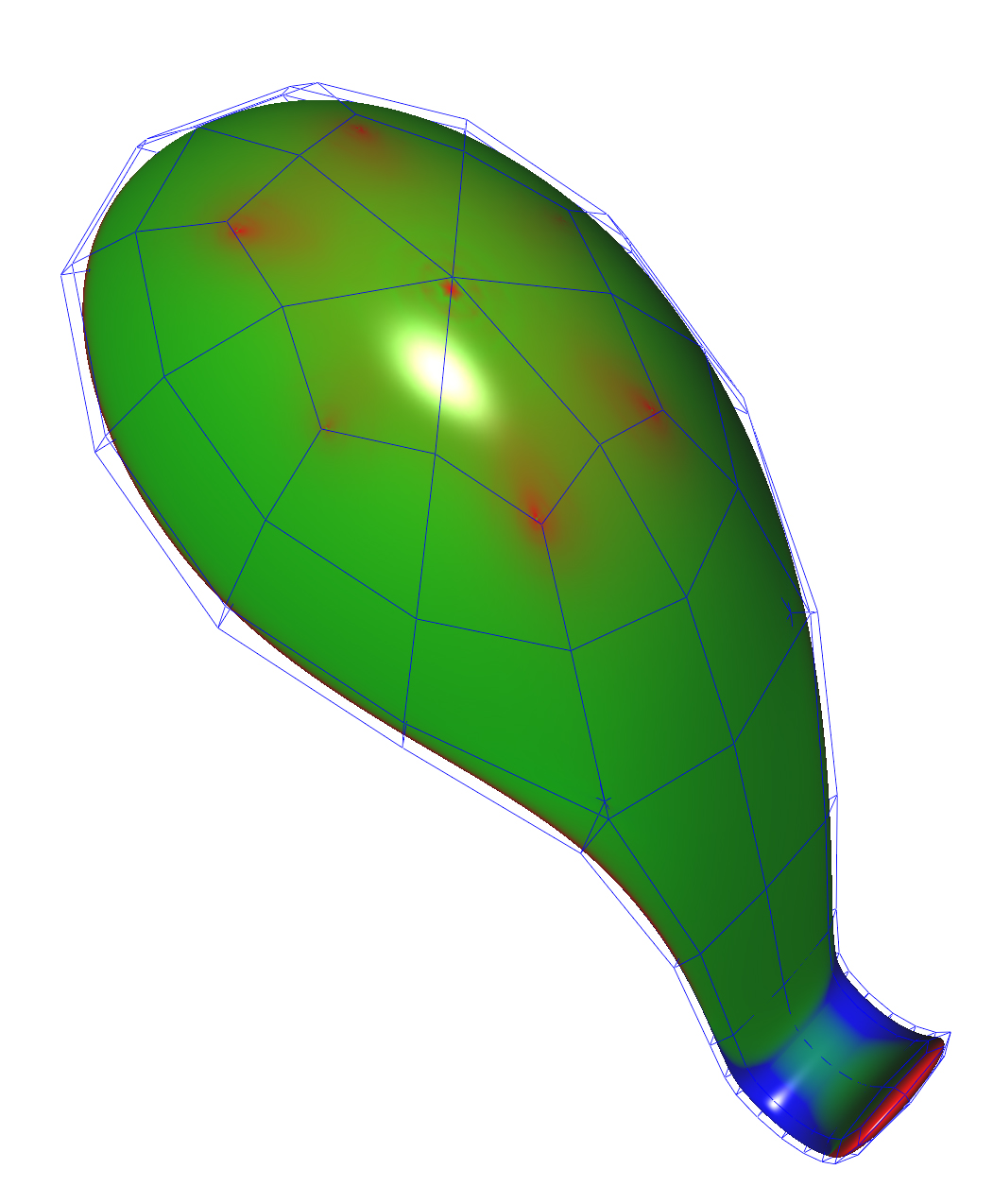}}
  \hskip -0.2cm
  \subfigure[]{\includegraphics[width=0.25\linewidth]{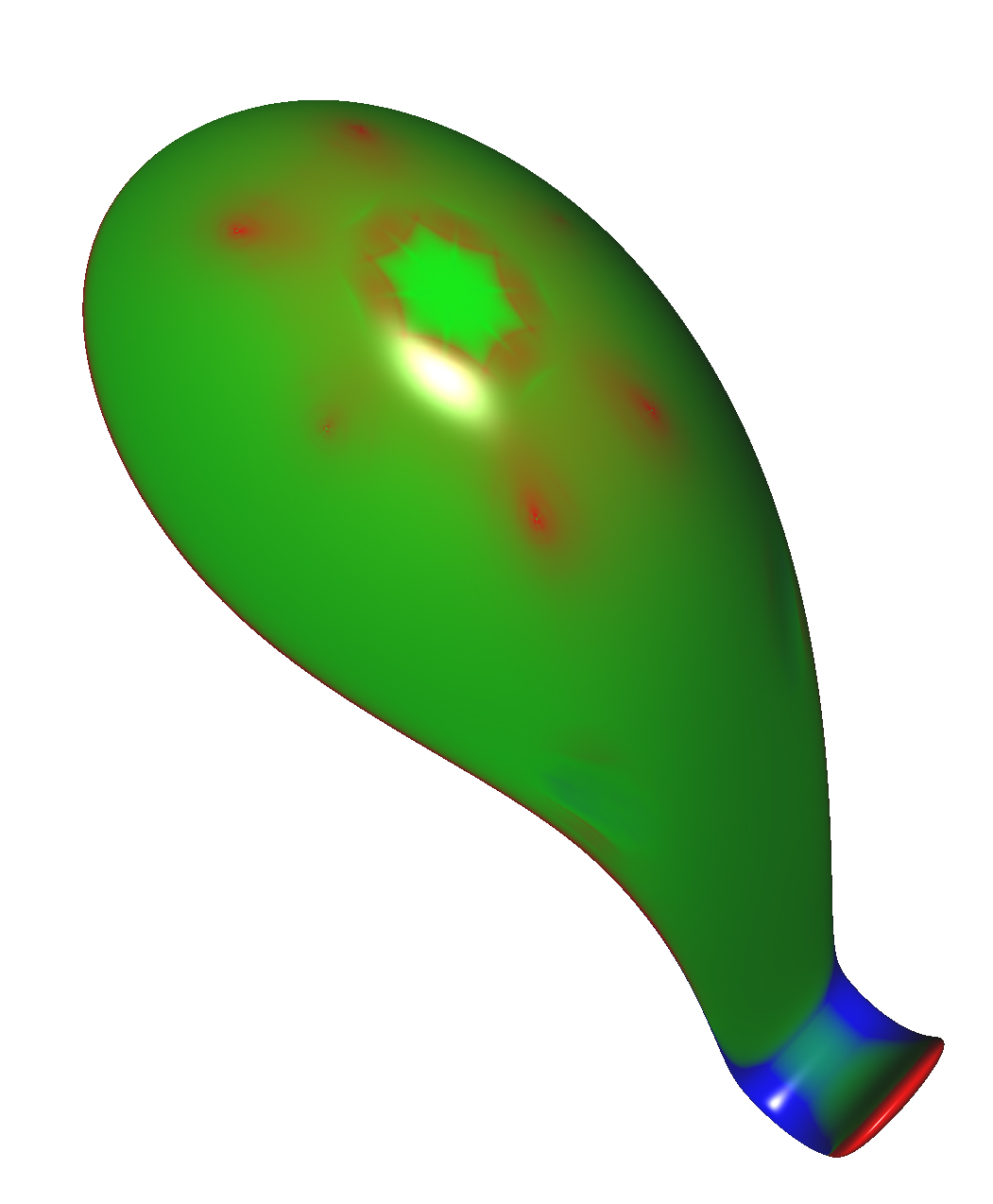}}
  \hskip -0.2cm
  \subfigure[]{\includegraphics[width=0.25\linewidth]{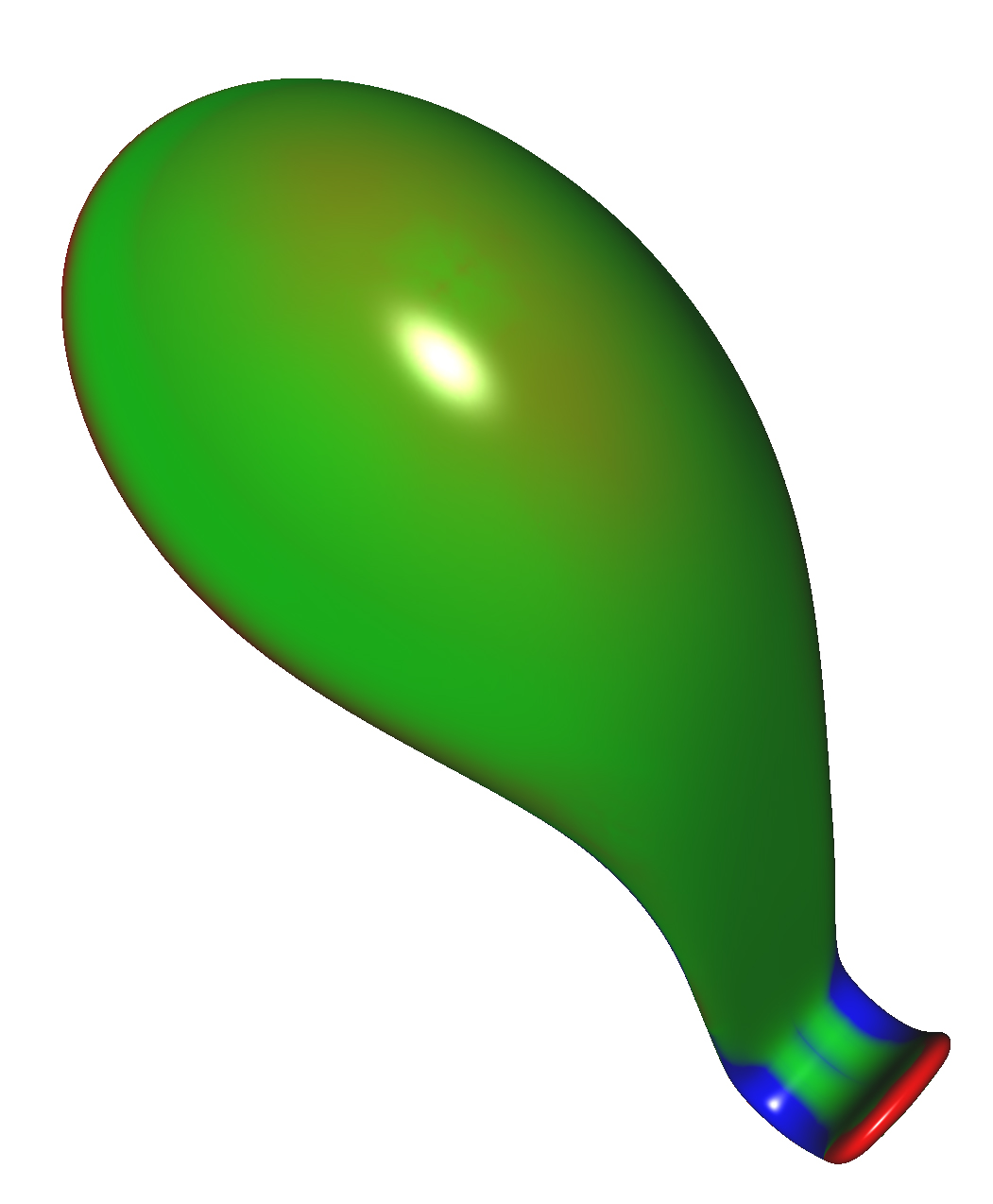}}
  \hskip -0.2cm
  \subfigure[]{\includegraphics[width=0.25\linewidth]{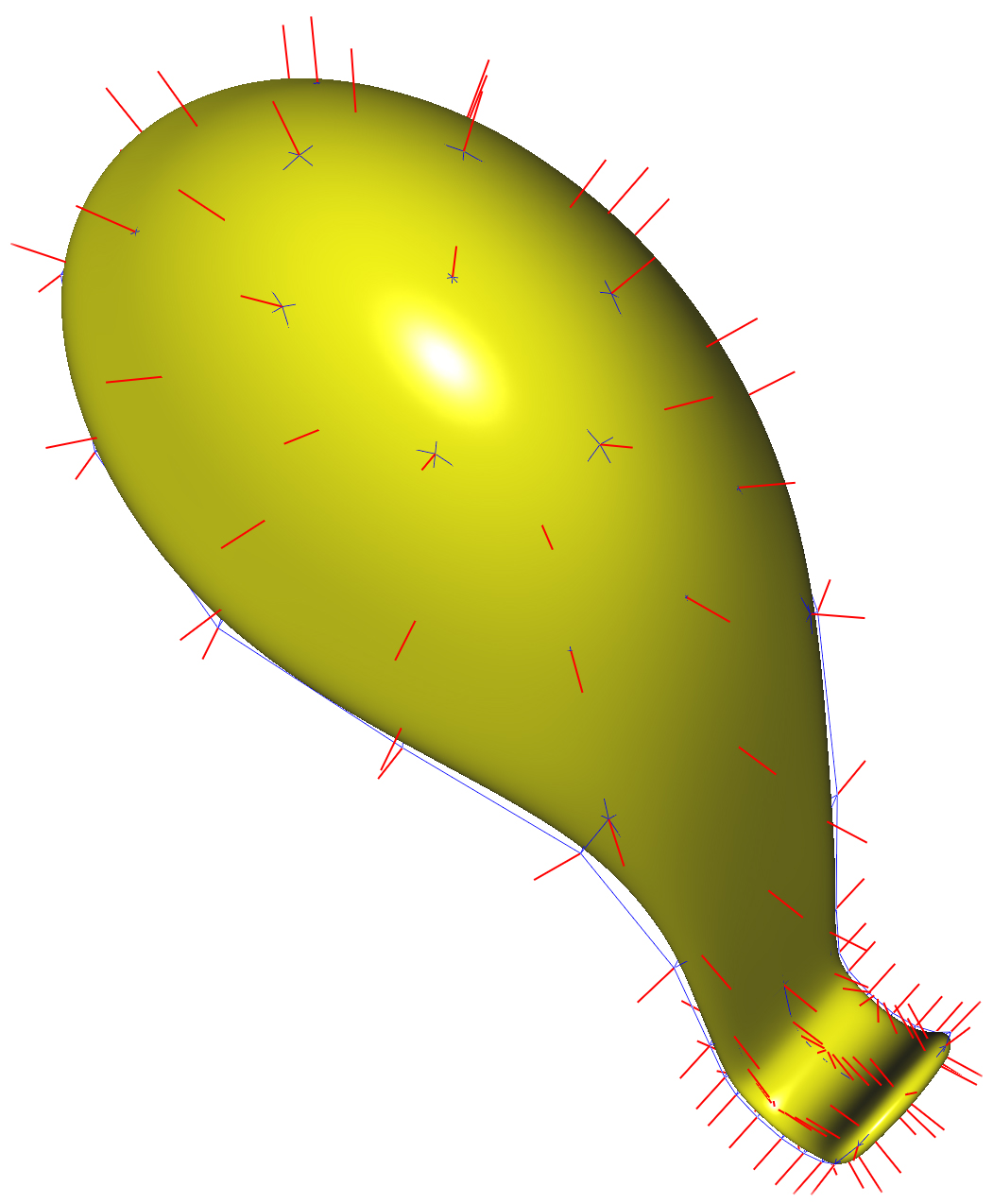}}
  \caption{Subdivision surfaces with Gaussian curvature plots or with control points and control normals by (a) Catmull-Clark subdivision; (b) modified Catmull-Clark subdivision \citep{Prautzsch1998G2subdivisionsurface}; (c)\&(d) PN-modified Catmull-Clark subdivision. }
  \label{Fig:Spoon modeling by PN modified CC subdivision}
\end{figure*}

\begin{figure*}[htb]
  \centering
  \subfigure[]{\includegraphics[width=0.25\linewidth]{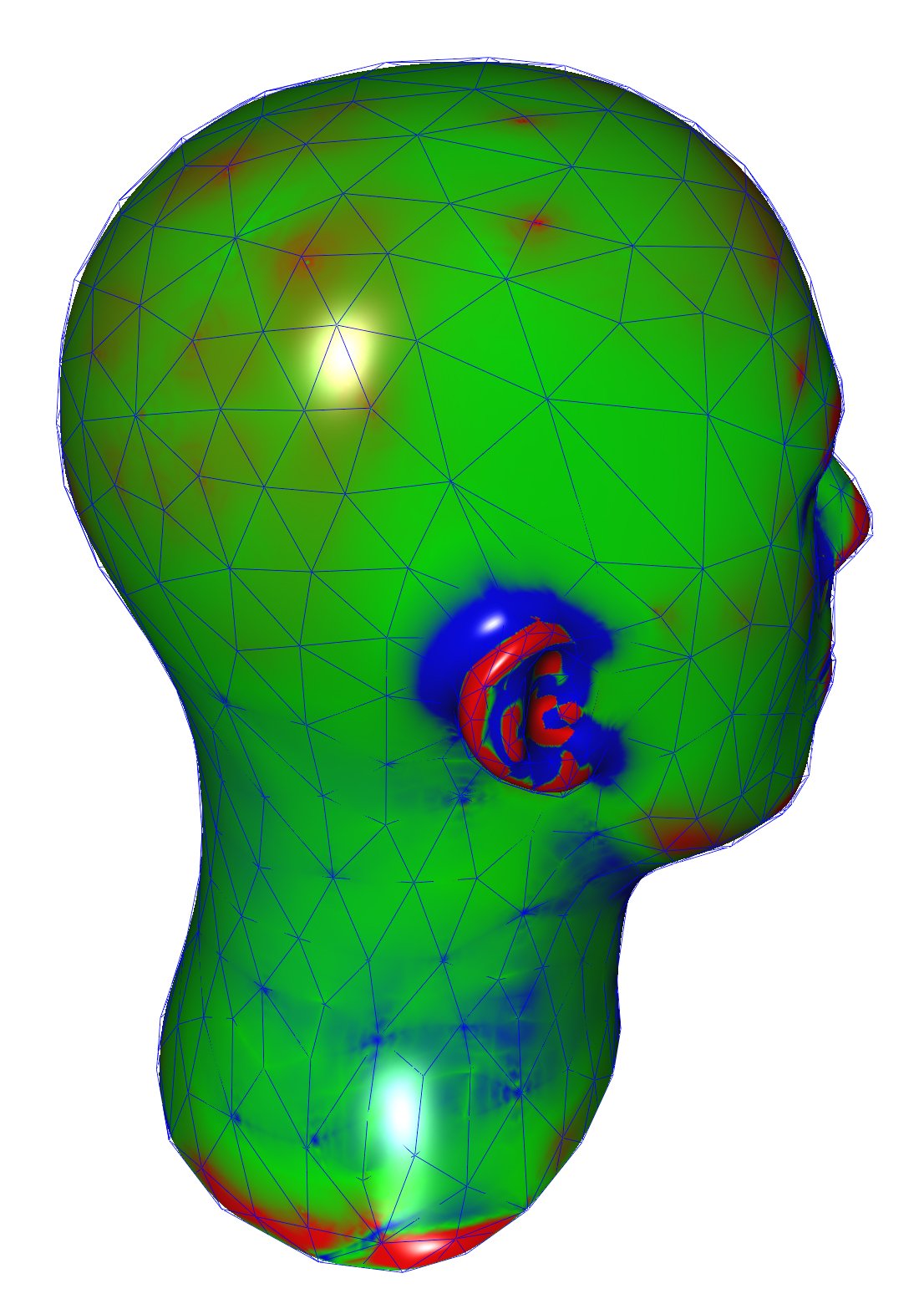}}
  \hskip -0.2cm
  \subfigure[]{\includegraphics[width=0.25\linewidth]{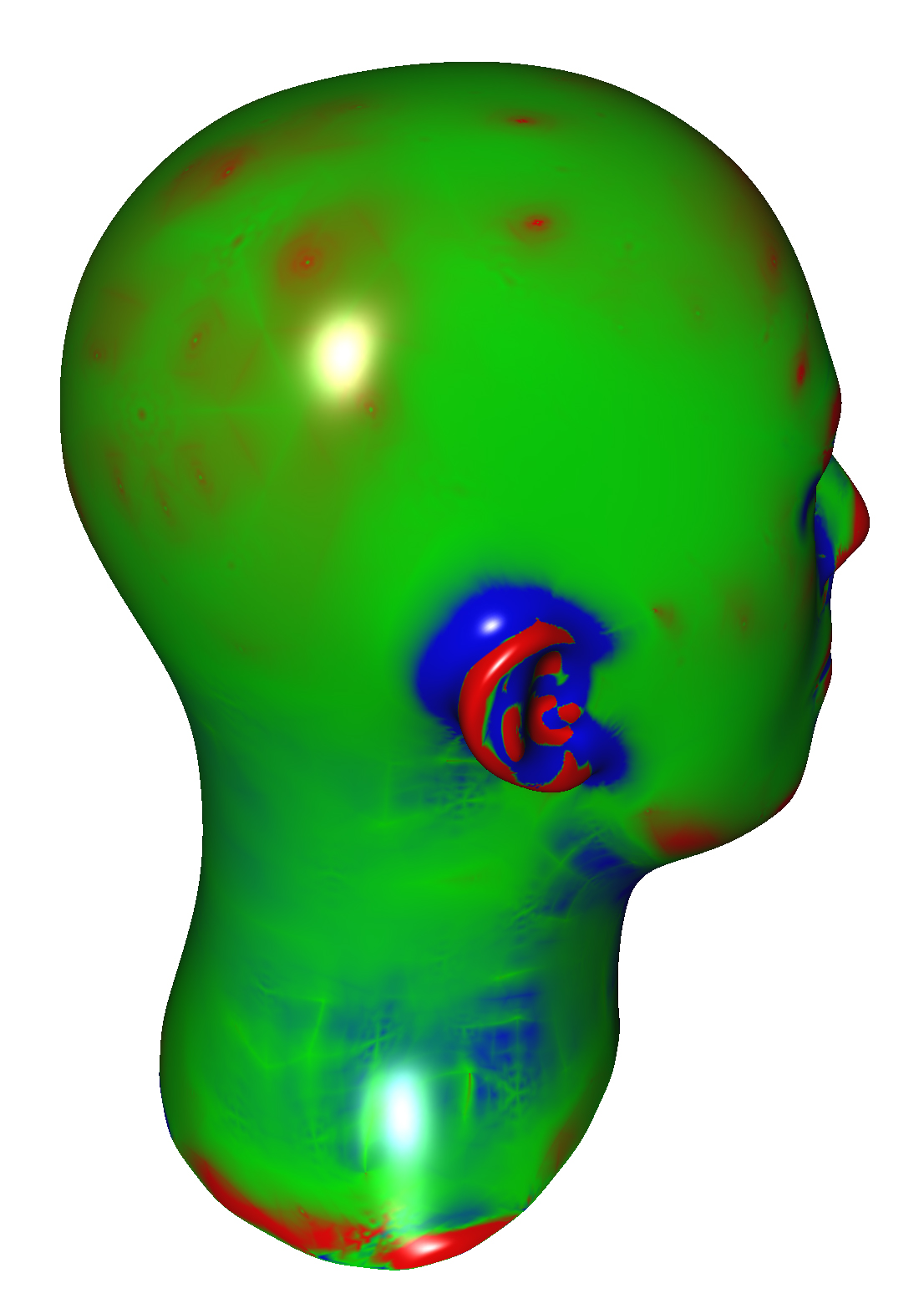}}
  \hskip -0.2cm
  \subfigure[]{\includegraphics[width=0.25\linewidth]{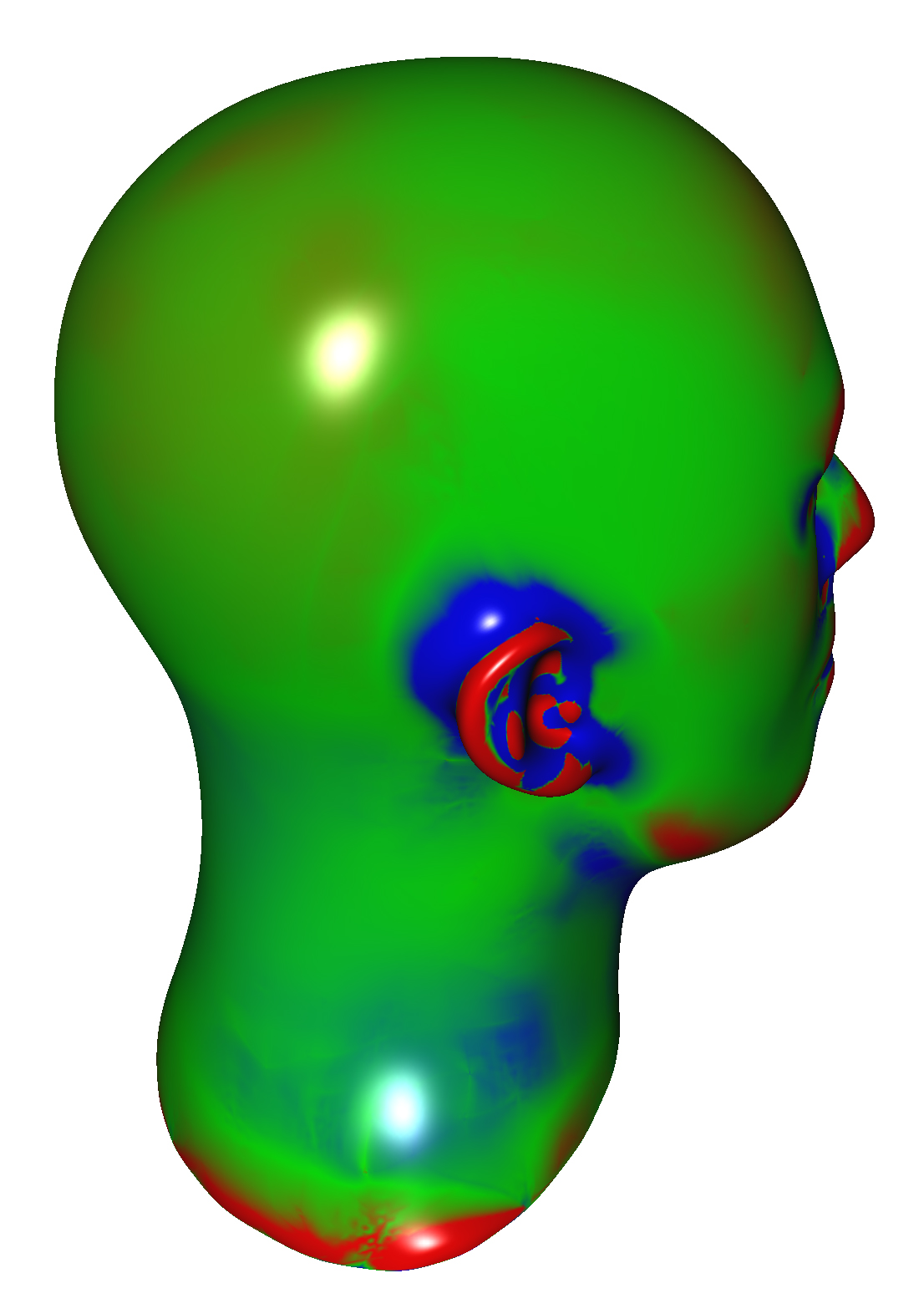}}
  \hskip -0.2cm
  \subfigure[]{\includegraphics[width=0.25\linewidth]{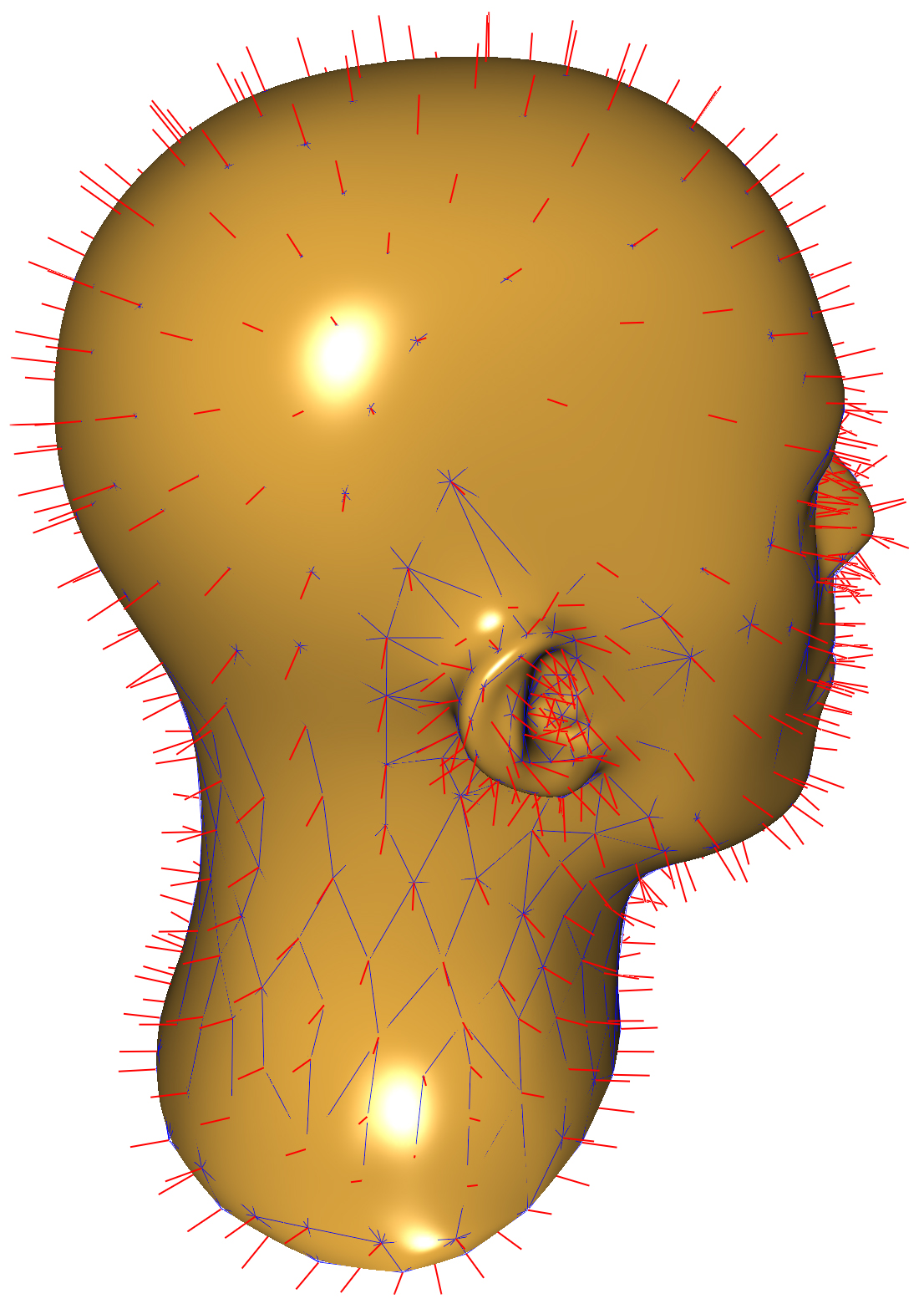}}
  \caption{Subdivision surfaces with Gaussian curvature plots or with control points and control normals by (a) Loop subdivision; (b) modified Loop subdivision \citep{Prautzsch2000G2Loopsurface}; (c)\&(d) PN-modified Loop subdivision. }
  \label{Fig:mannequin modeling by PN modified Loop subdivision}
\end{figure*}

Figure \ref{Fig:Spoon modeling by PN modified CC subdivision}(a) illustrates a quad mesh and the Catmull-Clark subdivision surface computed from the control mesh. The extraordinary points on the surface are evidently noticed based on the Gaussian curvature plot. Figure \ref{Fig:Spoon modeling by PN modified CC subdivision}(b) illustrates the $C^2$ subdivision surface with flat extraordinary points by the modified Catmull-Clark subdivision scheme proposed by \citep{Prautzsch1998G2subdivisionsurface}. Figures \ref{Fig:Spoon modeling by PN modified CC subdivision}(c) and \ref{Fig:Spoon modeling by PN modified CC subdivision}(d) are the subdivision surfaces with or without Gaussian curvature plot by our proposed PN-modified Catmull-Clark subdivision scheme. The control normals at all control vertices for this and the next example are computed as weighted sums of normal vectors of abutting faces with weights proportional to vertex angles of the faces. It is clearly seen that the curvature of the PN-modified Catmull-Clark subdivision surface is visually continuous and the extraordinary points are hardly to be distinguished due to the smoothness and fairness of the subdivision surface.

Figure \ref{Fig:mannequin modeling by PN modified Loop subdivision}(a) illustrates a triangular control mesh and the obtained Loop subdivision surface with Gaussian curvature plot while Figure \ref{Fig:mannequin modeling by PN modified Loop subdivision}(b) is the modified Loop subdivision surface by the technique proposed in \citep{Prautzsch2000G2Loopsurface}. We note that the subdivision rules for extraordinary vertices of valence 4 or 5 are not changed for the modified scheme due to the reason that the original stencils can already generate subdivision surfaces with bounded curvatures there. The curvature plot shows that the modified Loop subdivision surface still suffers concentric undulations around the extraordinary points of which the curvatures are forced zero. Figures \ref{Fig:mannequin modeling by PN modified Loop subdivision}(c) and \ref{Fig:mannequin modeling by PN modified Loop subdivision}(d) are the PN-modified Loop subdivision surfaces with Gaussian curvature plot or with control points and control normals. From the figure we see that the PN-modified Loop subdivision surface is smooth and fair with visually continuous curvature even at the extraordinary points.

%%%%%%%%%%%%%%%%%%%%%%%%%%%%%%%%%%%%%%%%%%%%%%
%%%% Section 7
%%%%%%%%%%%%%%%%%%%%%%%%%%%%%%%%%%%%%%%%%%%%%%
\section{Discussions}
\label{Sec:Discussions}

From the theories and experimental results of PN subdivision we learn that control normals together with control polygons or control meshes can achieve exact circular shapes, visually $C^2$ subdivision surfaces with non-flat extraordinary points and flexible detail editing on curves or surfaces. As control normals are subdivided independent of control points, the subdivided normals are \emph{generally not} the normals of subdivision curves or surfaces except that the control points and control normals lie on circles, circular cylinders or spheres. Even though, the effects of control normals on the shapes of PN subdivision curves and surfaces can be predicted well at least in the following two cases:
(1) the end points and end normals of each edge match a local convex curve or lie on a circular arc; (2) the control normals at the two ends of an edge are equal. In the first case the subdivided normals can approximate the normals of final subdivision curves or surfaces well. In the second case, the PN subdivision reduces to linear subdivision with no or less influence of control normals. To achieve even more modeling effects, these two kinds of control normals can be applied together for curve and surface modeling by PN subdivision.

\begin{figure*}[htb]
  \centering
  \subfigure[]{\includegraphics[width=0.28\linewidth]{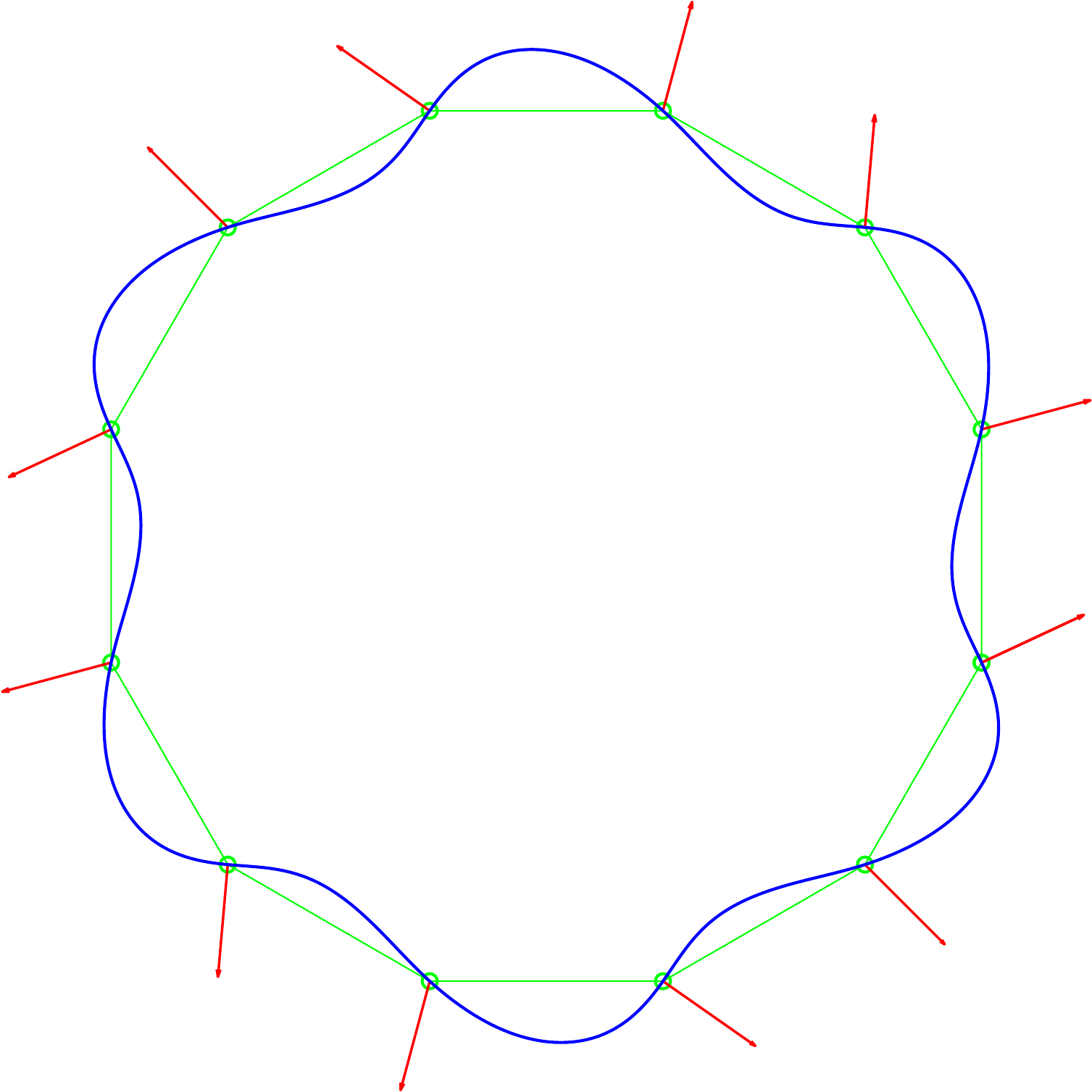}} \ \ \
  \subfigure[]{\includegraphics[width=0.28\linewidth]{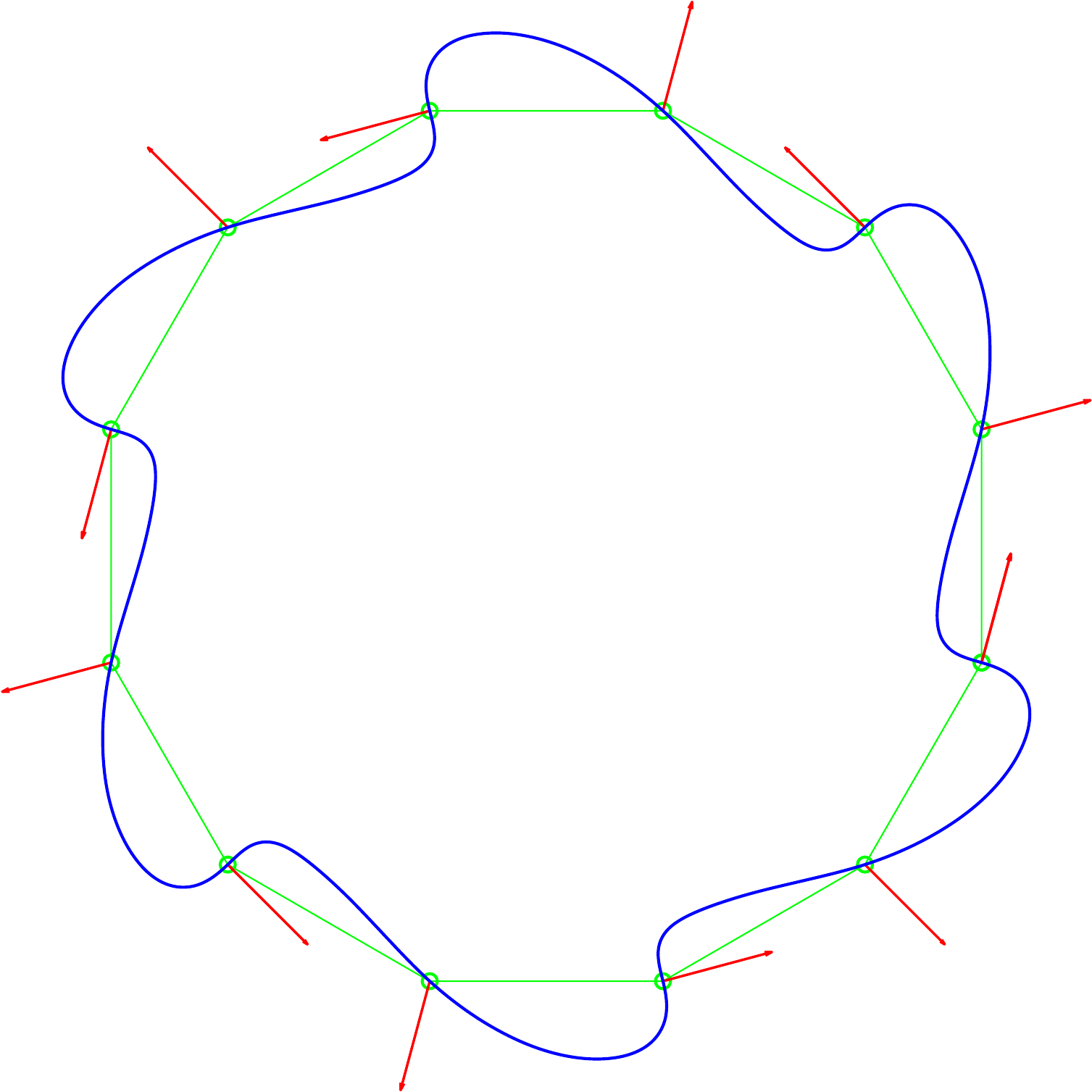}} \ \ \
  \subfigure[]{\includegraphics[width=0.28\linewidth]{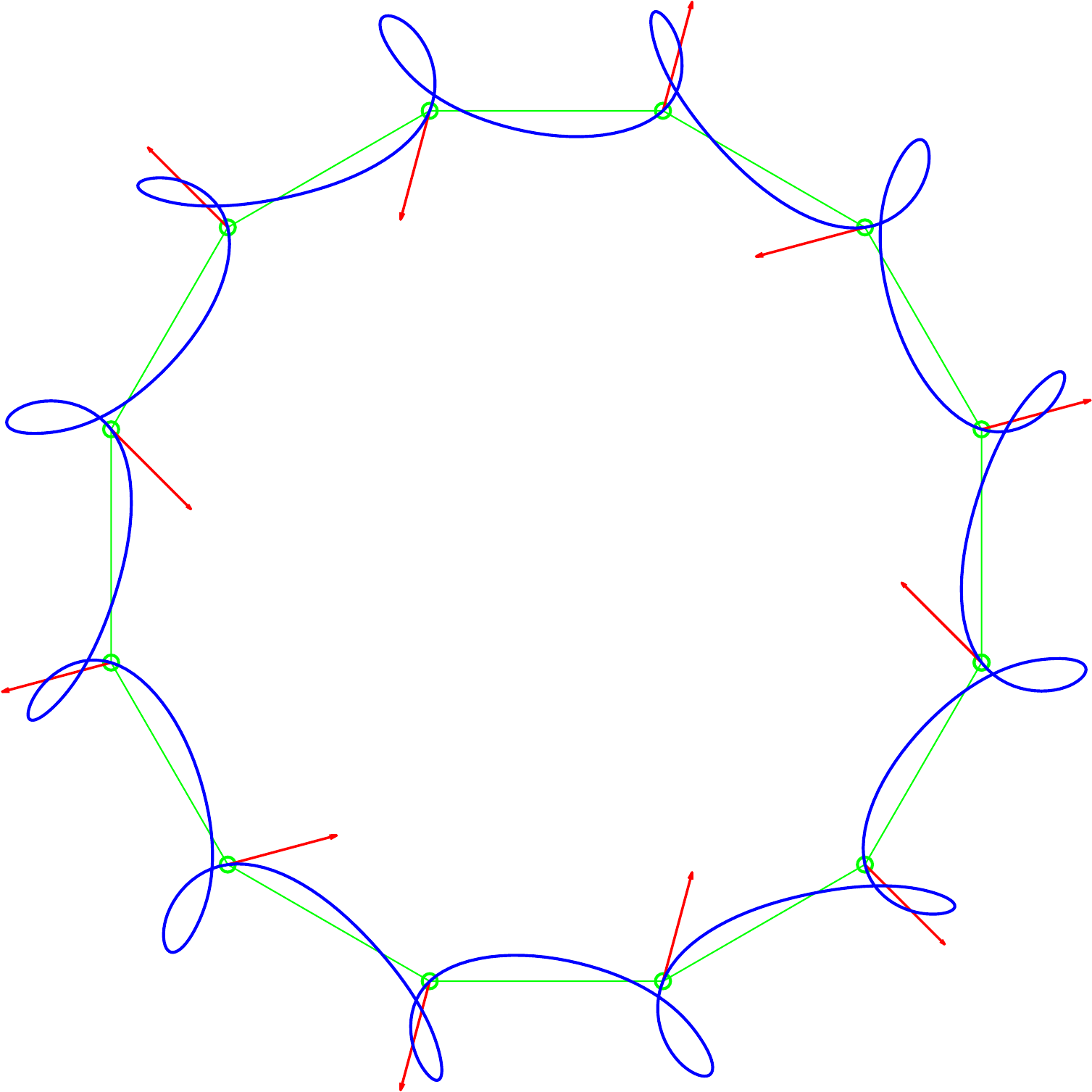}}
  \caption{Curve modeling by PN-10-point subdivision using control points and edited control normals.}
  \label{Fig:CircleNormalEdit}
\end{figure*}

Figure \ref{Fig:CircleNormalEdit} illustrates examples when the shape of a PN-2n-point subdivision curve can or cannot be predicted well from control normals. Similar results hold for other PN subdivision curves or surfaces. The control points and control normals within a closed polygon are first sampled from a circle and then every two initial normals are rotated by $40^\circ$, $90^\circ$ or $150^\circ$ but the remaining ones are kept unchanged. By adapting the recursive linear 2n-point subdivision scheme given in \citep{DengMa13:recursive2npointsubdivision} to PN subdivision, three PN-10-point subdivision curves are obtained from the control points and control normals. From the figure we see that the PN interpolatory subdivision curves can interpolate all control points but not necessarily the control normals. It is also noticed that the subdivision curves follow the shape of control polygon and the control normals as well when there exist local convex curves matching the end points and end normals for each edge; see Figures \ref{Fig:CircleNormalEdit}(a) and \ref{Fig:CircleNormalEdit}(b). Since every two neighboring normals in Figure \ref{Fig:CircleNormalEdit}(c) have almost opposite directions, the normals obtained by interpolatory subdivision also change rapidly and the subdivision curve even has unpredicted self-intersections. To avoid defects like self-intersections or creases, initial control normals should change smoothly or slowly along the control polygon or control mesh, or additional control points and control normals have to be added to help model curves or surfaces with more complex details.

Unlike their linear counterparts, curves and surfaces constructed by approximate PN subdivision schemes such as PN-B-spline subdivision, PN-Catmull-Clark subdivision, etc. may not lie in the convex hulls of their control points. The convex hulls of PN subdivision curves and surfaces have to be computed by taking consideration of control points and control normals together.
In contrast to stationary linear subdivision schemes by which the limit points or even the limit normals can be evaluated explicitly, the limit points of PN subdivision curves and surfaces may not be evaluated directly. They have to be evaluated iteratively at present.

\begin{figure}[htbp]
  \centering
  \subfigure[]{\includegraphics[width=0.28\linewidth]{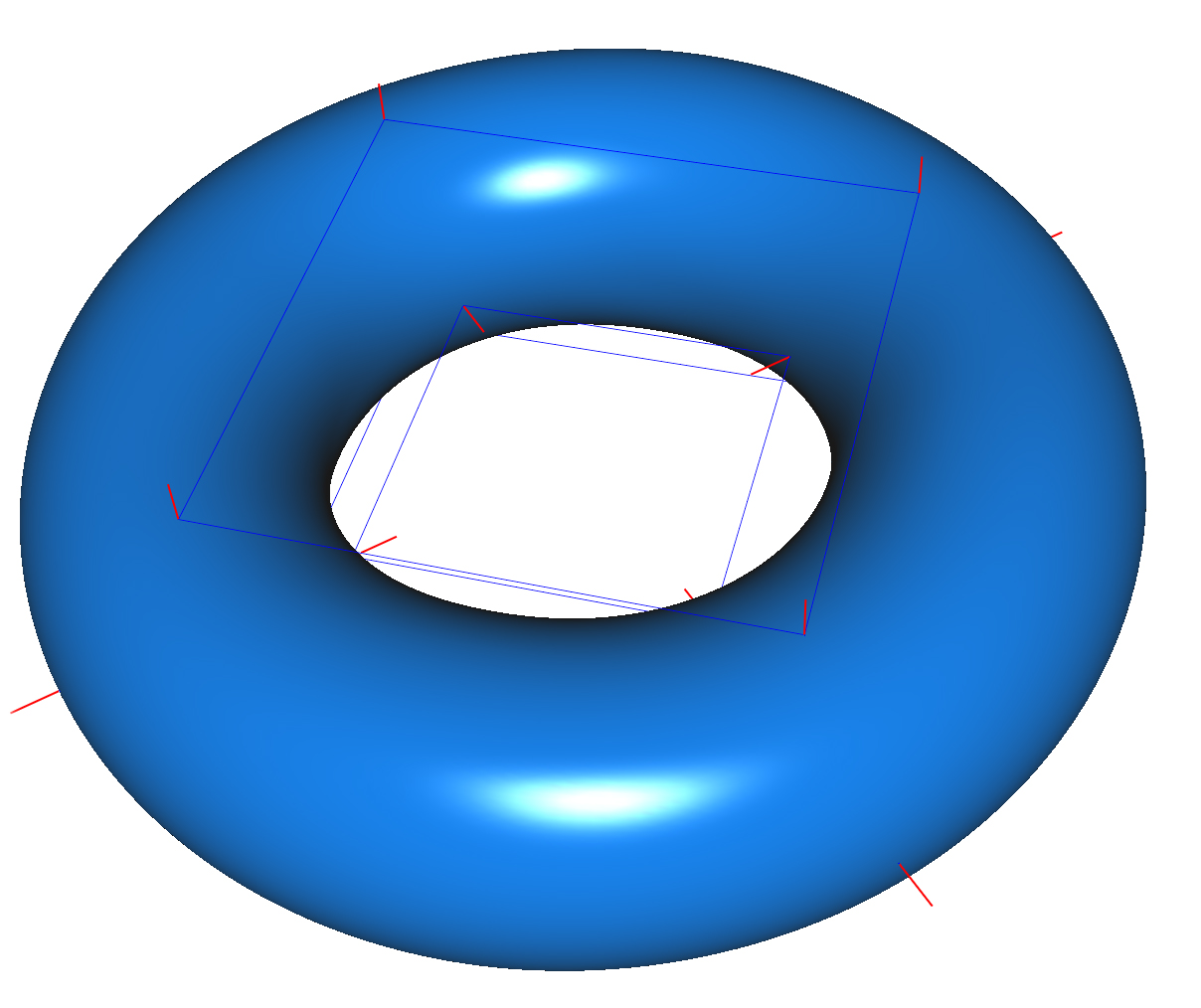}}
  \subfigure[]{\includegraphics[width=0.28\linewidth]{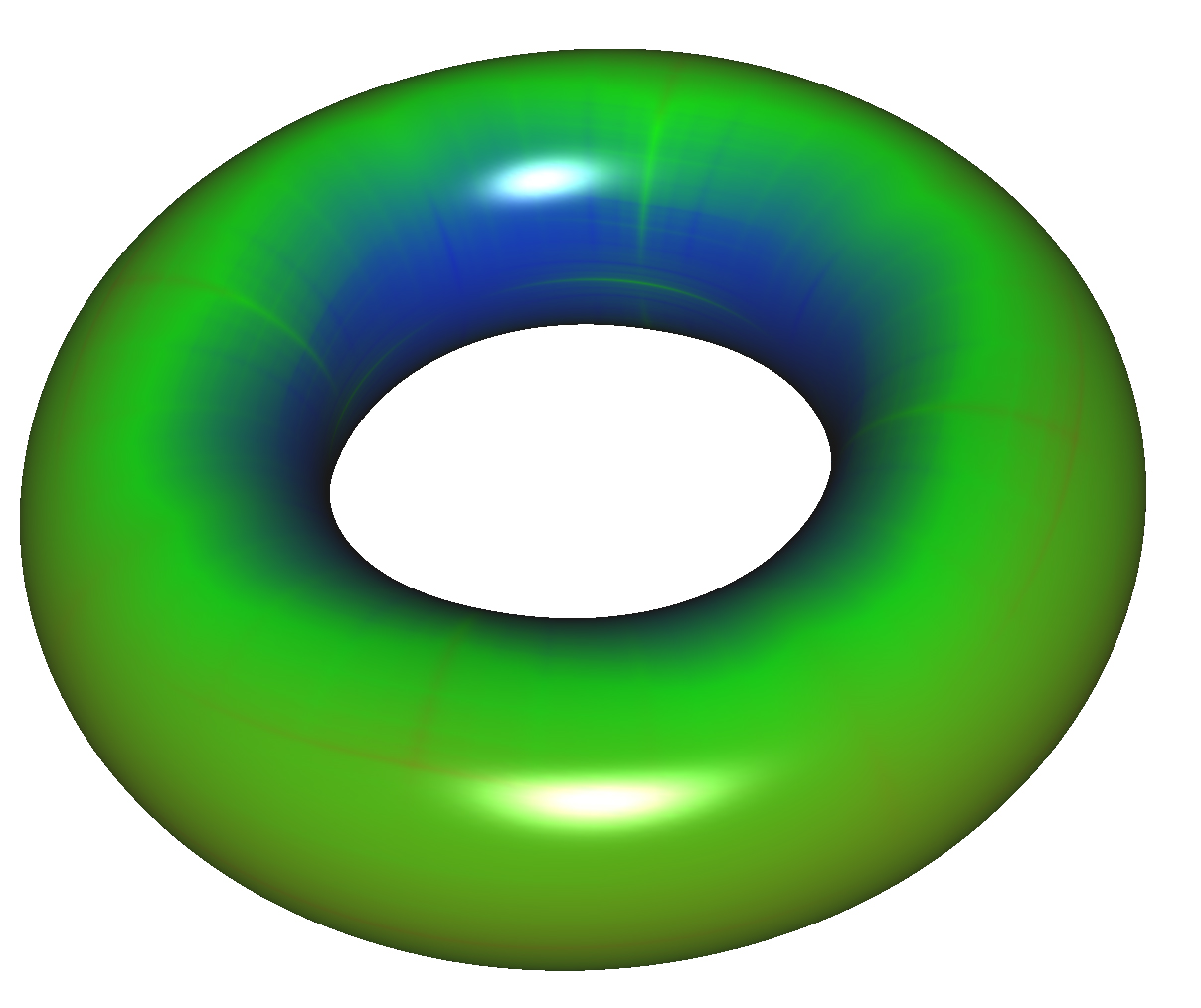}}
  \subfigure[]{\includegraphics[width=0.28\linewidth]{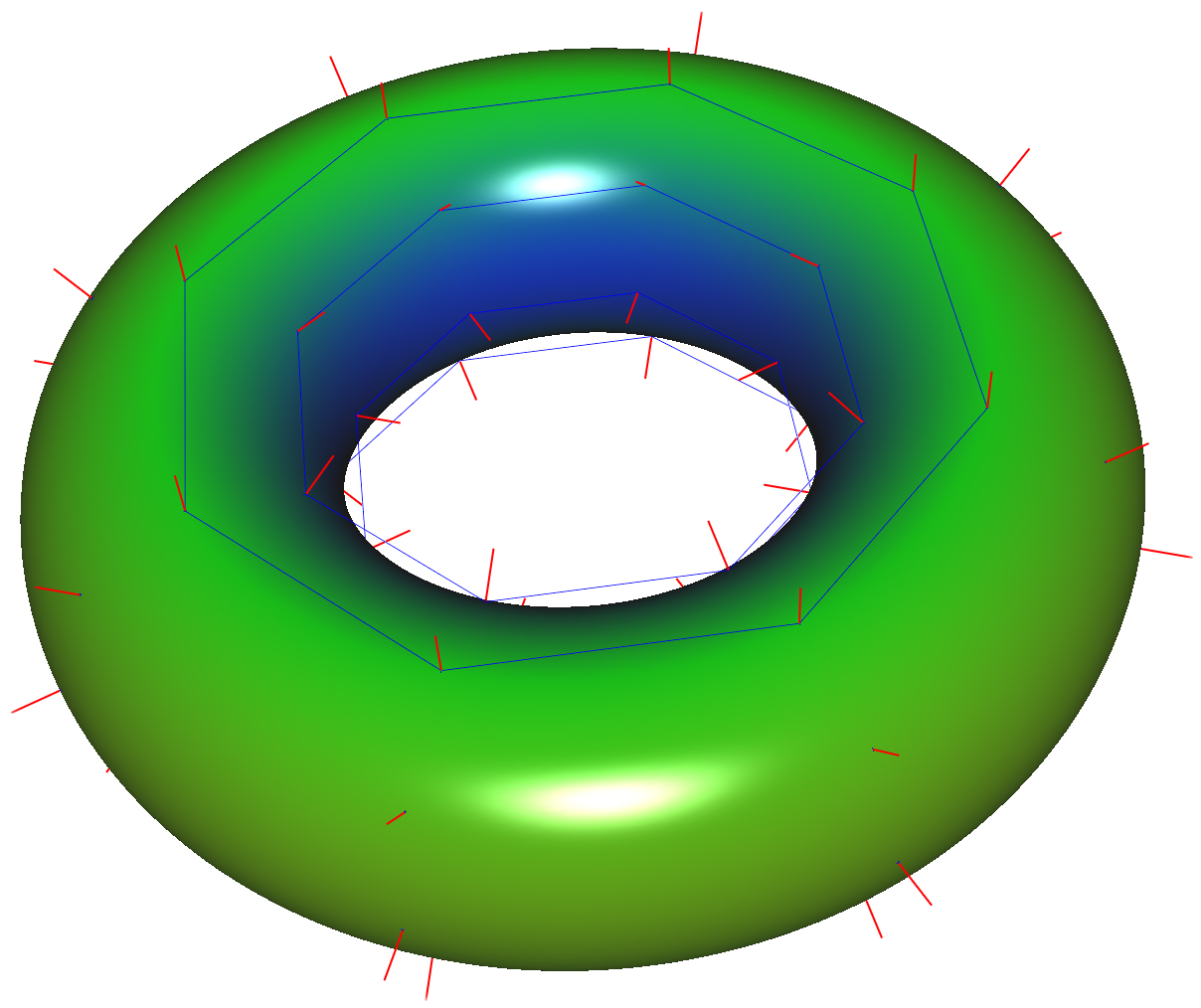}}
  \caption{Torus shape modeling by PN-Kobbelt subdivision: (a) the subdivision surface with $4\times4$ control points and control normals; (b) the Gaussian curvature plot of the surface in (a); (c) the subdivision surface with $8\times8$ control points and control normals. }
  \label{Fig:torusCircle}
\end{figure}

Though PN subdivision curves and surfaces can preserve typical shapes like circles, circular cylinders or spheres exactly, PN subdivision surfaces that generalize simple linear schemes do not preserve toruses or cyclides which are composed of families of circles. Figure \ref{Fig:torusCircle}(a) illustrates a PN-Kobbelt subdivision surface with a total of 16 control points and control normals sampled from a torus. Similar to the Dupin cyclide in Figure \ref{Fig:Complex PN subdivision surfce modeling}(c), several geodesic circles on the torus are preserved because the sampled points and normals on the surface are also the points and normals on the circles. As a result, the PN subdivision surface resembles a torus shape very well. Even so, the Gaussian curvature plot in Figure \ref{Fig:torusCircle}(b) illustrates that the subdivision surface is not exactly a torus. If the subdivision surface is constructed with more control points and control normals sampled from the torus, it resembles the original surface more accurately; see Figure \ref{Fig:torusCircle}(c).
The approximate PN subdivision surfaces may not pass through the control points, they do not preserve toruses or cyclides either.

%%%%%%%%%%%%%%%%%%%%%%%%%%%%%%%%%%%%%%%%%%%%%%
%%%% Section 8
%%%%%%%%%%%%%%%%%%%%%%%%%%%%%%%%%%%%%%%%%%%%%%

\section{Conclusions and future work}
\label{Sec:conclusions}

In this paper we have presented novel nonlinear subdivision schemes for constructing curves and surfaces with control points and control normals. Our proposed PN subdivision schemes generalize traditional linear subdivision schemes in a simple and efficient way and the nonlinear subdivision schemes can be implemented almost in the same way as the traditional linear ones.
PN subdivision schemes can have same convergence and smoothness orders as linear subdivision schemes, and they can reproduce circles, circular cylinders and spheres. The nice properties of the proposed subdivision schemes make them powerful tools for geometric modeling.
Besides modeling curves and surfaces with local details, PN subdivision schemes are also capable of modeling fair curves and surfaces using simply chosen control normals. Particularly, PN subdivision schemes can be simple solutions to modeling fair $C^2$ subdivision surfaces with arbitrary topology control meshes by adapting linear $C^2$ subdivision schemes that only generate subdivision surfaces with flat extraordinary points.

As future work, a few interesting topics deserve further study: (a) curvature continuity analysis of PN $C^2$ subdivision surfaces with arbitrary topology control meshes; (b) computation of convex hulls or limit points of PN subdivision curves and surfaces; (c) construction of PN subdivision curves and surfaces that have prescribed normals or curvatures at selected points or curves; (d) exploring surface subdivision schemes that preserve other geometric primitives such as toruses or cyclides.

\section*{Acknowledgment}

This work was supported by the National Natural Science Foundation of China under Grant No. 12171429.

\bibliographystyle{elsarticle-harv}
\bibliography{PNsubdivision-bib}




%% \bibitem must have one of the following forms:
%%   \bibitem[Jones et al.(1990)]{key}...
%%   \bibitem[Jones et al.(1990)Jones, Baker, and Williams]{key}...
%%   \bibitem[Jones et al., 1990]{key}...
%%   \bibitem[\protect\citepauthoryear{Jones, Baker, and Williams}{Jones
%%       et al.}{1990}]{key}...
%%   \bibitem[\protect\citepauthoryear{Jones et al.}{1990}]{key}...
%%   \bibitem[\protect\astroncitep{Jones et al.}{1990}]{key}...
%%   \bibitem[\protect\citepname{Jones et al., }1990]{key}...
%%   \harvarditem[Jones et al.]{Jones, Baker, and Williams}{1990}{key}...
%%


% \end{thebibliography}

\end{document}